\documentclass[10pt,reqno,a4paper]{amsart}
\usepackage{amsthm}
\usepackage{mathtools}
\usepackage{amssymb}
\usepackage{mathscinet}
\usepackage{mathrsfs}
\usepackage{hyperref}
\usepackage{enumitem}
\usepackage{graphicx}
\usepackage{xcolor}
\usepackage{kotex}
\usepackage{comment}
\usepackage{bm}

\numberwithin{equation}{section}
\numberwithin{figure}{section}
\allowdisplaybreaks

\theoremstyle{plain}
\newtheorem{thm}{\protect\theoremname}[section]
\newtheorem{lem}[thm]{\protect\lemmaname}
\newtheorem{prop}[thm]{\protect\propositionname}

\theoremstyle{remark}
\newtheorem{rem}[thm]{\protect\remarkname}

\providecommand{\theoremname}{Theorem}
\providecommand{\corollaryname}{Corollary}
\providecommand{\lemmaname}{Lemma}
\providecommand{\remarkname}{Remark}
\providecommand{\propositionname}{Proposition}

\begin{document}
\global\long\def\Im{\mathrm{Im}}%
\global\long\def\Re{\mathrm{Re}}%
\global\long\def\Right{\Rightarrow}%
\global\long\def\Left{\Leftarrow}%
\global\long\def\rad{\mathrm{rad}}%
\global\long\def\jp#1{\langle#1\rangle}%
\global\long\def\norm#1{\|#1\|}%
\global\long\def\ol#1{\overline{#1}}%
\global\long\def\wt#1{\widehat{#1}}%
\global\long\def\br#1{(#1)}%
\global\long\def\Bb#1{\Big(#1\Big)}%
\global\long\def\bb#1{\big(#1\big)}%
\global\long\def\lr#1{\left(#1\right)}%
\global\long\def\la{\lambda}%
\global\long\def\al{\alpha}%
\global\long\def\be{\beta}%
\global\long\def\ga{\gamma}%
\global\long\def\La{\Lambda}%
\global\long\def\De{\Delta}%
\global\long\def\na{\nabla}%
\global\long\def\fl{\flat}%
\global\long\def\sh{\sharp}%
\global\long\def\calN{\mathcal{N}}%
\global\long\def\avg{\mathrm{avg}}%
\global\long\def\bbR{\mathbf{\mathbb{R}}}%
\global\long\def\bbC{\mathbf{\mathbb{C}}}%
\global\long\def\bbZ{\mathbf{\mathbb{Z}}}%
\global\long\def\bbN{\mathbf{\mathbb{N}}}%
\global\long\def\bbT{\mathbb{T}}%
\global\long\def\bbM{\mathbb{M}}%
\global\long\def\bbP{\mathbb{P}}%
\global\long\def\bbS{\mathbb{S}}%
\global\long\def\bfA{\mathbf{A}}%
\global\long\def\bfD{\mathbf{D}}%
\global\long\def\bfL{\mathbf{L}}%
\global\long\def\bfu{\mathbf{u}}%
\global\long\def\bfv{\mathbf{v}}%
\global\long\def\calF{\mathcal{F}}%
\global\long\def\calH{\mathcal{H}}%
\global\long\def\calI{\mathcal{I}}%
\global\long\def\calL{\mathcal{L}}%
\global\long\def\calE{\mathcal{E}}%
\global\long\def\calZ{\mathcal{Z}}%
\global\long\def\calO{\mathcal{O}}%
\global\long\def\calR{\mathcal{R}}%
\global\long\def\calD{\mathcal{D}}%
\global\long\def\calT{\mathcal{T}}%
\global\long\def\calY{\mathcal{Y}}%
\global\long\def\calA{\mathcal{A}}%
\global\long\def\calB{\mathcal{B}}%
\global\long\def\calC{\mathcal{C}}%
\global\long\def\calG{\mathcal{G}}%
\global\long\def\calM{\mathcal{M}}%
\global\long\def\calP{\mathcal{P}}%
\global\long\def\calS{\mathcal{S}}%
\global\long\def\calU{\mathcal{U}}%
\global\long\def\Lmb{\Lambda}%
\global\long\def\eps{\varepsilon}%
\global\long\def\lmb{\lambda}%
\global\long\def\gmm{\gamma}%
\global\long\def\rd{\partial}%
\global\long\def\chf{\mathbf{1}}%
\global\long\def\td#1{\widetilde{#1}}%
\global\long\def\sgn{\mathrm{sgn}}%
\global\long\def\blambda{\boldsymbol\lambda}%
\global\long\def\biota{\boldsymbol\iota}%
\global\long\def\Red#1{\textcolor{red}{#1}}%
\global\long\def\NL{\mathrm{NL}}%
\global\long\def\calK{\mathcal{K}}%
\global\long\def\calV{\mathcal{V}}%
\global\long\def\calW{\mathcal{W}}%
\global\long\def\calX{\mathcal{X}}%

\global\long\def\frakX{\mathfrak{X}}%
\title[Smooth Type II blow-up for 3D NLS]{Smooth radial quantized Type II blow-up for the three-dimensional energy-critical nonlinear Schr\"odinger equation}

\subjclass[2020]{35B44 (primary), 35Q55, 35B33}

\author{Taegyu Kim}
\email{k1216300@kias.re.kr}
\address{School of Mathematics, Korea Institute for Advanced Study, 85 Hoegiro Dongdaemun-gu, Seoul 02455, Korea}

\begin{abstract}
	We consider the focusing energy-critical nonlinear Schr\"odinger equation in three dimensions. In \cite{Schmid2026CVPDE}, Schmid constructed finite-time type II blow-up solutions with a continuum of blow-up rates, while the existence of smooth finite-time type II blow-up solutions remained open. We resolve this problem by constructing, for each integer $L\ge1$, smooth radial solutions concentrating a single ground state with the blow-up rate $\lambda(t)=(\ell+o(1))(T-t)^{2L}$ for some $\ell>0$.
	
	Our proof follows a forward construction based on modulation analysis with the energy method. The main obstacle is the lack of a repulsive structure that can be directly exploited to control the remainder. To overcome this obstacle, we adapt Kato's smoothing theory from asymptotic stability analysis to the forward construction. This yields a global spacetime Morawetz estimate for the linearized flow, providing the spacetime control needed to bound the higher-order energy in place of the monotonicity derived from repulsive structures in previous blow-up constructions.
\end{abstract}

\maketitle
\setcounter{tocdepth}{1}
\tableofcontents

\section{Introduction}

We consider the initial value problem for the focusing energy-critical nonlinear Schr\"odinger equation in dimensions $d\ge3$:
\begin{align}\label{eq:NLS}
	\begin{cases}
		i\partial_tu+\Delta u+|u|^{\frac4{d-2}}u=0,\quad (t,x)\in\mathbb R\times\mathbb R^d,\\
		u(0)=u_0,
	\end{cases}
\end{align}
where $u:[-T,T]\times \mathbb{R}^d\to \mathbb{C}$.
The Cauchy problem for \eqref{eq:NLS} is locally well-posed in $\dot H^1(\mathbb R^d)$, and for sufficiently small initial data, the corresponding solutions are global and scatter in both time directions; see \cite{CazenaveWeissler1990}.

We briefly introduce symmetries and conservation laws. The equation conserves the energy and, for $H^1$ solutions, the mass and momentum:
\begin{equation}\label{eq:conservation laws}
	\begin{aligned}
		E(u)&\coloneqq\frac12\int_{\mathbb R^d}|\nabla u|^2dx-\frac{d-2}{2d}\int_{\mathbb R^d}|u|^{\frac{2d}{d-2}}dx,\\
		M(u)&\coloneqq\int_{\mathbb R^d}|u|^2dx,\qquad P(u)\coloneqq\Im\int_{\mathbb R^d}\ol u\nabla u dx.
	\end{aligned}
\end{equation}
Moreover, \eqref{eq:NLS} is invariant under scaling, phase rotation, and time and space translations:
\begin{equation*}
	u(t,x)\mapsto \frac1{\lambda^{\frac{d-2}{2}}}e^{i\gamma}u\left(\frac{t-t_0}{\lambda^2},\frac{x-x_0}{\lambda}\right),
\end{equation*}
where $\lambda>0$, $\gamma,t_0\in\mathbb R$, and $x_0\in\mathbb R^d$. The equation is also invariant under spatial rotations $u(t,x)\mapsto u(t,Rx)$, $R\in SO(d)$, and time reversal $u(t,x)\mapsto\ol{u(-t,x)}$. Since the scaling preserves both the $\dot H^1$ norm and the energy, \eqref{eq:NLS} is called energy-critical.

The \emph{ground state} or \emph{soliton} $W$ is also called the Aubin--Talenti bubble. It is the positive radial solution, normalized by $W(0)=1$, of
\begin{equation}
	-\Delta W=W^{\frac{d+2}{d-2}},\qquad W(x)\coloneqq \left(1+\frac{|x|^2}{d(d-2)}\right)^{-\frac{d-2}{2}}.
\end{equation}
The ground state $W$ attains equality in the sharp Sobolev inequality and marks the scattering threshold. Kenig and Merle \cite{KenigMerle2006Invent} proved global existence and scattering in both time directions for radial data in $3\le d\le5$ satisfying $E(u_0)<E(W)$ and $\|u_0\|_{\dot H^1}<\|W\|_{\dot H^1}$. 
Nonradial scattering was subsequently established for $d\ge5$ under $\sup_{t\in I}\|u(t)\|_{\dot H^1}<\|W\|_{\dot H^1}$ in \cite{KillipVisan2010AJM}, for $d=4$ in \cite{Dodson2019ASENS}, and recently, independently for $d=3$, in \cite{SuZhao2026arXivScattering} and \cite{ChungHan2026arXivScattering}.
At the threshold energy $E(u_0)=E(W)$, the radial dynamics were classified for $3\le d\le5$ in \cite{DuyckaertsMerle2009Threshold} and extended to $d\ge6$ in \cite{LiZhang2009JFA}. For nonradial threshold classifications, see \cite{SuZhao2023,MMMZ2025arXiv}.

Above the threshold, more varied dynamics are possible. For radial solutions in $d=3$ with energy slightly above $E(W)$, it was shown in \cite{NakanishiRoy2016CPAA} that solutions either remain in a tube around the ground-state family or leave it without returning and subsequently scatter or blow up in finite time. For finite-variance $H^{1,1}$ data, scattering and finite-time blow-up criteria at arbitrarily large energies were established in \cite{DuyckaertsRoudenko2015CMP}.

These results describe possible behaviors of solutions but do not characterize the nature of blow-up. For a solution with maximal forward existence time $T\in(0,\infty]$, we distinguish type I and type II solutions according to
\begin{equation*}
	\text{type I:}\quad\sup_{0\le t<T}\|u(t)\|_{\dot H^1}=\infty,\qquad
	\text{type II:}\quad\sup_{0\le t<T}\|u(t)\|_{\dot H^1}<\infty.
\end{equation*}

We briefly recall some representative constructions, beginning with $d=3$. Rapha\"el and Szeftel \cite{RaphaelSzeftel2009CMP} constructed stable type I standing-ring blow-up at the log-log rate for an open set of radial $H^3$ initial data; see also \cite{Zwiers2011APDE,HolmerRoudenko2012APDE}. Ortoleva and Perelman \cite{OrtolevaPerelman2013} constructed infinite-time type II blow-up solutions with single-bubble concentration, as well as nondispersive vanishing solutions. In higher dimensions, pure two-bubble solutions with one bubble concentrating as $t\to\infty$ were constructed for $d\ge7$ \cite{Jendrej2017AnalPDE} and for $d=6$ \cite{JendrejLiXu2026arXivNLS}. However, the existence of concentric bubble trees with two or more concentrating bubbles remains open for all $d\ge3$. See also \cite{JendrejKrieger2025arXiv,Kim2026arXivNobubbletree} for related results.

For finite-time type II blow-up in $d=3$, Schmid \cite{Schmid2026CVPDE} gave the first construction, obtaining single-bubble solutions with a continuum of blow-up rates. However, smoothness was not established in this construction, leaving the existence of smooth type II blow-up solutions open. In this paper, we construct smooth radial single-bubble type II blow-up solutions in $d=3$. In contrast to Schmid's continuum of blow-up rates, our solutions exhibit quantized rates.

We write $H^\infty_{\rad}\coloneqq \bigcap_{k\ge0}H^k_{\rad}$ and $f_{\lambda,\gamma}(x)\coloneqq\lambda^{-\frac12}e^{i\gamma}f(x/\lambda)$. Our main result is as follows.

\begin{thm}[Smooth quantized blow-up]\label{thm:main}
	For any integer $L\ge1$ and $\delta^*>0$, there exists a family of initial data in $H^\infty_{\rad}(\mathbb R^3)$ such that, for each datum $u_0$ in this family, the corresponding solution $u(t,x)$ to \eqref{eq:NLS} satisfies the following properties:
	\begin{enumerate}
		\item (Finite-time blow-up) $u$ blows up in finite time $T=T(u_0)\in(0,\infty)$.

		\item (Sharp description of the blow-up) There exist $C^1$ functions $\lambda(t)>0$ and $\gamma(t)\in\mathbb R$ such that
		\begin{equation}\label{eq:thm tube stability}
			\sup_{0\le t<T}\|u(t)-W_{\lambda(t),\gamma(t)}\|_{\dot H^1}<\delta^*.
		\end{equation}
		For some $\ell>0$, $\gamma_*\in\mathbb R$, and $u^*\in H^{2L+\frac12-}_{\rad}$, we have, as $t\to T$,
		\begin{equation}\label{eq:thm main law}
			\lambda(t)=(\ell+o(1))(T-t)^{2L},\qquad \gamma(t)\to\gamma_*,
		\end{equation}
		and
		\begin{equation}\label{eq:thm main remainder}
			u(t)\to u^*\quad\text{in}\quad L^2,\qquad
			u(t)-W_{\lambda(t),\gamma(t)}\to u^*\quad\text{in}\quad \dot H^1.
		\end{equation}

		\item (Asymptotic profile) There exists $a\in\mathbb C\setminus\{0\}$ such that
		\begin{equation}\label{eq:thm origin asymptotics}
			u^*(r)=ar^{2L-1}+O(r^{2L-\frac12})\quad\text{as }r\to0.
		\end{equation}
	\end{enumerate}
\end{thm}
This family has formal `codimension $2L$' stability; see Remark~\ref{rem:codimension}.

\subsection{Discussion of the results}

\begin{rem}[Method and novelty]
	Our approach follows the forward construction scheme based on modulation analysis with the energy method; see, for instance, \cite{RodnianskiSterbenz2010,RaphaelRodnianski2012,MerleRaphaelRodnianski2013Invention,RaphaelSchweyer2013CPAMHeat,RaphaelSchweyer2014APDE,MerleRaphaelRodnianski2015CambJMath,Kim2025JEMS,JeongKim2024arXiv}. The main difficulty in applying this scheme to our setting is the lack of a repulsive structure that can be directly exploited to control the remainder, unlike in previous blow-up constructions.
	
	Our main novelty is to overcome this difficulty by applying Kato's smoothing theory, as used in asymptotic stability analysis, to the forward construction. For the development of these methods, see, for instance, \cite{BuslaevPerelman1992,BuslaevPerelman1995,Cuccagna2001CPAM,KriegerSchlag2006JAMS,KairzhanPusateri2023PAA,Schlag2009annals,ErdoganSchlag2}. In particular, following \cite{Schlag2009annals,ErdoganSchlag2}, we establish a global spacetime Morawetz estimate for the linearized NLS flow, using the absence of nonzero embedded eigenvalues proved in our companion work \cite{Kim2026arXivNoEmbeeded}. This estimate provides the spacetime control needed to bound the higher-order energy, replacing the monotonicity derived from repulsive structures in previous blow-up constructions.
	
	We also note that, conversely, repulsive structures developed in blow-up analysis have been used to establish asymptotic stability \cite{Martel2022PMP,Martel2024invent}.
\end{rem}

\begin{rem}[Quantized blow-up]
	The study of quantized type II blow-up originated in nonlinear heat equations, beginning with the discovery of formal blow-up mechanisms \cite{HerreroVelazquez1992,HerreroVelazquez1994CRASPSI,FilippasHerreroVelazquez2000} and subsequent classification results for radial nonnegative solutions \cite{Mizoguchi2004AdvDE,Mizoguchi2007Math.Ann.,Mizoguchi2011TranAMS}; see also further constructions in \cite{Collot2017AnalPDE,CollotMerleRaphael2020JAMS}. Quantized blow-up was also studied in critical parabolic equations, including corotational harmonic map heat flow \cite{RaphaelSchweyer2013CPAMHeat,RaphaelSchweyer2014APDE}, the radial Stefan problem \cite{HadzicRaphael2019JEMS}, and Keller--Segel systems \cite{RaphaelSchweyer2014MathAnn,CollotGhoulMasmoudiNguyen2022CPAM}. Further examples for the energy-critical nonlinear heat equation were constructed in \cite{Schweyer2012JFA,PMW2019,delPinoMussoWeiZhang2020HeatQuantized,Harada2020AIHPC}.
	
	Beyond parabolic equations, smooth quantized blow-up was constructed for the energy-supercritical focusing NLS \cite{MerleRaphaelRodnianski2015CambJMath} and subsequently for energy-supercritical wave equations and wave maps \cite{Collot2018MemAmer,GhoulIbrahimNguyen2018JDE}. Blow-up constructions for compressible fluids and the energy-supercritical defocusing NLS were developed in \cite{MerleRaphaelRodnianskiSzeftel2022-1Ann.Math.,MerleRaphaelRodnianskiSzeftel2022Invention,MerleRaphaelRodnianskiSzeftel2022-2Ann.Math.}. In the energy-critical setting, smooth quantized blow-up was constructed for corotational wave maps \cite{Jeong2025AnalPDE} and the Calogero--Moser derivative nonlinear Schr\"odinger equation (CM-DNLS) \cite{JeongKim2024arXiv}. The single-bubble finite-time blow-up dynamics of (CM-DNLS) were subsequently classified into quantized and faster exotic regimes \cite{JeongKimKimKwon2026arXiv}. Very recently, smooth quantized blow-up was also constructed for the $1$-equivariant energy-critical Schr\"odinger map flow \cite{JXYZ2026arXiv}.
\end{rem}

\begin{rem}[Comparison with mass-critical NLS]
	We recall results on blow-up for the focusing mass-critical nonlinear Schr\"odinger equation
	\begin{equation*}
		i\partial_tu+\Delta u+|u|^{4/d}u=0,\qquad (t,x)\in\mathbb R\times\mathbb R^d.
	\end{equation*}
	At the threshold mass, blow-up uniqueness was proved in \cite{Merle1993Duke,RS2011JAMSinhomo}. Slightly above the threshold mass, Bourgain and Wang \cite{BourgainWang1997} constructed blow-up solutions with $\lambda(t)\sim T-t$; see also \cite{MRS2013AJM,MartelNaumkin2023Tunis}. Another class of blow-up solutions, exhibiting the log-log rate, was first constructed in \cite{Perelman2001}. Subsequently, Merle and Rapha\"el \cite{MerleRaphael2005AnnMath,MerleRaphael2003GAFA,Raphael2005MathAnnalen,MerleRaphael2004Invent,MerleRaphael2006JAMS,MerleRaphael2005CMP} established the stability of log-log blow-up and a dichotomy between $\lambda(t)\sim\sqrt{2\pi(T-t)/\log|\log(T-t)|}$ and $\lambda(t)\lesssim T-t$. More recently, Kwak and Kwon \cite{KwakKwon2026arXiv} extended the study beyond the single-bubble regime, establishing a universal log-log bound for radial solutions and ruling out the self-similar blow-up rate for general mass-critical NLS blow-up solutions.
	
	For the energy-critical NLS \eqref{eq:NLS} in $d=3$, log-log blow-up has been constructed in the standing-ring regime \cite{RaphaelSzeftel2009CMP,HolmerRoudenko2012APDE}, while the existence or exclusion of self-similar blow-up remains open. In the single-bubble regime, Schmid \cite{Schmid2026CVPDE} constructed solutions with a continuum of blow-up rates, and our result establishes quantized blow-up rates even for smooth solutions. Nevertheless, a rigidity result classifying the possible single-bubble blow-up rates remains open.
\end{rem}

\subsection{Strategy of the proof}
We use the notation introduced in Section~\ref{sec:notation}.
Our forward construction largely follows the modulation analysis developed in \cite{RodnianskiSterbenz2010, RaphaelRodnianski2012, MerleRaphaelRodnianski2013Invention, RaphaelSchweyer2013CPAMHeat, RaphaelSchweyer2014APDE, MerleRaphaelRodnianski2015CambJMath}.
However, unlike in these works, we do not have a repulsive structure that can be directly used to establish the required smallness of the remainder.
To overcome this difficulty, we adapt techniques from asymptotic stability theory, particularly those developed in \cite{Schlag2009annals,ErdoganSchlag2}. This is possible because forward construction and asymptotic stability share the same underlying principles of stability analysis.

We begin by examining the linearized flow around $W$. Writing $u=W+\eps$, we obtain the linearized equation
\begin{equation*}
	\partial_t\eps+iH\eps=0,\quad\text{where}\quad Hf\coloneqq-\Delta f-W^4f-4W^4\Re f.
\end{equation*}
The scaling and phase symmetries yield $H\Lambda W=H(iW)=0$, giving rise to the generalized kernel profiles defined recursively by
\begin{equation*}
	-iH\mathsf V_j=\mathsf V_{j-1},\qquad -iH\mathsf U_j=\mathsf U_{j-1},\qquad \mathsf V_0=\Lambda W,\quad \mathsf U_0=\tfrac{i}{2}W.
\end{equation*}
The operator $-iH$ also has eigenmodes $Y_\pm$ satisfying $-iHY_\pm=\pm e_0Y_\pm$, $e_0>0$, corresponding to exponentially growing and decaying solutions. These directions lead us to consider, for a fixed $L\ge1$, the spaces
\begin{align*}
	E_g(-iH)&\coloneqq\operatorname{span}_{\mathbb R}\{\mathsf V_j,\mathsf U_j:0\le j\le L\}\oplus\operatorname{span}_{\mathbb R}\{Y_+,Y_-\},\\
	E_g(-Hi)^\perp&=\{i\mathsf V_j,i\mathsf U_j:0\le j\le L\}^{\perp}\cap\{HY_+,HY_-\}^{\perp}.
\end{align*}
Here $E_g$ denotes the space spanned by the generalized kernel profiles and eigenmodes of the indicated operator, and orthogonality is taken with respect to the real inner product. Formally, both spaces are invariant under the linearized flow.

This decomposition leads us to control two distinct components of the perturbation.
For the part corresponding to the generalized kernel and eigenmodes, we use the modulation parameters $\lambda,\gamma,b_j,\eta_j$ to track the former and control the coefficients of the latter separately.
For the remaining part, we seek a Kato smoothing estimate for the linearized Schr\"odinger equation with decaying potentials.
Combined with the energy method, this estimate provides the spacetime control needed to propagate the smallness of the remainder.

To carry out this strategy at the nonlinear level, we need to reduce the errors arising from the approximate solution. To this end, we construct a modified profile $W+P$ of the form
\begin{equation*}
	P=\sum_{j=1}^L(b_j\mathsf V_j+\eta_j\mathsf U_j+S_j),
\end{equation*}
where $S_j$ are nonlinear corrections. We require $[W+P]_{\lambda,\gamma}$ to approximately solve \eqref{eq:NLS}. Writing $F(w)\coloneqq i(\Delta w+|w|^4w)$, this leads to the profile equation
\begin{equation*}
	\partial_sP-F(W+P)-\frac{\lambda_s}{\lambda}\Lambda(W+P)+i\gamma_s(W+P)\approx0,\qquad \frac{ds}{dt}=\frac{1}{\lambda^2}.
\end{equation*}
This is achieved by successively canceling the leading terms and the remaining nonlinear errors.

We first determine the formal modulation laws from the terms with the leading spatial growth.
The generalized kernel profiles $\mathsf V_j,\mathsf U_j$ grow like $r^{2j-1}$, and $\Lambda r^{2j-1}=(2j-\tfrac12)r^{2j-1}$.
Using these asymptotics and the generalized kernel relations, we obtain
\begin{gather*}
	\frac{\lambda_s}{\lambda}+b_1=0,\quad
	\gamma_s-\frac{\eta_1}{2}=0,
	\\
	\begin{aligned}
		(b_j)_s &= b_{j+1}-(2j-\tfrac12)b_1b_j-\tfrac12\eta_1\eta_j,
		\quad b_{L+1}\coloneqq0,
		\\
		(\eta_j)_s &= \eta_{j+1}-(2j-\tfrac12)b_1\eta_j+\tfrac12\eta_1b_j,
		\quad \eta_{L+1}\coloneqq0,
	\end{aligned}
	\qquad 1\le j\le L.
\end{gather*}
These laws admit a special solution
\begin{equation*}
	b_j^e(s)=c_js^{-j},\qquad \eta_j^e(s)=0,\qquad
	c_1=\frac{2L}{4L-1},\quad c_L\ne0.
\end{equation*}
The scale law gives $\lambda(s)\sim s^{-c_1}$, which, together with $dt/ds=\lambda^2$, yields $\lambda(t)\sim(T-t)^{2L}$.

While the modulation laws cancel the leading terms, nonlinear errors with lower spatial growth remain. Although these errors have better spatial decay, they are still too large for the remainder estimates. We therefore determine the nonlinear corrections $S_j$ recursively by applying $H^{-1}$ to eliminate these errors. Since these corrections have lower spatial growth than the generalized kernel profiles, they do not affect the leading modulation laws.
This procedure, which we call the \emph{tail computation}, yields an approximate profile consistent with the desired quantized blow-up dynamics.

We now turn to the construction of a nonlinear solution following the modified profile. Since $P$ grows spatially and does not belong to $\dot H^1$, we use different decompositions in the energy and higher-order topologies:
\begin{gather*}
	u=[w(s)]_{\lambda,\gamma},\qquad w=W+\td\eps=W+P+\eps,\\
	\td\eps\in\dot H^1,\qquad \eps\in\dot{\calH}^{2L+1},\qquad
	(\phi_j,\eps)=(\varphi_j,\eps)=0\quad(0\le j\le L).
\end{gather*}
Here $\phi_j,\varphi_j$ are fixed smooth compactly supported dual tests for $\mathsf V_j,\mathsf U_j$, and the orthogonality conditions determine the modulation parameters. The two decompositions play different roles. In the higher-order topology, we work with $\eps$ to exploit the cancellations built into the modified profile. In the energy topology, we instead work with the full perturbation $\td\eps=P+\eps$, avoiding the need to cut off the growing profile and the associated cutoff errors. This approach was introduced in \cite{KimKimKwon2024arxiv}.

We apply the energy method to the adapted radiation $\eps_L\coloneqq(-iH)^L\eps$.
Under this transformation, the generalized kernel profiles reduce to stationary kernel directions or vanish, while the eigenmodes persist.
Since the quadratic energy is not coercive in these directions, we introduce the projection $\Pi$ defined in \eqref{eq:projection} and control the eigenmode coefficients separately:
\begin{equation*}
	\alpha_\pm\coloneqq(-HY_\mp,\eps_L)=(\pm e_0)^L(-HY_\mp,\eps).
\end{equation*}
Combining coercivity on the projected space, we obtain
\begin{equation*}
	\|\eps\|_{\dot{\calH}^{2L+1}}\lesssim\|\Pi\eps_L\|_{\dot H^1}+|\alpha_+|+|\alpha_-|.
\end{equation*}
To propagate this control in time, we consider the evolution of the adapted radiation, which takes the approximate form
\begin{equation*}
	\partial_s\eps_L+iH\eps_L-\frac{\lambda_s}{\lambda}(\Lambda+2L)\eps_L+i\gamma_s\eps_L\approx0.
\end{equation*}
Along the formal dynamics, $\lambda_s/\lambda$ and $\gamma_s$ tend to zero. We therefore treat the scale and phase terms as perturbations and seek the required spacetime estimate for the projected linearized flow
\begin{equation*}
	\partial_tf=\Pi(-iH)f.
\end{equation*}
Denoting this flow by $U_c(t)$, we establish the global Morawetz estimate
\begin{equation*}
	\|U_c(t)f_0\|_{L^2_t\mathrm{Mor}_x(\mathbb R)}\lesssim\|f_0\|_{\dot H^1}.
\end{equation*}
Following Kato's smoothing theory, we derive this estimate from a uniform resolvent estimate \eqref{eq:uniform stationary esti} for
\begin{equation*}
	(z-\Pi(-iH))f=F,\qquad \Re z>0.
\end{equation*}
For large $|z|$, the estimate follows from Green function estimates. For bounded $|z|$, we argue by contradiction, showing that failure of the estimate would produce a nonzero solution of the homogeneous equation. Such a solution is ruled out by the projection conditions and the spectral properties of $-iH$, including the absence of embedded eigenvalues established in \cite{Kim2026arXivNoEmbeeded}.

Combining the Morawetz estimate with the higher-order energy and modulation estimates, we propagate the required bounds for the radiation and stable parameters. Finally, we select the unstable initial parameters through a topological shooting argument, thereby constructing a solution with the asymptotics stated in Theorem~\ref{thm:main}.

\vspace{5pt}
\noindent\mbox{\textbf{Acknowledgements.~}}The author would like to thank Kihyun Kim and Seunghwan Hwang for valuable discussions, and Soonsik Kwon for helpful comments. The author is supported by a KIAS Individual Grant (MG105201) at Korea Institute for Advanced Study.

\section{Notation and preliminaries}\label{sec:notation}

\subsection{Notation}

Unless otherwise specified, functions are complex-valued and radial on $\mathbb R^3$. We write $\mathbb R^+\coloneqq(0,\infty)$, $r\coloneqq |x|$, and $\langle x\rangle\coloneqq (1+|x|^2)^{\frac{1}{2}}$, or $\langle r\rangle\coloneqq (1+r^2)^{\frac{1}{2}}$. For a radial function, we write $f_r\coloneqq \partial_rf$, $f_{rr}\coloneqq \partial_r^2f$, and
\begin{equation*}
	\Delta f=f_{rr}+\frac2r f_r.
\end{equation*}
The notation $\nabla^jf$ denotes the full tensor of Cartesian derivatives of order $j$, with $\nabla^0f\coloneqq f$.
The Kronecker delta is $\delta_{ij}\coloneqq 1$ when $i=j$ and $\delta_{ij}\coloneqq 0$ otherwise.

For quantities $A\in\mathbb C$ and $B\geq0$, we write $A\lesssim B$ or $A=O(B)$ if $|A|\leq CB$ holds for some implicit constant $C$. For $A,B\geq0$, we write $A\sim B$ if $A\lesssim B$ and $B\lesssim A$. If $C$ depends on some parameters, then we write them as a subscript. For $B>0$, we write $A=o(B)$ if $A/B\to0$ in the limit under consideration.

We denote by $L^p= L^p(\mathbb R^3)$ the usual Lebesgue spaces. For radial functions and $1\le p<\infty$, we use the normalization
\begin{equation*}
	\|f\|_{L^{p}}^p\coloneqq \int_0^\infty|f(r)|^pr^2dr=\frac1{4\pi}\int_{\mathbb R^3}|f(x)|^pdx.
\end{equation*}
For nonradial functions and vector or tensor fields, we use the measure $dx/(4\pi)$. For a radial interval $I\subset\mathbb R^+$, $L^p(I,dr)$ uses $dr$; $L^p(dr)$ means $L^p((0,\infty),dr)$.

We denote by $H^k(\mathbb R^3)$ and $\dot H^k(\mathbb R^3)$ the usual inhomogeneous and homogeneous Sobolev spaces, respectively, and write $H^\infty(\mathbb R^3)\coloneqq\bigcap_{k\ge0}H^k(\mathbb R^3)$.
We denote the Schwartz space by $\calS(\mathbb R^3)$ and the radial subspace by a subscript $\rad$. We often omit $\mathbb R^3$ and the subscript $\rad$ when there is no confusion.
We use the adapted Sobolev spaces $\dot{\calH}^{m}$, with norms defined by
\begin{equation}\label{eq:dot calH def}
	\|f\|_{\dot{\calH}^{m}}^2\coloneqq \sum_{j=0}^m\|\langle x\rangle^{-(m-j)}\nabla^jf(x)\|_{L^2}^2.
\end{equation}
We use the radial subspace and note that $\dot{\calH}^{1}=\dot H^1$ with equivalent norms.
We also use the real Morawetz inner product and its induced norm:
\begin{align*}
	(f,g)_{\mathrm{Mor}}&\coloneqq\Re\int_0^\infty
	\left(\frac{f_r\ol{g_r}}{\langle r\rangle^2}
	+\frac{f\ol g}{\langle r\rangle^4}\right)r^2dr,\\
	\|f\|_{\mathrm{Mor}}^2&\coloneqq(f,f)_{\mathrm{Mor}}
	=\|\langle r\rangle^{-1}\nabla f\|_{L^2}^2
	+\|\langle r\rangle^{-2}f\|_{L^2}^2.
\end{align*}
For a radial interval $I\subset\mathbb R^+$, we write
\begin{equation*}
	\|f\|_{\mathrm{Mor}(I)}^2\coloneqq\int_I
	\left(\frac{|f_r|^2}{\langle r\rangle^2}
	+\frac{|f|^2}{\langle r\rangle^4}\right)r^2dr.
\end{equation*}
We write $B_R\coloneqq\{x\in\mathbb R^3:|x|<R\}$ and use $H^k(r<R)$ for $H^k(B_R)$.
We fix a nonnegative smooth radial cutoff $\chi$, equal to one on $B_1$ and zero outside $B_2$, and write $\chi_R(x)\coloneqq\chi(x/R)$.

 For a spatial Banach space $X$ and a time interval $I$, we write
\begin{equation*}
	\|f\|_{L^p_tX_x(I)}\coloneqq \left(\int_I\|f(t)\|_X^p dt\right)^{\frac{1}{p}}\qquad(1\le p<\infty),
\end{equation*}
with the usual modification for $p=\infty$.

For $\lambda>0$ and $\gamma\in\mathbb R$, we define the scaling and phase rotation by
\begin{equation}
	f_{\lambda,\gamma}(x)\coloneqq \lambda^{-\frac{1}{2}}e^{i\gamma}f(x/\lambda),\qquad
	f_\lambda\coloneqq f_{\lambda,0},\qquad \Lambda f\coloneqq \frac12f+x\cdot\nabla f.
\end{equation}
For radial $f$, $\Lambda f=\frac12f+rf_r$. We have
\begin{gather*}
	\lambda\partial_\lambda f_{\lambda,\gamma}=-(\Lambda f)_{\lambda,\gamma},\qquad
	\partial_\gamma f_{\lambda,\gamma}=if_{\lambda,\gamma},\\
	\|f_{\lambda,\gamma}\|_{L^{p}}=\lambda^{\frac{3}{p}-\frac{1}{2}}\|f\|_{L^{p}},\qquad
	\|f_{\lambda,\gamma}\|_{\dot H^j}=\lambda^{1-j}\|f\|_{\dot H^j}.
\end{gather*}
We use the real inner product
\begin{equation}
	(\phi,f)\coloneqq \Re\int_0^\infty\phi(r)\ol{f(r)}r^2dr.
\end{equation}
For vector or tensor fields, we use the componentwise real inner product.
Integrating by parts, we obtain
\begin{equation}\label{eq:notation dilation adjoint}
	(\Lambda f,g)=-(f,(\Lambda+2)g).
\end{equation}
For a radial function $f$ and an integer $k\ge0$, we write
\begin{equation}\label{eq:notation pointwise size}
	|f|_k(r)\coloneqq \max_{0\le j\le k}|r^j\partial_r^jf(r)|.
\end{equation}
Leibniz's rule gives $|fg|_k\lesssim|f|_k|g|_k$, and $|\Lambda f|_k\lesssim|f|_{k+1}$.

We recall the ground state $W$ of \eqref{eq:NLS}. It is the positive radial stationary solution normalized by $W(0)=1$, given by
\begin{equation*}
	W(r)\coloneqq (1+r^2/3)^{-\frac{1}{2}},\qquad -\Delta W=W^5.
\end{equation*}

The linearized operator at $W$ and the nonlinear evolution are
\begin{equation*}
	Hf\coloneqq -\Delta f-W^4f-4W^4\Re f,\qquad 
	F(u)\coloneqq i(\Delta u+|u|^4u).
\end{equation*}
The operator $H$ is symmetric for the real inner product. Products of operators denote composition; in particular, $(Hi)f=H(if)$ and $(-iH)^*f=H(if)$.
For complex-valued $f,g$, we define the nonlinear remainder by
\begin{equation}\label{eq:notation nonlinearity}
	\calN_f(g)\coloneqq i\big[|f+g|^4(f+g)-|f|^4f-3|f|^4g-2|f|^2f^2\ol g\big].
\end{equation} 

\subsection{Preliminary estimates}

We record the growth bound and mapping property needed to use the profiles in the adapted Sobolev spaces \eqref{eq:dot calH def}, and the commutator estimate used in the higher-order energy argument.

\begin{lem}\label{lem:dot calH pointwise}
	Let $k\ge2$ be an integer. For $f\in \dot{\calH}^{k}$, there is a continuous representative satisfying
	\begin{equation}\label{eq:dot cal pointwise}
		|f(x)|\lesssim\|f\|_{\dot{\calH}^{k}}\langle x\rangle^{k-\frac{3}{2}}.
	\end{equation}
	If $P$ is smooth and radial with $|P|_k\lesssim\langle r\rangle^{k-2}$, then $P\in \dot{\calH}^{k}$. Moreover, for $k\ge3$, $iH:\dot{\calH}^{k}\to \dot{\calH}^{k-2}$ is bounded. In particular,
	\begin{equation}\label{eq:dot calH linear map}
		\|(iH)^Lf\|_{\dot H^1}\lesssim\|f\|_{\dot{\calH}^{2L+1}}.
	\end{equation}
\end{lem}

\begin{proof}
	By the Sobolev embedding of $H^2$ into $L^\infty$, for $R\geq1$, we have
	\begin{equation*}
		\|\chi_{R}f\|_{L^\infty}
		\lesssim\sum_{j=0}^2 R^{j-\frac{3}{2}}\|\chi_{2R}\nabla^jf\|_{L^2}
		\lesssim R^{k-\frac{3}{2}}\|f\|_{\dot{\calH}^{k}},
	\end{equation*}
	which proves \eqref{eq:dot cal pointwise}.

	From \eqref{eq:notation pointwise size}, $|P|_k\lesssim\langle r\rangle^{k-2}$ implies $|\nabla^jP|\lesssim r^{k-2-j}$ for $r>1$. Thus $P\in \dot{\calH}^{k}$.

	To prove \eqref{eq:dot calH linear map}, we note that the Laplacian $\Delta:\dot{\calH}^{k}\to\dot{\calH}^{k-2}$ is bounded by the definition of the norms. The decay $|\nabla^jW^4|\lesssim\langle x\rangle^{-4-j}$ gives the same boundedness for the map $f\mapsto W^4f+4W^4\Re f$. Thus $iH:\dot{\calH}^{k}\to\dot{\calH}^{k-2}$ is bounded for $k\ge3$. Iterating $L$ times from $\dot{\calH}^{2L+1}$ to $\dot{\calH}^{1}=\dot H^1$ proves \eqref{eq:dot calH linear map}.
\end{proof}

For $m=2L+1$, these spaces contain profiles growing like $r^{2L-1}$. Lemma~\ref{lem:weighted elliptic} gives the converse to \eqref{eq:dot calH linear map} after removing the kernel of $(iH)^L$.

The following lemma gives pointwise bounds for the profiles and weighted bounds for the remainder. For operators $A,B$, we write $[A,B]f\coloneqq A(Bf)-B(Af)$.

\begin{lem}[Commutator bounds]
	For integers $j\ge1$ and $m\ge0$ and a smooth function $f$, we have
	\begin{equation}\label{eq:lin commutator coefficients}
		|\nabla^m([(iH)^j,\Lambda]f-2j(iH)^jf)|
		+|\nabla^m[(iH)^j,i]f|
		\lesssim_{j,m}\sum_{d=0}^{2j-2+m}\frac{|\nabla^df|}{\langle r\rangle^{2j+2+m-d}}.
	\end{equation}
	We also have
	\begin{equation}\label{eq:lin iterated products}
		|\nabla^m(iH)^jf|\lesssim_{j,m}\sum_{d=0}^{2j+m}\frac{|\nabla^df|}{\langle r\rangle^{2j+m-d}}.
	\end{equation}
	For $f\in\dot{\calH}^{2j+1}$,
	\begin{equation}\label{eq:lin weighted commutators}
		\begin{aligned}
			&\|\langle r\rangle\{[(iH)^j,\Lambda]f-2j(iH)^jf\}\|_{L^2}\\
			&+\|\langle r\rangle^2\nabla\{[(iH)^j,\Lambda]f-2j(iH)^jf\big\}\|_{L^2}
			\lesssim_j\|f\|_{\dot{\calH}^{2j+1}},\\
			&\|\langle r\rangle[(iH)^j,i]f\|_{L^2}
			+\|\langle r\rangle^2\nabla[(iH)^j,i]f\|_{L^2}
			\lesssim_j\|f\|_{\dot{\calH}^{2j+1}}.
		\end{aligned}
	\end{equation}
	The same estimates hold with $iH$ replaced by $-iH$.
\end{lem}
\begin{proof}
	Since $iHf=-i\Delta f-3iW^4f-2iW^4\ol f$, from Leibniz's rule, we have
	\begin{equation*}
		|\nabla^m(iH)f|\lesssim_m|\nabla^{m+2}f|+\sum_{d=0}^m\frac{|\nabla^df|}{\langle r\rangle^{4+m-d}}.
	\end{equation*}
	Iterating this estimate, we get \eqref{eq:lin iterated products}. Since $[\Delta,\Lambda]=2\Delta$, we also have
	\begin{equation}\label{eq:lin basic commutators}
		[iH,\Lambda]f-2iHf=4iW^3(\Lambda W)(3f+2\ol f),\qquad
		[iH,i]f=-4W^4\ol f.
	\end{equation}
	Since we have $|\nabla^mW^4|+|\nabla^m(4W^3\Lambda W)|\lesssim_m\langle r\rangle^{-4-m}$ together with \eqref{eq:lin iterated products}, we obtain, for $k,m\ge0$,
	\begin{equation}\label{eq:coefficient products}
		|\nabla^m(iH)^k(W^4f)|+|\nabla^m(iH)^k(4W^3(\Lambda W)f)|
		\lesssim_{k,m}\sum_{d=0}^{2k+m}\frac{|\nabla^df|}{\langle r\rangle^{4+2k+m-d}}.
	\end{equation}
	For $j\ge1$, by a direct computation, we have
	\begin{equation}
		\begin{aligned}
			[(iH)^j,\Lambda]-2j(iH)^j
			&=\sum_{k=0}^{j-1}(iH)^k\big([iH,\Lambda]-2iH\big)(iH)^{j-1-k},\\
			[(iH)^j,i]
			&=\sum_{k=0}^{j-1}(iH)^k[iH,i](iH)^{j-1-k}.
		\end{aligned}\label{eq:lin power commutators}
	\end{equation}
	Combining \eqref{eq:lin basic commutators}--\eqref{eq:lin power commutators} with \eqref{eq:lin iterated products}, we obtain
	\begin{align*}
		&|\nabla^m([(iH)^j,\Lambda]f-2j(iH)^jf)|+|\nabla^m[(iH)^j,i]f|\\
		&\lesssim_{j,m}\sum_{k=0}^{j-1}\sum_{n=0}^{2k+m}\frac{|\nabla^n(iH)^{j-1-k}f|}{\langle r\rangle^{4+2k+m-n}}
		\lesssim_{j,m}\sum_{d=0}^{2j-2+m}\frac{|\nabla^df|}{\langle r\rangle^{2j+2+m-d}},
	\end{align*}
	which proves \eqref{eq:lin commutator coefficients}.
	
	Multiplying \eqref{eq:lin commutator coefficients} by $\langle r\rangle$ or $\langle r\rangle^{2}$, taking $L^2$ norms, and using \eqref{eq:dot calH def}, we conclude \eqref{eq:lin weighted commutators}.
\end{proof}

\section{Linearized operators and coercivity}

We first recall the well-known spectral properties of the linearized operators from \cite{DuyckaertsMerle2009Threshold}. We then construct the generalized kernel using a regular Green inverse and prove coercivity in the adapted Sobolev spaces.

\subsection{Spectral properties and generalized kernel}

We first recall the coercivity of $H$ and the two real eigenvalues of $-iH$, which determine the coordinates in \eqref{eq:eigenmode mode def}.

\begin{lem}[Spectral properties and coercivity of $H$]\label{lem:first coercivity}
	The following properties hold.
	\begin{itemize}
		\item (Kernel) The energy-space kernel is
		\begin{equation}\label{eq:energy kernel}
			\begin{gathered}
				\{f\in\dot H^1:Hf=0\}=\operatorname{span}_{\mathbb R}\{\Lambda W,iW,\nabla W\},\\
				\{f\in\dot H^1_{\rad}:Hf=0\}=\operatorname{span}_{\mathbb R}\{\Lambda W,iW\}.
			\end{gathered}
		\end{equation}

		\item (Negative direction) The quadratic form $(Hf,f)$ has exactly one negative direction, and
		\begin{equation*}
			(HW,W)=-4\|W\|_{\dot H^1}^2.
		\end{equation*}

		\item (Eigenmodes of $-iH$) There exist $e_0>0$ and nonzero $Y_+,Y_-\in\calS_{\rad}$ such that
		\begin{equation}\label{eq:eigenfunctions}
			-iHY_\pm=\pm e_0Y_\pm,\qquad Y_-=\ol{Y_+}
		\end{equation}
		with
		\begin{equation}\label{eq:eigenmode normalization}
			(HY_+,Y_+)=(HY_-,Y_-)=0,\qquad (HY_+,Y_-)=-1.
		\end{equation}

		\item (Energy coercivity) Let $f\in\dot H^1_{\rad}$ satisfy $(\nabla W,\nabla f)=0$, $(\nabla\Lambda W,\nabla f)=0$, and $(\nabla(iW),\nabla f)=0$. Then
		\begin{equation*}
			(Hf,f)\sim\|f\|_{\dot H^1}^2.
		\end{equation*}

		\item (Coercivity modulo the kernel) For $f\in\dot H^1_{\rad}$ with $(HY_+,f)=(HY_-,f)=0$, we have
		\begin{equation}\label{eq:continuous coercivity}
			(Hf,f)\sim
			\inf_{a_1,a_2\in\mathbb R}
			\|f-a_1\Lambda W-a_2iW\|_{\dot H^1}^2.
		\end{equation}
		The infimum is attained uniquely at
		\begin{equation*}
			a_1=\frac{(\nabla\Lambda W,\nabla f)}
			{\|\Lambda W\|_{\dot H^1}^2},
			\qquad
			a_2=\frac{(\nabla(iW),\nabla f)}
			{\|W\|_{\dot H^1}^2}.
		\end{equation*}
	\end{itemize}
\end{lem}
The proof of Lemma~\ref{lem:first coercivity} can be found in \cite{DuyckaertsMerle2009Threshold}.
By \eqref{eq:energy kernel} and \eqref{eq:eigenmode normalization}, we have
\begin{equation}\label{eq:dual mode tests}
	(-HY_\pm,Y_\pm)=(-HY_\pm,\Lambda W)=(-HY_\pm,iW)=0.
\end{equation}

To construct the generalized kernel, we define a regular inverse of $H$ by its Green function, following the constructions in \cite{RaphaelSchweyer2014APDE}. The fundamental solutions and weights below are computed for the present NLS operator.

For the real and imaginary parts of $H$, we write
\begin{equation}\label{eq:definition H}
	H_+\coloneqq -\Delta-5W^4,\qquad H_-\coloneqq -\Delta-W^4,
	\qquad Hf=H_+\Re f+iH_-\Im f.
\end{equation}
For each $\iota\in\{+,-\}$, the Green function uses two linearly independent solutions of the homogeneous radial equation $H_\iota f=0$: the regular solution $h_\iota$, with $h_\iota(0)=1$, and the singular solution $\Gamma_\iota$, with $\Gamma_\iota(r)=-r^{-1}+O(r)$ as $r\to0$. They are given by
\begin{align}
	h_+(r)&\coloneqq 2(\Lambda W)(r)=\frac{\sqrt3(3-r^2)}{(r^2+3)^{\frac{3}{2}}},&
	\Gamma_+(r)&\coloneqq \frac{-r^4+18r^2-9}{\sqrt3\,r(r^2+3)^{\frac{3}{2}}},\notag\\
	h_-(r)&\coloneqq W(r)=\frac{\sqrt3}{(r^2+3)^{\frac{1}{2}}},&
	\Gamma_-(r)&\coloneqq \frac{r^2-3}{\sqrt3\,r(r^2+3)^{\frac{1}{2}}}.\label{eq:green fundamentals}
\end{align}
Direct differentiation gives, for $\iota\in\{+,-\}$,
\begin{gather}\label{eq:green wronskian}
	H_\iota h_\iota=H_\iota\Gamma_\iota=0,\qquad
	h_\iota(0)=1,\qquad
	r^2(h_\iota\partial_r\Gamma_\iota-\partial_r h_\iota\Gamma_\iota)=1,\\
	|h_\iota|_j\lesssim\langle r\rangle^{-1},\qquad
	|\Gamma_\iota|_j\lesssim\frac{\langle r\rangle}{r}.
	\label{eq:green size}
\end{gather}
In particular, $\Gamma_+$ is smooth at $r=\sqrt3$, where $\Lambda W=0$.

For $0<\rho<r$, we define
\begin{equation}\label{eq:green kernel}
	G_\iota(r,\rho)\coloneqq h_\iota(r)\Gamma_\iota(\rho)-\Gamma_\iota(r)h_\iota(\rho),
	\qquad
	(H_\iota^{-1}g)(r)\coloneqq \int_0^rG_\iota(r,\rho)g(\rho)\rho^2\,d\rho.
\end{equation}
We use the real-linear inverse
\begin{equation}\label{eq:green inverse}
	H^{-1}g\coloneqq H_+^{-1}(\Re g)+iH_-^{-1}(\Im g).
\end{equation}
The normalization is $H^{-1}g(0)=0$; \eqref{eq:energy kernel} prevents invertibility on the energy space.

\begin{lem}[Inversion of $H$]\label{lem:inverse}
	For a smooth regular radial $g$, the function $f\coloneqq H^{-1}g$ in \eqref{eq:green inverse} is the unique smooth regular radial solution of
	\begin{equation}
		Hf=g,\qquad f(0)=0.
	\end{equation}
	Moreover, the following bounds hold for all integers $j\ge0$:
	\begin{itemize}
		\item If $|g|_j\lesssim\langle r\rangle^p$ for all $j\ge0$ and some $p\ge0$, then
		\begin{equation}\label{eq:inverse growth}
			|H^{-1}g|_j\lesssim r^2\langle r\rangle^p.
		\end{equation}

		\item If $|g-cr^{-1}|_j\lesssim(r\langle r\rangle^2)^{-1}$ for all $j\ge0$ and some $c\in\mathbb C$, then
		\begin{equation}\label{eq:first inverse tail}
			\left|H^{-1}g+\frac c2r\right|_j\lesssim\frac r{\langle r\rangle}.
		\end{equation}

		\item If $|g-cr^p|_j\lesssim\langle r\rangle^{p-1}$ for all $j\ge0$ and some $p\ge1$, then
		\begin{equation}\label{eq:higher inverse tail}
			\left|H^{-1}g+\frac{c}{(p+2)(p+3)}r^{p+2}\right|_j
			\lesssim r^2\langle r\rangle^{p-1}.
		\end{equation}
	\end{itemize}
\end{lem}

\begin{proof}
	Equation \eqref{eq:green wronskian} gives $G_\iota(r,r)=0$ and $\partial_rG_\iota(r,\rho)|_{r=\rho}=-\rho^{-2}$. Therefore, differentiating \eqref{eq:green kernel} twice, we deduce $H_\iota H_\iota^{-1}g=g$. The function $h_\iota$ is smooth and even, and $\Gamma_\iota=-r^{-1}+O(r)$ is odd near $r=0$. Substituting their Taylor expansions in \eqref{eq:green kernel}, we have
	\begin{equation*}
		(H_\iota^{-1}g)(r)=-\frac{g(0)}6r^2+O(r^4).
	\end{equation*}
	Thus, if $g$ is smooth, regular, and radial, then $H_\iota^{-1}g$ is also smooth, regular, and radial, and vanishes at the origin. Uniqueness follows since each regular homogeneous solution is a multiple of $h_\iota$, with $h_\iota(0)=1$.

	For \eqref{eq:inverse growth}, \eqref{eq:green size} gives, when $r\ge1$,
	\begin{equation*}
		|H_\iota^{-1}g(r)|
		\lesssim \frac{1}{\langle r\rangle}\int_0^r \rho\langle \rho\rangle^{p+1}\,d\rho
		+\frac{\langle r\rangle}{r}\int_0^r \rho^2\langle \rho\rangle^{p-1}\,d\rho
		\lesssim r^2\langle r\rangle^p.
	\end{equation*}
	Similarly, differentiating \eqref{eq:green kernel}, we get $|r\partial_r H_\iota^{-1}g|\lesssim r^2\langle r\rangle^p$. For higher weighted derivatives, we use
	\begin{equation}\label{eq:radial inverse equation}
		\partial_r^2H^{-1}g
		=-\frac2r\partial_rH^{-1}g-g
		-W^4\big(H^{-1}g+4\Re(H^{-1}g)\big),
	\end{equation}
	which follows from $H(H^{-1}g)=g$.
	Iterating \eqref{eq:radial inverse equation} yields \eqref{eq:inverse growth}.

	Next, we prove \eqref{eq:first inverse tail}. For $h$ satisfying $|h|_j\lesssim(r\langle r\rangle^2)^{-1}$ for all $j\ge0$, the Green formula and its first derivative give, for $r\ge1$,
	\begin{equation*}
		|H^{-1}h|_1
		\lesssim \frac{1}{\langle r\rangle}\int_0^r\langle \rho\rangle^{-1}\,d\rho
		+\frac{\langle r\rangle}{r}\int_0^r \rho\langle \rho\rangle^{-3}\,d\rho
		\lesssim \frac{r}{\langle r\rangle}.
	\end{equation*}
	Higher derivatives follow from \eqref{eq:radial inverse equation} as above, proving
	\begin{equation}\label{eq:green decaying source}
		|H^{-1}h|_j\lesssim r\langle r\rangle^{-1}.
	\end{equation}
	We now set $h\coloneqq H^{-1}g+cr/2$. Then
	\begin{equation*}
		Hh=g-\frac c r-\frac r2W^4(c+4\Re (c)),\qquad
		|Hh|_j\lesssim (r\langle r\rangle^2)^{-1}.
	\end{equation*}
	Since $h(0)=0$, we have $h=H^{-1}(Hh)$. By \eqref{eq:green decaying source}, this proves \eqref{eq:first inverse tail}.

	To prove \eqref{eq:higher inverse tail}, we set $h\coloneqq H^{-1}g+cr^{p+2}/((p+2)(p+3))$. Then
	\begin{equation*}
		Hh=g-cr^p-\frac{r^{p+2}W^4}{(p+2)(p+3)}(c+4\Re c),\qquad
		|Hh|_j\lesssim\langle r\rangle^{p-1}.
	\end{equation*}
	As above, with $h=H^{-1}(Hh)$, we get \eqref{eq:higher inverse tail}.
\end{proof}

Using \eqref{eq:green inverse}, we define
\begin{equation}\label{eq:g kernels}
	\begin{aligned}
		\mathsf V_0&\coloneqq \Lambda W,\qquad &\mathsf U_0&\coloneqq \tfrac{i}{2}W,\\
		\mathsf V_k&\coloneqq H^{-1}(i\mathsf V_{k-1}),\qquad
		&\mathsf U_k&\coloneqq H^{-1}(i\mathsf U_{k-1})\quad(k\ge1).
	\end{aligned}
\end{equation}
Thus $H\mathsf V_k=i\mathsf V_{k-1}$, $H\mathsf U_k=i\mathsf U_{k-1}$, and $\mathsf V_k(0)=\mathsf U_k(0)=0$ for $k\ge1$.

We denote $C_k\coloneqq -\frac{\sqrt3}{2}\frac{(-i)^k}{(2k)!}$ for $k\ge1$.

\begin{lem}[Generalized-kernel bounds]
	For $k\ge1$ and $j\ge0$,
	\begin{gather}
		|\mathsf V_k|_j+|\mathsf U_k|_j
		\lesssim r^{2k}\langle r\rangle^{-1},\label{eq:kernel growth}\\
		\left|\mathsf V_k-C_kr^{2k-1}\right|_j
		+\left|\mathsf U_k+iC_kr^{2k-1}\right|_j
		\lesssim r^{2k-1}\langle r\rangle^{-1}.\label{eq:kernel tails}
	\end{gather}
	Consequently,
	\begin{equation}\label{eq:tail cancellations}
		\begin{gathered}
			\left|\big(\Lambda-2k+\tfrac12\big)\mathsf V_k\right|_j
			+\left|\big(\Lambda-2k+\tfrac12\big)\mathsf U_k\right|_j
			\lesssim r^{2k}\langle r\rangle^{-2},\\
			|i\mathsf V_k+\mathsf U_k|_j+|i\mathsf U_k-\mathsf V_k|_j
			\lesssim r^{2k}\langle r\rangle^{-2}.
		\end{gathered}
	\end{equation}
\end{lem}

\begin{proof}
	We prove \eqref{eq:kernel growth} and \eqref{eq:kernel tails} by induction on $k$. The explicit formula for $W$ gives, for $j\ge0$,
	\begin{equation}\label{eq:W tails}
		|W-\sqrt{3}r^{-1}|_j
		+|\Lambda W+\tfrac{\sqrt3}{2}r^{-1}|_j
		\lesssim (r\langle r\rangle^2)^{-1}.
	\end{equation}
	By \eqref{eq:first inverse tail}, we have \eqref{eq:kernel tails} for $k=1$. Since $\mathsf V_1$ and $\mathsf U_1$ are smooth regular radial functions vanishing at the origin, $|\mathsf V_1|_j+|\mathsf U_1|_j\lesssim r^2$ for $r\le1$. Together with \eqref{eq:kernel tails}, this proves \eqref{eq:kernel growth} for $k=1$.

	For $k\ge2$, we assume that \eqref{eq:kernel growth} and \eqref{eq:kernel tails} hold for $k-1$. Since $r^{2k-3}\langle r\rangle^{-1}\le\langle r\rangle^{2k-4}$, we can apply \eqref{eq:higher inverse tail} to $i\mathsf V_{k-1}$ and $i\mathsf U_{k-1}$ with $p=2k-3$. Using
	\begin{equation*}
		C_k=-\frac{iC_{k-1}}{(2k-1)(2k)},
	\end{equation*}
	we obtain
	\begin{equation*}
		\left|\mathsf V_k-C_kr^{2k-1}\right|_j
		+\left|\mathsf U_k+iC_kr^{2k-1}\right|_j
		\lesssim r^2\langle r\rangle^{2k-4}.
	\end{equation*}
	This gives \eqref{eq:kernel growth} and \eqref{eq:kernel tails} for $r\ge1$. For $r\le1$, \eqref{eq:kernel growth} with $k-1$ and the Green formula \eqref{eq:green kernel} yield
	\begin{equation*}
		|\mathsf V_k|_j+|\mathsf U_k|_j\lesssim r^{2k}.
	\end{equation*}
	This completes the induction for \eqref{eq:kernel growth} and \eqref{eq:kernel tails}.

	For \eqref{eq:tail cancellations}, it remains to consider $r\ge1$, since \eqref{eq:kernel growth} applies on $r\le1$. Since
	\begin{equation*}
		\big(\Lambda-2k+\tfrac12\big)r^{2k-1}=0,
	\end{equation*}
	\eqref{eq:kernel tails} at order $j+1$ gives the first bound for $r\ge1$. The leading terms also cancel in $i\mathsf V_k+\mathsf U_k$ and $i\mathsf U_k-\mathsf V_k$, giving the second bound there.
\end{proof}

\subsection{Coercivity in adapted Sobolev spaces}

The generalized kernel in \eqref{eq:g kernels} is the obstruction to coercivity of $(iH)^L$. We construct compact dual tests to remove it and prove \eqref{eq:weighted reconstruction}. The same tests fix the modulation parameters in Section~\ref{sec:decomposition}.

For each integer $L\ge1$, we define
\begin{equation}
	\calK_L\coloneqq \operatorname{span}_{\mathbb R}\{\mathsf V_j,\mathsf U_j:0\le j<L\}.
\end{equation}
We define
\begin{equation}\label{eq:base tests}
	\td\phi\coloneqq \frac{\chi\Lambda W}{\int_0^\infty\chi(\Lambda W)^2r^2dr},\qquad
	\td\varphi\coloneqq \frac{\chi(iW/2)}{\frac14\int_0^\infty\chi W^2r^2dr}.
\end{equation}

We also define
\begin{equation}\label{eq:adapted tests}
	\phi_L\coloneqq(Hi)^L\td\phi,\qquad
	\varphi_L\coloneqq(Hi)^L\td\varphi.
\end{equation}

\begin{lem}[Kernel and compact dual tests]\label{lem:dual tests}
	The kernel of $(iH)^L:\dot{\calH}^{2L+1}_{\rad}\to\dot H^1_{\rad}$ is
	\begin{equation}\label{eq:regular kernel}
		\{f\in\dot{\calH}^{2L+1}_{\rad}:(iH)^Lf=0\}=\calK_L.
	\end{equation}
	The functions $\{\mathsf V_j,\mathsf U_j:0\le j\le L\}$ form a real basis of $\calK_{L+1}\subset\dot{\calH}^{2L+1}_{\rad}$. There exist fixed radial functions $\phi_j,\varphi_j\in C_c^\infty$, $0\le j<L$, such that, for $0\le j,k\le L$,
	\begin{equation}\label{eq:duality}
		(\phi_j,\mathsf V_k)=(\varphi_j,\mathsf U_k)=\delta_{jk},\qquad
		(\phi_j,\mathsf U_k)=(\varphi_j,\mathsf V_k)=0.
	\end{equation}
\end{lem}

\begin{proof}
	We first prove \eqref{eq:regular kernel}. Elliptic regularity gives $f\in C^\infty_{\rad}$ whenever $f\in\dot{\calH}^{2L+1}_{\rad}$ and $(iH)^Lf=0$. For smooth radial functions, we show by induction that $(iH)^jf=0$ implies $f\in\calK_j$. The case $j=1$ follows from Lemma~\ref{lem:inverse}, after subtracting the real linear combination of $\mathsf V_0,\mathsf U_0$ with the same value at the origin. For $j\ge2$, the induction hypothesis gives $iHf\in\calK_{j-1}$. By \eqref{eq:g kernels}, we can choose $g\in\calK_j$ with $iHg=iHf$, and then $f-g\in\calK_1$. This proves the required inclusion with $j=L$.

	Conversely, $\calK_{L+1}\subset\dot{\calH}^{2L+1}_{\rad}$ by Lemma~\ref{lem:dot calH pointwise} and \eqref{eq:kernel growth}. Since $(iH)^Lf=0$ for $f\in\calK_L$ by \eqref{eq:g kernels}, this proves \eqref{eq:regular kernel}. Applying $(iH)^L$ to a linear relation among $\mathsf V_j,\mathsf U_j$, $0\le j\le L$, makes the top two coefficients vanish. Repeating with lower powers proves linear independence.

	We next construct the compact tests satisfying \eqref{eq:duality}. By \eqref{eq:g kernels}, for $0\le k\le j\le L$,
	\begin{equation*}
		((Hi)^j\td\phi,\mathsf V_k)=((Hi)^j\td\varphi,\mathsf U_k)=\delta_{jk},\qquad
		((Hi)^j\td\phi,\mathsf U_k)=((Hi)^j\td\varphi,\mathsf V_k)=0.
	\end{equation*}
	Starting from \eqref{eq:adapted tests}, we subtract suitable real linear combinations of $\phi_k,\varphi_k$, $k>j$, from $(Hi)^j\td\phi,(Hi)^j\td\varphi$ for $j=L-1,\ldots,0$. The resulting smooth compactly supported radial tests satisfy \eqref{eq:duality}.
\end{proof}

The compact tests in \eqref{eq:duality} give the converse to \eqref{eq:dot calH linear map}.

\begin{lem}[Coercivity for $(iH)^L$]\label{lem:weighted elliptic}
	For $f\in\dot{\calH}^{2L+1}_{\rad}$,
	\begin{equation}\label{eq:weighted reconstruction}
		\|f\|_{\dot{\calH}^{2L+1}}
		\sim \|(iH)^Lf\|_{\dot H^1}+\sum_{j=0}^{L-1}\big(|(\phi_j,f)|+|(\varphi_j,f)|\big).
	\end{equation}
\end{lem}
\begin{proof}
	We claim that for a sufficiently large fixed $R$,
	\begin{equation}\label{eq:elliptic compact error}
		\|f\|_{\dot{\calH}^{2L+1}}\lesssim\|(iH)^Lf\|_{\dot H^1}+C_R\|\chi_{R}f\|_{H^{2L}}.
	\end{equation}
	For an integer $m\ge1$ and smooth compactly supported $g$ vanishing near zero, integration by parts and Cauchy--Schwarz give
	\begin{equation}\label{eq:weighted Hardy}
		\||x|^{-m}g\|_{L^2}^2
		=-\frac{2}{3-2m}(|x|^{-2m}g,x\cdot\nabla g)
		\lesssim \||x|^{1-m}\nabla g\|_{L^2}^2.
	\end{equation}
	Iterating \eqref{eq:weighted Hardy} with $g=(1-\chi_R)\nabla^jf$ and $m=2L+1-j$ for $j=2L,\ldots,0$, we obtain
	\begin{equation}\label{eq:iterated Hardy}
		\|f\|_{\dot{\calH}^{2L+1}}\lesssim\|f\|_{\dot H^{2L+1}}+C_R\|\chi_{R}f\|_{H^{2L}}
	\end{equation}
	with $C_R\to \infty$ as $R\to \infty$.

	We next compare $\|f\|_{\dot H^{2L+1}}$ with $\|(iH)^Lf\|_{\dot H^1}$. The principal part of $\nabla(iH)^Lf$ is $\nabla(-i\Delta)^Lf$, whose $L^2$ norm is equivalent to $\|f\|_{\dot H^{2L+1}}$, and the remaining terms are $a_\alpha(x)\partial^\alpha f$, $|\alpha|\le2L-1$, with
	\begin{equation}\label{eq:operator coefficients}
		|a_\alpha(x)|\lesssim\langle x\rangle^{-(2L+3-|\alpha|)}.
	\end{equation}
	Indeed, each lower-order term contains $k\ge1$ factors of $W^4$. If a total of $j$ derivatives fall on these factors, then $|\alpha|=2L+1-2k-j$ and the coefficient is $O(\langle x\rangle^{-4k-j})$. Since $4k+j=2L+1-|\alpha|+2k\ge2L+3-|\alpha|$, we deduce \eqref{eq:operator coefficients}. Therefore, we obtain
	\begin{equation*}
		\|f\|_{\dot H^{2L+1}}\lesssim \|(iH)^Lf\|_{\dot H^1} +O(R^{-2}\|f\|_{\dot{\calH}^{2L+1}}) +O(\|\chi_{R}f\|_{H^{2L-1}})
	\end{equation*}
	Taking $R$ large in \eqref{eq:iterated Hardy}, we prove \eqref{eq:elliptic compact error}.

	Now, we finish the proof. The bound $(\gtrsim)$ follows from \eqref{eq:dot calH linear map} and the Cauchy--Schwarz inequality.
	For $(\lesssim)$, we argue by contradiction. If it failed, there would exist a sequence $f_n\in\dot{\calH}^{2L+1}_{\rad}$ such that
	\begin{equation*}
		\|f_n\|_{\dot{\calH}^{2L+1}}=1,\qquad
		\|(iH)^Lf_n\|_{\dot H^1}
		+{\textstyle\sum_{j=0}^{L-1}}(|(\phi_j,f_n)|+|(\varphi_j,f_n)|) \to 0.
	\end{equation*}
	For the fixed $R$ in \eqref{eq:elliptic compact error}, $\chi_{R}f_n$ is bounded in $H^{2L+1}$. By the Rellich--Kondrachov compactness theorem, it is Cauchy in $H^{2L}$, up to a subsequence. Applying \eqref{eq:elliptic compact error} to $f_n-f_k$, we obtain
	\begin{equation*}
		\|f_n-f_k\|_{\dot{\calH}^{2L+1}}
		\lesssim \|(iH)^L(f_n-f_k)\|_{\dot H^1}
		+C_R\|\chi_{R}(f_n-f_k)\|_{H^{2L}}
		\to0
	\end{equation*}
	as $n,k\to\infty$. Thus $f_n\to f$ strongly in $\dot{\calH}^{2L+1}$ with $\|f\|_{\dot{\calH}^{2L+1}}=1$. Passing to the limit, we get
	\begin{equation*}
		\nabla(iH)^Lf=0,\qquad
		(\phi_j,f)=(\varphi_j,f)=0\quad(0\le j<L).
	\end{equation*}
	Since $(iH)^Lf\in\dot{\calH}^1$ and nonzero constants do not belong to $\dot{\calH}^1$, we have $(iH)^Lf=0$. Lemma~\ref{lem:dual tests} therefore gives $f\in\calK_L$. Writing $f=\sum_{j=0}^{L-1}(a_j\mathsf V_j+b_j\mathsf U_j)$, \eqref{eq:duality} gives $a_j=(\phi_j,f)=0$ and $b_j=(\varphi_j,f)=0$ for $0\le j<L$. Hence $f=0$, contradicting $\|f\|_{\dot{\calH}^{2L+1}}=1$.
\end{proof}

\section{Global spacetime Morawetz estimate}
In this section, we prove a global Morawetz-type estimate for the linearized NLS flow. Our argument is closely related to the arguments in \cite{ErdoganSchlag2,Schlag2009annals}, which use the absence of embedded eigenvalues to establish linear estimates for the linearized NLS flow. This estimate provides the spacetime control for the higher-order energy.

A key spectral input is the absence of nonzero embedded eigenvalues of $-iH$. Schmid \cite[Proposition~5.4]{Schmid2026CVPDE} previously provided numerical evidence for this property in the radial setting. In our companion work \cite{Kim2026arXivNoEmbeeded}, we establish the following rigorous result.

\begin{prop}[Absence of nonzero embedded eigenvalues, \cite{Kim2026arXivNoEmbeeded}]\label{prop:embedded eigenvalues}
	The operator $-iH$ on $L^2(\mathbb R^3;\mathbb C^2)$, with domain $H^2(\mathbb R^3;\mathbb C^2)$, has no eigenvalues in $i\mathbb R\setminus\{0\}$.
\end{prop}
The main idea is to seek a virial identity for the phase gradient, rather than for the density $|u|^2$ as in the usual NLS virial identity. To this end, we use the Madelung transform $u=\sqrt{\rho}e^{i\theta}$, where $\rho$ is the fluid density and $2\nabla\theta$ is the fluid velocity. Linearizing the resulting fluid system around the soliton and eliminating the density perturbation, we obtain a biharmonic wave equation for the phase gradient $\nabla\theta$. A virial identity for this equation reveals a repulsive structure, which allows us to prove Proposition~\ref{prop:embedded eigenvalues}.

Even after excluding nonzero embedded eigenvalues, the eigenmodes $Y_\pm$, $\Lambda W$, and $iW$ remain an obstruction to the global Morawetz estimate. Indeed, by \eqref{eq:energy kernel} and \eqref{eq:eigenfunctions},
\begin{equation*}
	(\partial_t+iH)(e^{\pm e_0t}Y_\pm)=0,\qquad
	(\partial_t+iH)\Lambda W=(\partial_t+iH)(iW)=0.
\end{equation*}
Since these nonzero solutions do not belong to $L^2_t\mathrm{Mor}_x(\mathbb R)$, we remove the corresponding eigenmodes and apply Kato's smoothing theory \cite{Kato1965} to establish the global Morawetz estimate; see also \cite{LinaresPonce}.

We first define the projection $\Pi$ to remove the eigenmodes. Using the tests in \eqref{eq:base tests} for the scale and phase directions, we set
\begin{equation}\label{eq:projection}
	\begin{aligned}
		\Pi f\coloneqq&f-(-HY_-,f)Y_+-(-HY_+,f)Y_-\\
		&-\bigg[(\td\phi,f)-\sum_{\iota=\pm}(-HY_{-\iota},f)(\td\phi,Y_\iota)\bigg]\Lambda W\\
		&-\frac12\bigg[(\td\varphi,f)-\sum_{\iota=\pm}(-HY_{-\iota},f)(\td\varphi,Y_\iota)\bigg]iW,
		\qquad f\in\dot H^1_{\rad}.
	\end{aligned}
\end{equation}
By \eqref{eq:eigenmode normalization}, \eqref{eq:dual mode tests}, and \eqref{eq:base tests}, we have
\begin{equation}\label{eq:projection kernel}
	\Pi^2=\Pi,\qquad
	\ker\Pi=\operatorname{span}_{\mathbb R}\{Y_+,Y_-,\Lambda W,iW\},
\end{equation}
with range
\begin{equation}
	\Pi\dot H^1_{\rad}
	=\{f\in\dot H^1_{\rad}:(HY_+,f)=(HY_-,f)=(\td\phi,f)=(\td\varphi,f)=0\}.
\end{equation}
Since the tests in \eqref{eq:projection} belong to $L^{6/5}$, $\Pi$ is bounded on $\dot H^1_{\rad}$. Furthermore, for $f\in\Pi\dot H^1_{\rad}$, the boundedness of $\Pi$, together with \eqref{eq:continuous coercivity} and \eqref{eq:projection kernel}, yields
\begin{equation}\label{eq:projection coercivity}
	\begin{aligned}
		\|f\|_{\dot H^1}^2
		&=\inf_{a_1,a_2\in\mathbb R}\|\Pi(f-a_1\Lambda W-a_2iW)\|_{\dot H^1}^2\\
		&\lesssim\inf_{a_1,a_2\in\mathbb R}\|f-a_1\Lambda W-a_2iW\|_{\dot H^1}^2
		\sim(Hf,f)\lesssim\|f\|_{\dot H^1}^2.
	\end{aligned}
\end{equation}
Moreover, using the orthogonality to $HY_\pm$ and $H\Lambda W=H(iW)=0$, we obtain
\begin{equation}\label{eq:projection bounds}
	(Hf,\Pi g)=(Hf,g),\qquad
	\|\Pi g\|_{\mathrm{Mor}}\lesssim\|g\|_{\mathrm{Mor}}.
\end{equation}
For $f_0\in\Pi\dot H^1_{\rad}$, we denote by $U_c(t)f_0$ the solution of
\begin{equation}\label{eq:continuous flow}
	\partial_t f=\Pi(-iH)f,\qquad f(0)=f_0.
\end{equation}
From \eqref{eq:projection bounds}, we have
\begin{equation}\label{eq:energy conservation}
	U_c(t+\sigma)=U_c(t)U_c(\sigma), \qquad (HU_c(t)g,U_c(t)h)=(Hg,h).
\end{equation}
We will prove the following spacetime estimates for $U_c(t)$.

\begin{prop}[Global spacetime estimates]\label{prop:global spacetime}
	The following estimates hold:
	\begin{enumerate}
		\item For $f_0\in\Pi\dot H^1_{\rad}$,
		\begin{equation}\label{eq:weighted smoothing}
			\|U_c(t)f_0\|_{L^2_t\mathrm{Mor}_x(\mathbb R)}
			\lesssim\|f_0\|_{\dot H^1}.
		\end{equation}
		\item Let $I=[t_1,t_2]$, $t_1<t_2\le\infty$, and $F\in L^1_{t,\mathrm{loc}}\Pi\dot H^1_{\rad}(I)$. If a nonnegative $\psi_1\in L^2_t(I)$ satisfies
		\begin{equation}\label{eq:retarded source}
			|(Hg,F(t))|\le\psi_1(t)\|g\|_{\mathrm{Mor}}
			\qquad(g\in\Pi\dot H^1_{\rad})
		\end{equation}
		for almost every $t\in I$, then
		\begin{equation}\label{eq:retarded kernel bound}
			\left\|\int_{t_1}^tU_c(t-\sigma)F(\sigma)\,d\sigma\right\|_{L^2_t\mathrm{Mor}_x(I)}
			\lesssim\|\psi_1\|_{L^2_t(I)}.
		\end{equation}
	\end{enumerate}
\end{prop}

\begin{rem}
	The estimate \eqref{eq:weighted smoothing} can also be obtained using the distorted Fourier transform. See \cite{GermainPusateri2022Pi,KairzhanPusateri2023PAA} and the references therein for related applications of the distorted Fourier transform to linear estimates.
\end{rem}

To prove these estimates, we first establish a resolvent estimate. For this purpose and the subsequent time Fourier transform, we identify each complex-valued scalar function $h$ with the real vector $(\Re h,\Im h)^\top$ and complexify the resulting real vector space. Thus, $f=(f_1,f_2)^\top\in\dot H^1_{\rad}(\mathbb R^3;\mathbb C^2)$ has two independent complex-valued components. We extend the real inner product and the Morawetz norm componentwise:
\begin{equation*}
	(f,g)\coloneqq\Re\int_0^\infty(\ol{f_1}g_1+\ol{f_2}g_2)r^2dr,\qquad
	\|f\|_{\mathrm{Mor}}^2\coloneqq\sum_{j=1}^2\|f_j\|_{\mathrm{Mor}}^2.
\end{equation*}
The Sobolev norms are understood componentwise as well.

We use the same notation for the complex-linear extensions of the operators, with
\begin{equation*}
	H=\begin{pmatrix}H_+&0\\0&H_-\end{pmatrix},\qquad
	-iH=\begin{pmatrix}0&H_-\\-H_+&0\end{pmatrix},
\end{equation*}
where $iH$ is the negative of the displayed matrix operator and $z$ acts by scalar multiplication on both components. The projection $\Pi$ is extended complex-linearly by
\begin{equation*}
	\Pi(f+ig)\coloneqq\Pi f+i\Pi g
	\qquad(f,g\in\dot H^1_{\rad}(\mathbb R^3;\mathbb R^2)).
\end{equation*}
The flow $U_c(t)$ is also extended complex-linearly. The space $\Pi\dot H^1_{\rad}$ denotes the range of the extended projection in $\dot H^1_{\rad}(\mathbb R^3;\mathbb C^2)$. The identities and bounds in \eqref{eq:projection coercivity} and \eqref{eq:projection bounds} extend by applying them to the real and imaginary parts separately. These conventions apply only in this section; from Section~\ref{sec:profiles} onward, we return to the original scalar notation.

\begin{lem}[Resolvent estimates]\label{lem:resolvent esti}
	Let $\Re z>0$, $F\in\Pi\dot H^1_{\rad}$, and $\psi\ge0$. If
	\begin{equation}\label{eq:resolvent source}
		|(Hg,F)|\le\psi\|g\|_{\mathrm{Mor}}\qquad(g\in\Pi\dot H^1_{\rad}),
	\end{equation}
	then the solution $f\in\Pi\dot H^1_{\rad}$ of
	\begin{equation}\label{eq:resolvent equ}
		(z-\Pi(-iH))f=F
	\end{equation}
	satisfies, uniformly in $z$,
	\begin{equation}\label{eq:uniform stationary esti}
		\|f\|_{\mathrm{Mor}}\lesssim\psi.
	\end{equation}
\end{lem}
The main obstruction to a uniform resolvent estimate is the presence of eigenmodes, which allow the resolvent to become arbitrarily large as $z$ approaches an eigenvalue. The projection $\Pi$ removes the known eigenmodes, while Proposition~\ref{prop:embedded eigenvalues} rules out nonzero embedded eigenvalues. In contrast, the essential spectrum is associated with radiation that can propagate to spatial infinity, where the Morawetz weights weaken its contribution. This suggests that uniform resolvent control may still be possible in the Morawetz norm, despite the presence of the essential spectrum.

\begin{proof}[Proof of Lemma~\ref{lem:resolvent esti}]
	We argue by contradiction. If no constant independent of $z$ satisfies \eqref{eq:uniform stationary esti}, then for each $n$ there are $z_n$, $f_n=(f_{1,n},f_{2,n})^\top$, $F_n=(F_{1,n},F_{2,n})^\top$, and $\psi_n$ satisfying the hypotheses with $\|f_n\|_{\mathrm{Mor}}>n\psi_n$. Dividing $f_n$, $F_n$, and $\psi_n$ by $\|f_n\|_{\mathrm{Mor}}$ preserves \eqref{eq:resolvent source} and \eqref{eq:resolvent equ}. We therefore obtain sequences satisfying, for any $g\in\Pi\dot H^1_{\rad}$,
	\begin{equation}\label{eq:resolvent contradiction seq}
		\begin{gathered}
			(z_n-\Pi(-iH))f_n=F_n,\qquad \Re z_n>0,\qquad \|f_n\|_{\mathrm{Mor}}=1,\\
			|(Hg,F_n)|\le\psi_n\|g\|_{\mathrm{Mor}},\qquad 0\le\psi_n<1/n.
		\end{gathered}
	\end{equation}
	By \eqref{eq:projection coercivity}, \eqref{eq:projection bounds}, and \eqref{eq:resolvent contradiction seq} with $g=f_n$, we have
	\begin{equation*}
		0<\Re z_n\lesssim(\Re z_n)(Hf_n,f_n)=(Hf_n,F_n)\le\psi_n\to0.
	\end{equation*}
	Passing to a subsequence, either $z_n\to z_*$ with $z_*\in i\mathbb R$, or $|z_n|\to\infty$. We first prove an exterior estimate uniform in $z$ to obtain a nonzero local limit in the first case. We then exclude this limit using the conditions imposed by $\Pi$ and Proposition~\ref{prop:embedded eigenvalues}. In the second case, direct estimates of the Green function give a contradiction.

	\textbf{Step 1.} We first prove the following exterior estimate for a solution $f$ of \eqref{eq:resolvent equ}, uniformly for $\Re z>0$ and $R\ge1$:
	\begin{equation}\label{eq:compact source bound}
		\|f\|_{\mathrm{Mor}(R,\infty)}\lesssim\psi+R^{-2}\|f\|_{\mathrm{Mor}(R,\infty)}
		+R^{-3/2}|Rf(R)|+R^{-1/2}|\left.\partial_r(rf)\right|_{r=R}|.
	\end{equation}
	We will use this to show that the sequence $f_n$ in \eqref{eq:resolvent contradiction seq} has a nonzero limit.
	
	We write
	\begin{equation*}
		(-iH)f=zf-F+(1-\Pi)(-iH)f\in\dot H^1.
	\end{equation*}
	By $\Pi f=f$ and \eqref{eq:eigenfunctions}, we have $(1-\Pi)(-iH)f\in\operatorname{span}_{\mathbb C}\{\Lambda W,iW\}$. Thus, applying $H$ to \eqref{eq:resolvent equ}, we arrive at
	\begin{equation*}
		H(z+iH)f=HF.
	\end{equation*}
	Using $r\Delta f=\partial_r^2(rf)$, we obtain
	\begin{equation}\label{eq:stationary fourth order}
		\begin{pmatrix}-\partial_r^2-5W^4&0\\0&-\partial_r^2-W^4\end{pmatrix}
		\begin{pmatrix}z&\partial_r^2+W^4\\-\partial_r^2-5W^4&z\end{pmatrix}(rf)=rHF.
	\end{equation}
	To obtain an integral representation using a Green function, we first rewrite \eqref{eq:stationary fourth order} as one-dimensional biharmonic wave equations, instead of computing a Green function for the system directly.
	
	For $\sigma=\pm1$, we write 
	\begin{equation*}
		u_\sigma\coloneqq r(f_1-\sigma if_2),\qquad \kappa^2=\sigma iz,\qquad \Re\kappa>0,\qquad HF=((HF)_1,(HF)_2)^\top.
	\end{equation*}
	Expanding \eqref{eq:stationary fourth order}, we obtain
	\begin{equation}\label{eq:stationary biharmonic equation}
		\begin{aligned}
			(\partial_r^4-\kappa^2\partial_r^2)u_\sigma
			=&\sigma ir((HF)_1-\sigma i(HF)_2)
			+\kappa^2W^4(3u_\sigma+2u_{-\sigma})
			\\
			&-\partial_r^2(W^4(3u_\sigma+2u_{-\sigma}))-W^4\partial_r^2(3u_\sigma-2u_{-\sigma})-5W^8u_\sigma.
		\end{aligned}
	\end{equation}
	Since $(\partial_r^4-\kappa^2\partial_r^2)h=-\partial_r[(-\partial_r^2+\kappa^2)\partial_r h]$, we solve the once-integrated equation for $\partial_r u_\sigma$ on $(R,\infty)$.
	For $R\ge0$, we use the Green's function in \cite[Section~7.5, (7.58)]{Teschl2014Quantum} and odd reflection about $r=R$ to define
	\begin{equation}\label{eq:free Dirichlet resolvent}
		(\calR_{\kappa,R}h)(r)\coloneqq\int_R^\infty K_R(r,\rho)h(\rho)d\rho,\quad K_R(r,\rho)\coloneqq \frac{e^{-\kappa|r-\rho|}-e^{-\kappa(r+\rho-2R)}}{2\kappa}.
	\end{equation}
	Then $(-\partial_r^2+\kappa^2)(\calR_{\kappa,R}h)=h$ and $(\calR_{\kappa,R}h)(R)=0$ for $h\in L^2((R,\infty),dr)$.
	Integrating \eqref{eq:stationary biharmonic equation} once, we obtain, for some constant $C_\sigma$,
	\begin{equation}\label{eq:stationary biharmonic integrate}
		(-\partial_r^2+\kappa^2)\partial_r u_\sigma=C_\sigma+\calI_1+\calI_2+\calI_3,
	\end{equation}
	where
	\begin{gather*}
		\calI_1\coloneqq\sigma i\int_r^\infty\rho((HF)_1-\sigma i(HF)_2)(\rho)d\rho,\quad
		\calI_2\coloneqq\kappa^2\int_r^\infty W^4(3u_\sigma+2u_{-\sigma})d\rho,\\
		\calI_3\coloneqq\partial_r(W^4(3u_\sigma+2u_{-\sigma}))
		-\int_r^\infty(W^4\partial_\rho^2(3u_\sigma-2u_{-\sigma})+5W^8u_\sigma)d\rho.
	\end{gather*}
	Here the general solution of $(-\partial_r^2+\kappa^2)h=C_\sigma$ is
	\begin{equation*}
		h=\kappa^{-2}C_\sigma+a_\sigma e^{-\kappa(r-R)}+b_\sigma e^{\kappa(r-R)}.
	\end{equation*}
	Hence, \eqref{eq:stationary biharmonic integrate} implies
	\begin{equation}\label{eq:integrated wave general}
		\partial_r u_\sigma-\calR_{\kappa,R}(\calI_1+\calI_2+\calI_3)
		=\kappa^{-2}C_\sigma+a_\sigma e^{-\kappa(r-R)}+b_\sigma e^{\kappa(r-R)}.
	\end{equation}
	Since $f,F\in\dot H^1$, Hardy's inequality and integration by parts in $\calI_1$ and $\calI_3$ show that $\partial_r u_\sigma,\calI_1,\calI_2,\calI_3\in L^2((R,\infty),dr)$. Moreover, $\calR_{\kappa,R}$ is bounded on this space since $\Re\kappa>0$. Thus, the left-hand side of \eqref{eq:integrated wave general} belongs to $L^2((R,\infty),dr)$, which implies $C_\sigma=b_\sigma=0$.
	Evaluating at $R$ then gives $a_\sigma=(\partial_r u_\sigma)(R)$. 
	Hence, \eqref{eq:stationary biharmonic integrate} becomes
	\begin{equation}\label{eq:integrated wave}
		\partial_r u_\sigma(r)=(\partial_r u_\sigma)(R)e^{-\kappa(r-R)}+\calR_{\kappa,R}(\calI_1+\calI_2+\calI_3).
	\end{equation}

	We now estimate the terms in \eqref{eq:integrated wave} to prove \eqref{eq:compact source bound}.
	The boundary term satisfies
	\begin{equation*}
		\|r^{-1}(\partial_r u_\sigma)(R)e^{-\kappa(r-R)}\|_{L^2((R,\infty),dr)}\le R^{-1/2}|(\partial_r u_\sigma)(R)|\lesssim R^{-1/2}|\left.\partial_r(rf)\right|_{r=R}|.
	\end{equation*}
	For the integral terms, we have $|K_R(r,\rho)|\le\min\{r-R,\rho-R\}$. From Schur's test with $\int_R^\infty\frac{\rho^{-1/2}}{\max\{r,\rho\}}d\rho\le4r^{-1/2}$, we obtain
	\begin{equation}\label{eq:free weighted resolvent}
		\|r^{-1}\calR_{\kappa,R}h\|_{L^2((R,\infty),dr)}\lesssim \|rh\|_{L^2((R,\infty),dr)}.
	\end{equation}
	We also use the one-dimensional Hardy inequalities
	\begin{equation}\label{eq:exterior Hardy}
		\begin{aligned}
			\|r^{-2}u\|_{L^2((R,\infty),dr)}&\lesssim R^{-3/2}|u(R)|+\|r^{-1}\partial_r u\|_{L^2((R,\infty),dr)},\\
			\left\|r\int_r^\infty h(\rho)d\rho\right\|_{L^2((R,\infty),dr)}&\le\frac23\|r^2h\|_{L^2((R,\infty),dr)}.
		\end{aligned}
	\end{equation}

	To estimate $\calR_{\kappa,R}\calI_1$, we use the following equivalent norm to apply the Riesz representation theorem:
	\begin{equation}\label{eq:half line weights}
		\|f\|_{\mathrm{Mor}}^2\sim\int_0^\infty\left(\frac{|\partial_r(rf)|^2}{\langle r\rangle^2}+\frac{|rf|^2}{\langle r\rangle^4}\right)dr.
	\end{equation}
	By \eqref{eq:projection bounds}, \eqref{eq:resolvent source} holds, up to a fixed constant, for $g\in\dot H^1_{\rad}$. Then, for the map $g\mapsto (Hg,F)$ on the Hilbert space equipped with the norm in \eqref{eq:half line weights}, we apply the Riesz representation theorem to obtain a unique $q_0$ with $\|q_0\|_{\mathrm{Mor}}\lesssim\psi$ such that, for any $g\in\dot H^1_{\rad}$,
	\begin{equation}\label{eq:half line Riesz represent}
		(Hg,F)=\Re\int_0^\infty
		\left[\frac{\ol{\partial_r(rg)}\cdot q_1}{\langle r\rangle}+\frac{\ol{rg}\cdot q_2}{\langle r\rangle^2}\right]dr,
	\end{equation}
	where $q_1\coloneqq\partial_r(rq_0)/\langle r\rangle$ and $q_2\coloneqq rq_0/\langle r\rangle^2$. Hence
	\begin{equation}\label{eq:half line source}
		rHF=-\partial_r\left(\frac{q_1}{\langle r\rangle}\right)+\frac{q_2}{\langle r\rangle^2},\qquad
		\|q_1\|_{L^2(dr)}+\|q_2\|_{L^2(dr)}\lesssim\psi.
	\end{equation}
	Applying \eqref{eq:free weighted resolvent} and the second inequality in \eqref{eq:exterior Hardy} to \eqref{eq:half line source}, we obtain
	\begin{equation}\label{eq:calI1 esti}
		\begin{aligned}
			\|r^{-1}\calR_{\kappa,R}\calI_1\|_{L^2((R,\infty),dr)}\lesssim&\|r\calI_1\|_{L^2((R,\infty),dr)}\\
			\lesssim&\left\|r\left(\frac{q_1}{\langle r\rangle}+\int_r^\infty\frac{q_2(\rho)}{\langle\rho\rangle^2}d\rho\right)\right\|_{L^2((R,\infty),dr)}\\
			\lesssim&\|q_1\|_{L^2(dr)}+\|q_2\|_{L^2(dr)}\lesssim\psi.
		\end{aligned}
	\end{equation}

	For $\calI_2$, with $u=u_\sigma$ or $u_{-\sigma}$, we use
	\begin{equation*}
		\kappa^2\calR_{\kappa,R}\left[\int_r^\infty W^4u d\rho\right]
		=\int_r^\infty W^4u d\rho-e^{-\kappa(r-R)}\int_R^\infty W^4u d\rho
		-\calR_{\kappa,R}[\partial_r(W^4u)].
	\end{equation*}
	Using \eqref{eq:free weighted resolvent} for the last term, $W^4\lesssim r^{-4}$, $|\partial_r W^4|\lesssim r^{-5}$, and Cauchy--Schwarz, we obtain
	\begin{align*}
		&\|r^{-1}\calR_{\kappa,R}\calI_2\|_{L^2((R,\infty),dr)}\\
		&\lesssim R^{-2}(\|\tfrac1{r^2}u\|_{L^2((R,\infty),dr)}+\|\tfrac1r\partial_r u\|_{L^2((R,\infty),dr)})
		\lesssim R^{-2}\|f\|_{\mathrm{Mor}(R,\infty)}.
	\end{align*}

	For $\calI_3$, we have
	\begin{equation}\label{eq:calI3 identity}
		\int_r^\infty W^4(\partial_\rho^2u) d\rho=-W^4\partial_r u-\int_r^\infty(\partial_\rho W^4)(\partial_\rho u)d\rho.
	\end{equation}
	Thus \eqref{eq:free weighted resolvent} and \eqref{eq:exterior Hardy} imply
	\begin{align*}
		&\|r^{-1}\calR_{\kappa,R}\calI_3\|_{L^2((R,\infty),dr)}\\
		&\lesssim\|r\partial_r(W^4u)\|_{L^2((R,\infty),dr)}+\|rW^4\partial_r u\|_{L^2((R,\infty),dr)}\\
		&+\|r^2(\partial_r W^4)(\partial_r u)\|_{L^2((R,\infty),dr)}+\|r^2W^8u\|_{L^2((R,\infty),dr)}
		\lesssim R^{-2}\|f\|_{\mathrm{Mor}(R,\infty)}.
	\end{align*}
	Since $R\ge1$ and $\partial_r f=r^{-1}\partial_r(rf)-f/r$,
	\begin{equation*}
		\|f\|_{\mathrm{Mor}(R,\infty)}\sim\|r^{-1}\partial_r(rf)\|_{L^2((R,\infty),dr)}+\|r^{-1}f\|_{L^2((R,\infty),dr)}.
	\end{equation*}
	Substituting the preceding estimates into \eqref{eq:integrated wave} and summing over $\sigma=\pm1$, we obtain
	\begin{equation*}
		\|r^{-1}\partial_r(rf)\|_{L^2((R,\infty),dr)}
		\lesssim\psi+R^{-1/2}|\left.\partial_r(rf)\right|_{r=R}|
		+R^{-2}\|f\|_{\mathrm{Mor}(R,\infty)}.
	\end{equation*}
	Applying the first inequality in \eqref{eq:exterior Hardy} with $u=rf$ to estimate $\|r^{-1}f\|_{L^2((R,\infty),dr)}$, we conclude \eqref{eq:compact source bound}.

	\textbf{Step 2.} We consider $z_n\to z_*$. By \eqref{eq:projection coercivity} and \eqref{eq:resolvent source} with $g=F$,
	\begin{equation}\label{eq:stationary source energy}
		\|F\|_{\dot H^1}^2\lesssim(HF,F)\le\psi\|F\|_{\mathrm{Mor}}\lesssim\psi\|F\|_{\dot H^1}.
	\end{equation}
	Thus $\|F_n\|_{\dot H^1}\lesssim\psi_n\to0$. Using \eqref{eq:projection} in \eqref{eq:resolvent equ}, we obtain, for every $R_1\ge1$,
	\begin{equation*}
		\begin{gathered}
			\|(1-\Pi)(-iH)f_n\|_{\dot H^1}+\|Hf_n\|_{\mathrm{Mor}}\lesssim1,\\
			-\Delta f_n=Hf_n+W^4\begin{pmatrix}5f_{1,n}\\f_{2,n}\end{pmatrix},\qquad 
			\|f_n\|_{H^2(r<R_1)}\lesssim_{R_1}1.
		\end{gathered}
	\end{equation*}
	Passing to a subsequence, $f_n\to f$ in $H^1_{\mathrm{loc}}$ and $C^1_{\mathrm{loc}}((0,\infty))$. Thus, by \eqref{eq:projection}, $(1-\Pi)(-iH)f_n$ converges in $H^1_{\mathrm{loc}}$. Moreover, since we have
	\begin{equation*}
		rHf_n=\begin{pmatrix}0&-1\\1&0\end{pmatrix}
		\left[z_nrf_n-rF_n+r(1-\Pi)(-iH)f_n\right],
	\end{equation*}
	we get $Hf_n\to Hf$ in $H^1_{\mathrm{loc}}$. Using the preceding convergence, Fatou's lemma, and passing to the limit in \eqref{eq:resolvent equ} with $\|F_n\|_{\dot H^1}\lesssim\psi_n\to0$ and \eqref{eq:compact source bound}, we obtain, for $R$ fixed sufficiently large,
	\begin{equation}\label{eq:stationary limiting sol}
		\begin{gathered}
			\|f\|_{\mathrm{Mor}}\le1,\qquad \Pi f=f,\qquad z_*f=\Pi(-iH)f,\\
			1\lesssim_R\|f\|_{H^1(r<R)}+|Rf(R)|+|\left.\partial_r(rf)\right|_{r=R}|,\qquad f\neq 0.
		\end{gathered}
	\end{equation}
	
	If $z_*=0$, \eqref{eq:eigenfunctions}, \eqref{eq:projection}, and \eqref{eq:stationary limiting sol} imply $(-iH)f=a_1iW+a_2\Lambda W$ for some $a_1,a_2\in\mathbb C$. Integrating this equation with $\|f\|_{\mathrm{Mor}}\le1$, we have
	\begin{align*}
		\partial_r f=\frac{1}{r^2}\int_0^r\rho^2
		\begin{pmatrix}
			a_1W-5W^4f_1\\-a_2\Lambda W-W^4f_2
		\end{pmatrix}d\rho
		=\frac{\sqrt3}4
		\begin{pmatrix}
			2a_1\\a_2
		\end{pmatrix}+O((1+|\log r|)r^{-2}).
	\end{align*}
	Since $\partial_rf_1,\partial_rf_2\in L^2((1,\infty),dr)$ by $\|f\|_{\mathrm{Mor}}\le1$, we have $a_1=a_2=0$. Thus $(-iH)f=0$. By \eqref{eq:green fundamentals}, $f\in\operatorname{span}_{\mathbb C}\{\Lambda W,iW\}$, hence $f=\Pi f=0$ by \eqref{eq:projection kernel} and \eqref{eq:stationary limiting sol}. This contradicts \eqref{eq:stationary limiting sol}.
	
	We now consider $z_*=i\omega$ with $\omega\in\mathbb R\setminus\{0\}$. To derive a contradiction, we will apply Proposition~\ref{prop:embedded eigenvalues} to $(-iH)f$. This requires $Hf\in H^2$. We establish this by analyzing the asymptotic behavior of $Hf$ through the identity
	\begin{equation}\label{eq:stationary Hf identity}
		r((Hf)_1-\sigma i(Hf)_2)=-\partial_r^2u_\sigma-W^4(3u_\sigma+2u_{-\sigma}).
	\end{equation}
	Hence, we first determine the asymptotic behavior of $u_\sigma$.
	We apply \eqref{eq:integrated wave} to $f_n$, with $\kappa_n^2=\sigma iz_n$ and $\Re\kappa_n>0$. To prove the convergence of $\calR_{\kappa,R}(\calI_1+\calI_2+\calI_3)$, we control the part of its defining integral over the exterior region uniformly in $n$.
	For $R_1\ge R$, \eqref{eq:calI1 esti} and Cauchy--Schwarz give
	\begin{equation}\label{eq:calI1 L1 esti}
		\|\calI_1\|_{L^1((R_1,\infty),dr)}\lesssim R_1^{-1/2}\|r\calI_1\|_{L^2((R_1,\infty),dr)}\lesssim R_1^{-1/2}\psi_n \to 0.
	\end{equation}
	For $\calI_2$, we use its definition, change the order of integration, and apply $\sup_n|z_n|<\infty$ and $\|f_n\|_{\mathrm{Mor}}=1$ to obtain
	\begin{equation}\label{eq:calI2 L1 esti}
		\begin{aligned}
			\|\calI_2\|_{L^1((R_1,\infty),dr)}
			&\lesssim |z_n|\int_{R_1}^\infty\int_r^\infty W(\rho)^4|u(\rho)|d\rho dr
			\lesssim\int_{R_1}^\infty r^{-3}|u|dr\\
			&\lesssim R_1^{-1/2}\|f_n\|_{\mathrm{Mor}}
			\lesssim R_1^{-1/2}.
		\end{aligned}
	\end{equation}
	To estimate $\calI_3$, we first use \eqref{eq:calI3 identity} to remove $\partial_r^2u$ and then change the order of integration, using $W^4\lesssim r^{-4}$, $|\partial_r W^4|\lesssim r^{-5}$, and $\|f_n\|_{\mathrm{Mor}}=1$:
	\begin{equation}\label{eq:calI3 L1 esti}
		\begin{aligned}
			\|\calI_3\|_{L^1((R_1,\infty),dr)}
			&\lesssim\int_{R_1}^\infty\big[r^{-5}|u|+r^{-4}|\partial_r u|+(r-R_1)(r^{-5}|\partial_r u|+r^{-8}|u|)\big]dr\\
			&\lesssim R_1^{-5/2}\|f_n\|_{\mathrm{Mor}}
			\lesssim R_1^{-5/2}.
		\end{aligned}
	\end{equation}
	Using \eqref{eq:calI1 L1 esti}, \eqref{eq:calI2 L1 esti}, and \eqref{eq:calI3 L1 esti}, together with $f_n\to f$ in $H^1_{\mathrm{loc}}$ and $C^1_{\mathrm{loc}}((0,\infty))$, we pass to the limit in \eqref{eq:integrated wave} and obtain
	\begin{equation}\label{eq:stationary first order}
		\partial_ru_\sigma=(\partial_ru_\sigma)(R)e^{-\kappa(r-R)}+\calR_{\kappa,R}(\calI_2+\calI_3),
	\end{equation}
	where $u_\sigma=r(f_1-\sigma if_2)$ and $\kappa=\lim_{n\to\infty}\kappa_n$, so $\kappa^2=\sigma iz_*=-\sigma \omega$. Thus, we have
	\begin{equation*}
		\kappa=
		\begin{cases}
			i\operatorname{sgn}(\omega)\sqrt{|\omega|},&\sigma=\operatorname{sgn}(\omega),\\
			\sqrt{|\omega|},&\sigma=-\operatorname{sgn}(\omega).
		\end{cases}
	\end{equation*}
	On the other hand, applying $-\partial_r^2+\kappa^2$ to \eqref{eq:stationary first order}, we get
	\begin{equation}\label{eq:stationary third order}
		\partial_r^3u_\sigma=\kappa^2\partial_ru_\sigma-\calI_2-\calI_3.
	\end{equation}
	
	The bounds \eqref{eq:calI2 L1 esti} and \eqref{eq:calI3 L1 esti}, with $R_1=R$, pass to the limiting function $f$ by local convergence and Fatou's lemma. Thus $\calI_2+\calI_3\in L^1((R,\infty),dr)$. Using \eqref{eq:stationary first order} with $|K_R(r,\rho)|\le |\kappa|^{-1}$ and then integrating the bound on $\partial_ru_\sigma$, we obtain, for $r\ge R$,
	\begin{equation}\label{eq:stationary pointwise}
		\begin{aligned}
			|\partial_ru_\sigma(r)|\le&|(\partial_ru_\sigma)(R)|
			+|\kappa|^{-1}\|\calI_2+\calI_3\|_{L^1((R,\infty),dr)}=O(1),\\
			|u_\sigma(r)|=&O(r).
		\end{aligned}
	\end{equation}
	Using \eqref{eq:stationary pointwise}, we estimate $\calI_2$ directly from its definition, while for $\calI_3$ we first apply \eqref{eq:calI3 identity} to remove the second derivatives:
	\begin{equation}\label{eq:stationary integrals pointwise}
		\begin{aligned}
			|\calI_2(r)|
			\le&|\kappa|^2\int_r^\infty W^4(3|u_\sigma|+2|u_{-\sigma}|)d\rho
			\lesssim|\kappa|^2\int_r^\infty\rho^{-3}d\rho\lesssim r^{-2},\\
			|\calI_3(r)|
			\lesssim&|\partial_rW^4|(|u_\sigma|+|u_{-\sigma}|)+W^4|\partial_ru_\sigma|\\
			&+\int_r^\infty[|\partial_\rho W^4|(|\partial_\rho u_\sigma|+|\partial_\rho u_{-\sigma}|)+W^8|u_\sigma|]d\rho\lesssim r^{-4}.
		\end{aligned}
	\end{equation}
	
	For the case $\kappa=i\operatorname{sgn}(\omega)\sqrt{|\omega|}$, we separate the integrals over $(R,r)$ and $(r,\infty)$ in \eqref{eq:stationary first order}. Collecting the terms proportional to $e^{-\kappa r}$, we obtain, for some $c_\sigma\in\mathbb C$,
	\begin{equation}\label{eq:stationary oscillatory representation}
		\begin{aligned}
			\partial_ru_\sigma
			=&c_\sigma e^{-\kappa r}+\int_r^\infty\frac{e^{-\kappa(\rho-r)}-e^{\kappa(\rho-r)}}{2\kappa}(\calI_2+\calI_3)(\rho)d\rho,\\
			\partial_r^2u_\sigma
			=&-\kappa c_\sigma e^{-\kappa r}+\frac12\int_r^\infty(e^{-\kappa(\rho-r)}+e^{\kappa(\rho-r)})(\calI_2+\calI_3)(\rho)d\rho.
		\end{aligned}
	\end{equation}
	By \eqref{eq:stationary integrals pointwise}, both integrals in \eqref{eq:stationary oscillatory representation} are $O(r^{-1})$.
	
	For the other case $\kappa=\sqrt{|\omega|}$, we use $|K_R(r,\rho)|+|\partial_rK_R(r,\rho)|\lesssim_\kappa e^{-\kappa|r-\rho|}$ in \eqref{eq:stationary first order}. Differentiating once in $r$, splitting the integral at $\rho=r/2$, and applying \eqref{eq:stationary integrals pointwise}, we obtain, for $r\ge2R$,
	\begin{equation}\label{eq:stationary positive root estimate}
		\begin{aligned}
			|\partial_ru_\sigma(r)|+|\partial_r^2u_\sigma(r)|
			\lesssim_\kappa&|(\partial_ru_\sigma)(R)|e^{-\kappa(r-R)}\\
			&+\int_R^\infty e^{-\kappa|r-\rho|}|\calI_2+\calI_3|d\rho=O(r^{-2}).
		\end{aligned}
	\end{equation}
	Thus, by \eqref{eq:stationary third order} and \eqref{eq:stationary integrals pointwise}--\eqref{eq:stationary positive root estimate}, we obtain, for $\sigma=\operatorname{sgn}(\omega)$ and $\kappa=i\sigma\sqrt{|\omega|}$,
	\begin{equation}\label{eq:stationary oscillatory asymptotics}
		\partial_r^ju_\sigma=c_\sigma(-\kappa)^{j-1}e^{-\kappa r}+O(r^{-1}),\qquad 
		\partial_r^ju_{-\sigma}=O(r^{-2})\qquad(1\le j\le3).
	\end{equation}
	By \eqref{eq:stationary Hf identity}, \eqref{eq:stationary pointwise}, and \eqref{eq:stationary oscillatory asymptotics}, we obtain, for some $c\in\mathbb C$ and $j=0,1$,
	\begin{equation}\label{eq:stationary Hf asymp}
		\begin{aligned}
			\partial_r^j(r(Hf)_1)=&c(-i\operatorname{sgn}(\omega)\sqrt{|\omega|})^je^{-i\operatorname{sgn}(\omega)\sqrt{|\omega|}r}+O(r^{-1}),\\
			\partial_r^j(r(Hf)_2)=&i\operatorname{sgn}(\omega)\partial_r^j(r(Hf)_1)+O(r^{-1}).
		\end{aligned}
	\end{equation}
	To exclude the oscillatory term, we use
	\begin{equation}\label{eq:stationary Hf equation}
		(z_*+iH)((-iH)f)=0.
	\end{equation}
	Indeed, by \eqref{eq:energy kernel}, \eqref{eq:eigenfunctions}, and \eqref{eq:projection}, $H(1-\Pi)(-iH)f=0$. Applying $-iH$ to the equation in \eqref{eq:stationary limiting sol}, we obtain \eqref{eq:stationary Hf equation}. Since $z_*=i\omega$, once we establish $Hf\in H^2$, we can apply Proposition~\ref{prop:embedded eigenvalues} to $(-iH)f$ to obtain a contradiction. On the other hand, \eqref{eq:stationary Hf equation} with $z_*=i\omega$ implies
	\begin{align*}
		&\partial_r\Im[r^2(\ol{(Hf)_1}\partial_r(Hf)_1+\ol{(Hf)_2}\partial_r(Hf)_2)]\\
		&=\Im[i\omega r^2(\ol{(Hf)_1}(Hf)_2-\ol{(Hf)_2}(Hf)_1)]=0.
	\end{align*}
	Thus, for some constant $C$,
	\begin{equation}\label{eq:HF cross}
		\Im[r^2(\ol{(Hf)_1}\partial_r(Hf)_1+\ol{(Hf)_2}\partial_r(Hf)_2)]=C.
	\end{equation}
	Since $Hf\in H^1_{\mathrm{loc}}$, its radial equation \eqref{eq:stationary Hf equation} implies $(rHf)(0)=0$ and $\partial_r(rHf)=O(1)$ at the origin. Thus, $C=0$. Taking $r\to\infty$ in \eqref{eq:HF cross} with \eqref{eq:stationary Hf asymp}, we get
	\begin{equation*}
		0=-2\operatorname{sgn}(\omega)\sqrt{|\omega|}|c|^2,\qquad c=0.
	\end{equation*}
	Consequently, by \eqref{eq:stationary Hf asymp}, we have $Hf=O(r^{-2})$ as $r\to\infty$, and hence $Hf\in L^2$. Then \eqref{eq:stationary Hf equation} implies $Hf\in H^2$. Therefore, \eqref{eq:stationary Hf equation} and Proposition~\ref{prop:embedded eigenvalues} yield $(-iH)f=0$, hence $Hf=0$. Now, \eqref{eq:stationary limiting sol} implies $f=0$, contradicting $f\ne0$.

	\textbf{Step 3.} We now consider $|z_n|\to\infty$. We write $h_n\coloneqq(-iH)f_n=(h_{1,n},h_{2,n})^\top$. By \eqref{eq:resolvent contradiction seq} and $H(1-\Pi)(-iH)f_n=0$,
	\begin{equation*}
		(z_n+iH)h_n=(-iH)F_n,\qquad f_n=z_n^{-1}\Pi(F_n+h_n).
	\end{equation*}
	For $\sigma=\pm1$, we choose $\kappa_n^2=\sigma iz_n$ with $\Re\kappa_n>0$. The equation for $h_n$ becomes
	\begin{equation*}
		(-\partial_r^2+\kappa_n^2)[r(h_{1,n}-\sigma ih_{2,n})]
		=rW^4(5h_{1,n}-\sigma ih_{2,n})-r((HF_n)_1-\sigma i(HF_n)_2).
	\end{equation*}
	Since $h_n\in\dot H^1$, we have $(rh_n)(0)=0$ and $h_n/\langle r\rangle\in L^2$. Applying \eqref{eq:free Dirichlet resolvent} with $R=0$, we obtain
	\begin{equation}\label{eq:stationary large z sol}
		\begin{aligned}
			r(h_{1,n}-\sigma ih_{2,n})
			=&\calR_{\kappa_n,0}[rW^4(5h_{1,n}-\sigma ih_{2,n})]\\
			&-\calR_{\kappa_n,0}[r((HF_n)_1-\sigma i(HF_n)_2)].
		\end{aligned}
	\end{equation}
	Here, $h_n/\langle r\rangle\in L^2$ excludes $e^{\kappa_n r}$, while $(rh_n)(0)=0$ excludes $e^{-\kappa_n r}$. We have $|K_0(r,\rho)|\le|\kappa_n|^{-1}$ and $|\partial_rK_0(r,\rho)|,|\partial_\rho K_0(r,\rho)|\le1$. Its mixed derivative is
	\begin{equation*}
		\partial_r\partial_\rho K_0(r,\rho)
		=\delta(r-\rho)-\tfrac12 \kappa_n(e^{-\kappa_n|r-\rho|}+e^{-\kappa_n(r+\rho)}),
	\end{equation*}
	where $\delta$ is the Dirac delta.
	Using \eqref{eq:half line source} with $F=F_n$ and $\psi=\psi_n$, we obtain
	\begin{equation*}
		\calR_{\kappa_n,0}(rHF_n)(r)
		=\int_0^\infty\left[
		\partial_\rho K_0(r,\rho)\frac{q_1(\rho)}{\langle\rho\rangle}
		+K_0(r,\rho)\frac{q_2(\rho)}{\langle\rho\rangle^2}
		\right]d\rho.
	\end{equation*}
	By Cauchy--Schwarz and $\langle r\rangle^{-1}\in L^2(dr)$, we get
	\begin{equation}\label{eq:stationary large z forcing esti}
		\begin{aligned}
			\|\langle r\rangle^{-1}\calR_{\kappa_n,0}(rHF_n)\|_{L^2(dr)}
			&\lesssim\|q_1\|_{L^2(dr)}+|\kappa_n|^{-1}\|q_2\|_{L^2(dr)},\\
			\|\langle r\rangle^{-1}\partial_r\calR_{\kappa_n,0}(rHF_n)\|_{L^2(dr)}
			&\lesssim(1+|\kappa_n|)\|q_1\|_{L^2(dr)}+\|q_2\|_{L^2(dr)}.
		\end{aligned}
	\end{equation}
	For the terms containing $W^4h_n$, the same bounds on $K_0(r,\rho)$, with $\langle r\rangle W^4\in L^2(dr)$, imply
	\begin{equation}\label{eq:stationary large z potential esti}
		\begin{aligned}
			\|\langle r\rangle^{-1}\calR_{\kappa_n,0}(rW^4h_n)\|_{L^2(dr)}
			&\lesssim|\kappa_n|^{-1}\|\langle r\rangle^{-1}rh_n\|_{L^2(dr)},\\
			\|\langle r\rangle^{-1}\partial_r\calR_{\kappa_n,0}(rW^4h_n)\|_{L^2(dr)}
			&\lesssim\|\langle r\rangle^{-1}rh_n\|_{L^2(dr)}.
		\end{aligned}
	\end{equation}
	Using \eqref{eq:half line weights} and \eqref{eq:half line source}, we take norms in \eqref{eq:stationary large z sol} for each $\sigma=\pm1$ and apply \eqref{eq:stationary large z forcing esti} and \eqref{eq:stationary large z potential esti}. Summing the resulting bounds, we obtain, for $|\kappa_n|=|z_n|^{1/2}\ge1$,
	\begin{equation}\label{eq:stationary large z esti}
		\begin{aligned}
			\|\langle r\rangle^{-1}rh_n\|_{L^2(dr)}
			&\lesssim\psi_n+|z_n|^{-1/2}\|\langle r\rangle^{-1}rh_n\|_{L^2(dr)},\\
			\|h_n\|_{\mathrm{Mor}}
			&\lesssim |z_n|^{1/2}\psi_n+\|\langle r\rangle^{-1}rh_n\|_{L^2(dr)}.
		\end{aligned}
	\end{equation}
	All constants in \eqref{eq:stationary large z forcing esti}--\eqref{eq:stationary large z esti} are independent of $n$. Since $|z_n|\to\infty$, the first inequality in \eqref{eq:stationary large z esti} gives $\|\langle r\rangle^{-1}rh_n\|_{L^2(dr)}\lesssim\psi_n$, and the second gives $\|h_n\|_{\mathrm{Mor}}\lesssim |z_n|^{1/2}\psi_n$. Using \eqref{eq:projection bounds} and \eqref{eq:stationary source energy} in $f_n=z_n^{-1}\Pi(F_n+h_n)$, we conclude
	\begin{equation*}
		1=\|f_n\|_{\mathrm{Mor}}=|z_n|^{-1}\|\Pi(F_n+h_n)\|_{\mathrm{Mor}}
		\lesssim |z_n|^{-1/2}\psi_n\to0.
	\end{equation*}
	This contradiction completes the proof of \eqref{eq:uniform stationary esti}.
\end{proof}

\begin{proof}[Proof of Proposition~\ref{prop:global spacetime}]
	We first prove \eqref{eq:retarded kernel bound} by Plancherel's theorem in time, and then deduce \eqref{eq:weighted smoothing} by duality. We initially assume that $F$ has compact support in time. To apply the Fourier transform in time on $\mathbb R$, we set $F(t)=\psi_1(t)=0$ for $t\notin I$. For $\epsilon>0$, we define
	\begin{equation*}
		f_\epsilon(t)\coloneqq\int_{-\infty}^t e^{-\epsilon(t-\sigma)}U_c(t-\sigma)F(\sigma)d\sigma.
	\end{equation*}
	By \eqref{eq:projection coercivity} and \eqref{eq:energy conservation}, $f_\epsilon\in L^2_t\dot H^1_x(\mathbb R)$. We write $\wt{\cdot}$ for the componentwise unitary Fourier transform in time. Using \eqref{eq:continuous flow}, we obtain
	\begin{equation*}
		(\epsilon+i\tau-\Pi(-iH))\wt{f_\epsilon}(\tau)=\wt F(\tau).
	\end{equation*}
	To apply Lemma~\ref{lem:resolvent esti}, we need to check that $\wt F$ satisfies \eqref{eq:resolvent source}.
	By the assumption \eqref{eq:retarded source} on $F$, we can apply the Riesz representation \eqref{eq:half line Riesz represent} to $F$, obtaining $q_1,q_2$ measurable in $t$. By \eqref{eq:half line source}, we have $\|q_1(t)\|_{L^2(dr)}+\|q_2(t)\|_{L^2(dr)}\lesssim\psi_1(t)$. Taking the time Fourier transform in \eqref{eq:half line source} and applying Cauchy--Schwarz together with \eqref{eq:half line weights}, we obtain, for $g\in\Pi\dot H^1_{\rad}$,
	\begin{align*}
		|(Hg,\wt{F}(\tau))|\lesssim
		\|g\|_{\mathrm{Mor}}(\|\wt{q_1}(\tau)\|_{L^2(dr)}+\|\wt{q_2}(\tau)\|_{L^2(dr)}).
	\end{align*}
	We can therefore apply Lemma~\ref{lem:resolvent esti} with $z=\epsilon+i\tau$. The resulting estimate and Plancherel's theorem in time give
	\begin{align*}
		\|\wt{f_\epsilon}(\tau)\|_{\mathrm{Mor}}
		&\lesssim\|\wt{q_1}(\tau)\|_{L^2(dr)}+\|\wt{q_2}(\tau)\|_{L^2(dr)},\\
		\|f_\epsilon\|_{L^2_t\mathrm{Mor}_x(\mathbb R)}^2
		&\lesssim\int_{\mathbb R}(\|\wt{q_1}(\tau)\|_{L^2(dr)}^2+\|\wt{q_2}(\tau)\|_{L^2(dr)}^2)d\tau\\
		&=\int_{\mathbb R}(\|q_1(t)\|_{L^2(dr)}^2+\|q_2(t)\|_{L^2(dr)}^2)dt
		\lesssim\|\psi_1\|_{L^2_t(\mathbb R)}^2.
	\end{align*}
	As $\epsilon\to0$, $f_\epsilon(t)\to\int_{-\infty}^tU_c(t-\sigma)F(\sigma)d\sigma$ in $\dot H^1$ for each $t$, so Fatou's lemma proves \eqref{eq:retarded kernel bound} for compactly supported $F$. If $t_2=\infty$, we apply \eqref{eq:retarded kernel bound} to $\mathbf 1_{[t_1,t_3]}F$. This truncation leaves the Duhamel integral unchanged for $t_1\le t\le t_3$, so letting $t_3\to \infty$ proves the estimate on $I$.

	To prove \eqref{eq:weighted smoothing}, we take compactly supported $h\in L^2_t\Pi\dot H^1_{\rad}(\mathbb R)$. Using Hardy's inequality and \eqref{eq:projection coercivity}, we apply the Riesz representation theorem for the energy inner product to obtain $F(t)\in\Pi\dot H^1_{\rad}$ such that, for any $g\in\Pi\dot H^1_{\rad}$,
	\begin{equation}\label{eq:dual Riesz}
		(h(t),g)_{\mathrm{Mor}}=(Hg,F(t)).
	\end{equation}
	Thus, \eqref{eq:retarded source} holds with $\psi_1(t)=\|h(t)\|_{\mathrm{Mor}}$. By \eqref{eq:energy conservation} and \eqref{eq:dual Riesz}, we have
	\begin{equation*}
		\int_{\mathbb R}(h(t),U_c(t)f_0)_{\mathrm{Mor}}dt
		=(Hq,f_0)\lesssim \|q\|_{\dot H^1}\|f_0\|_{\dot H^1}, \qquad q\coloneqq\int_{\mathbb R}U_c(-t)F(t)dt.
	\end{equation*}
	Expanding $(Hq,q)$ as a double integral and separating the regions $\sigma<t$ and $\sigma>t$, we obtain from \eqref{eq:projection coercivity}, \eqref{eq:energy conservation}, and \eqref{eq:retarded kernel bound}
	\begin{equation*}
		\|q\|_{\dot H^1}^2\lesssim(Hq,q)
		=2\int_{\mathbb R}(h(t),\int_{-\infty}^tU_c(t-\sigma)F(\sigma)d\sigma)_{\mathrm{Mor}}dt
		\lesssim\|h\|_{L^2_t\mathrm{Mor}_x(\mathbb R)}^2.
	\end{equation*}
	Hence,
	\begin{equation*}
		\left|\int_{\mathbb R}(h(t),U_c(t)f_0)_{\mathrm{Mor}}dt\right|\lesssim\|f_0\|_{\dot H^1}\|h\|_{L^2_t\mathrm{Mor}_x(\mathbb R)}.
	\end{equation*}
	Taking $h(t)=\mathbf 1_{[-T,T]}(t)U_c(t)f_0$ and letting $T\to\infty$ proves \eqref{eq:weighted smoothing}.
\end{proof}

\section{Formal modulation law and modified profiles}\label{sec:profiles}
The linearized NLS flow is governed by $\partial_t+iH$, whose dynamics can be separated into the generalized-kernel directions, the eigenmodes, and the radiation part. The generalized kernel relations \eqref{eq:g kernels} suggest introducing modulation parameters $b_k,\eta_k$ to describe the leading-order dynamics along these directions. Following the forward constructions in \cite{RodnianskiSterbenz2010,RaphaelRodnianski2012,MerleRaphaelRodnianski2013Invention,RaphaelSchweyer2013CPAMHeat,RaphaelSchweyer2014APDE,MerleRaphaelRodnianski2015CambJMath}, we seek a modified profile near $W$ and the corresponding formal modulation laws governing the desired quantized blow-up dynamics.

To construct the modified profile, we first determine the formal modulation laws by canceling the terms with the leading spatial growth along the generalized kernel directions. These laws govern the leading profile terms and the desired quantized blow-up dynamics. However, nonlinear interactions produce additional errors with lower spatial growth, which are still too large to close the remainder estimates. We therefore construct the nonlinear corrections $S_n$ recursively through $H^{-1}$ to eliminate these errors. Since these corrections have lower spatial growth than the generalized kernel profiles, they do not affect the leading modulation laws. This procedure, which we call the \emph{tail computation}, yields a modified profile with sufficiently small errors for the subsequent analysis.

We first pass to renormalized variables by defining $(s,y)$ and the renormalized solution $w$ as follows:
\begin{equation}\label{eq:renormal}
	s(t)=s_0+\int_0^t\frac{1}{\lambda(\tau)^{2}}d\tau,\qquad y=\frac r\lambda,\qquad
	w(s,y)=\lambda^{\frac{1}{2}}e^{-i\gamma}u(t,\lambda y)|_{t=t(s)}.
\end{equation}
Then
\begin{equation}\label{eq:w equ}
	\partial_s w-F(w)-\frac{\lambda_s}{\lambda}\Lambda w+i\gamma_s w=0,
	\qquad F(w)=i(\Delta w+|w|^4w).
\end{equation}
We look for a modified profile $W+P$. Its error is given by
\begin{equation}\label{eq:profile error}
	\Psi\coloneqq \partial_sP-F(W+P)-\frac{\lambda_s}{\lambda}\Lambda(W+P)+i\gamma_s(W+P).
\end{equation}
We will choose $P$ so that, after expanding $\Psi$ in the parameters, only the leading tails associated with the generalized kernel directions are retained, and the remaining error is sufficiently small for the subsequent analysis.

To organize this expansion, fix $L\ge1$. For notational convenience, we assign weight $k$ to $b_k,\eta_k$ and write $\{A\}_n$ for the part of a parameter polynomial $A$ of weight $n$:
\begin{equation}\label{eq:weighted homogeneity}
	\{A\}_n(\theta b_1,\theta\eta_1,\ldots,\theta^Lb_L,\theta^L\eta_L)
	=\theta^n\{A\}_n(b_1,\eta_1,\ldots,b_L,\eta_L).
\end{equation}
We first determine the formal evolution of $b_k,\eta_k$ from the generalized kernel directions. By \eqref{eq:g kernels}, the terms along $\Lambda W,iW$ are
\begin{equation}\label{eq:weight one}
	iH(b_1\mathsf V_1+\eta_1\mathsf U_1)-\frac{\lambda_s}{\lambda}\Lambda W+i\gamma_s W=-\left(\frac{\lambda_s}{\lambda}+b_1\right)\Lambda W+\left(\gamma_s-\frac{\eta_1}{2}\right)iW.
\end{equation}
Cancelling these terms gives the first formal modulation laws
\begin{equation}\label{eq:first formal law}
	\frac{\lambda_s}{\lambda}=-b_1,\qquad \gamma_s=\frac{\eta_1}{2}.
\end{equation}
For $1\le k\le L$, under \eqref{eq:first formal law}, the leading $r^{2k-1}$ tails come from
\begin{equation*}
	(\partial_s b_k-b_{k+1})\mathsf V_k+(\partial_s\eta_k-\eta_{k+1})\mathsf U_k+(b_1\Lambda+\tfrac{i}{2}\eta_1)(b_k\mathsf V_k+\eta_k\mathsf U_k).
\end{equation*}
Using \eqref{eq:kernel tails}, vanishing of the $r^{2k-1}$ tail yields the modulation laws
\begin{equation}\label{eq:formal mod laws}
	\begin{aligned}
		\partial_s b_k-b_{k+1}+(2k-\tfrac12)b_1b_k+\tfrac12\eta_1\eta_k&=0,\\
		\partial_s\eta_k-\eta_{k+1}+(2k-\tfrac12)b_1\eta_k-\tfrac12\eta_1b_k&=0,
	\end{aligned}\qquad 1\le k\le L.
\end{equation}
Here $b_{L+1}=\eta_{L+1}=0$. In \eqref{eq:first formal law} and \eqref{eq:formal mod laws}, the coefficients are grouped according to the leading spatial tails of the corresponding profiles. From this viewpoint, $\lambda_s/\lambda,\gamma_s$ can be regarded as having weight $1$, and $\partial_s b_j,\partial_s\eta_j$ as having weight $j+1$.

With the formal laws now fixed, the cancellation of the leading tails leaves only better-decaying errors. We introduce correction profiles $S_n$ to remove these remaining errors and seek
\begin{equation}\label{eq:profile form}
	P=\sum_{n=1}^L P_n,\qquad
	P_n=b_n\mathsf V_n+\eta_n\mathsf U_n+S_n,\qquad S_1=0.
\end{equation}
We write
\begin{equation*}
	\mathbf{Mod}\coloneqq
	\begin{pmatrix}
		\lambda_s/\lambda\\\gamma_s\\
		\partial_s b_1\\\partial_s\eta_1\\
		\vdots\\
		\partial_s b_L\\\partial_s\eta_L
	\end{pmatrix}
	+\mathbf{Mod}_0,
	\qquad
	\mathbf{Mod}_0\coloneqq
	\begin{pmatrix}
		b_1\\ -\frac12\eta_1\\
		-b_2+\frac32 b_1^2+\frac12\eta_1^2\\
		-\eta_2+b_1\eta_1\\
		\vdots\\
		(2L-\frac12)b_1b_L+\frac12\eta_1\eta_L\\
		(2L-\frac12)b_1\eta_L-\frac12\eta_1b_L
	\end{pmatrix},
\end{equation*}
and, for $1\le j\le L$,
\begin{equation*}
	\begin{aligned}
		\mathbf{Mod}_\lambda\coloneqq&\frac{\lambda_s}{\lambda}+b_1,\\
		\mathbf{Mod}_\gamma\coloneqq&\gamma_s-\frac{\eta_1}{2},\\
		\mathbf{Mod}_{b_j}\coloneqq&\partial_s b_j-b_{j+1}\\
		&+(2j-\tfrac12)b_1b_j+\tfrac12\eta_1\eta_j,\\
		\mathbf{Mod}_{\eta_j}\coloneqq&\partial_s\eta_j-\eta_{j+1}\\
		&+(2j-\tfrac12)b_1\eta_j-\tfrac12\eta_1b_j,
	\end{aligned}
	\quad
	\textbf{v}\coloneqq
	\begin{pmatrix}
		\Lambda(W+P)\\
		-i(W+P)\\
		-\partial_{b_1}P\\
		-\partial_{\eta_1}P\\
		\vdots\\
		-\partial_{b_L}P\\
		-\partial_{\eta_L}P
	\end{pmatrix},
	\,\,
	\textbf{v}_n\coloneqq
	\begin{pmatrix}
		\Lambda P_n\\
		-iP_n\\
		-\partial_{b_1}P_n\\
		-\partial_{\eta_1}P_n\\
		\vdots\\
		-\partial_{b_L}P_n\\
		-\partial_{\eta_L}P_n
	\end{pmatrix}.
\end{equation*}

To construct the corrections, write
\begin{equation}\label{eq:nonlinearity}
	\calN_W(\td\eps)=F(W+\td\eps)+iH\td\eps
	=i\{|W+\td\eps|^4(W+\td\eps)-W^5-W^4(3\td\eps+2\ol{\td\eps})\}.
\end{equation}
Then, \eqref{eq:profile error} becomes
\begin{equation*}
	\Psi=-\mathbf{Mod}\cdot\textbf{v}
	+\mathbf{Mod}_0\cdot\textbf{v}+iHP-\calN_W(P).
\end{equation*}
We retain the modulation term $-\mathbf{Mod}\cdot\textbf{v}$ and choose $S_n$ to cancel the remaining terms successively in weight. The weight-$1$ terms cancel by \eqref{eq:weight one}. For $2\le n\le L$, suppose that $S_2,\ldots,S_{n-1}$ have been constructed. The weight-$n$ error to be cancelled is
\begin{equation}
	\{\Psi+\mathbf{Mod}\cdot\textbf{v}\}_n=\Theta_n+iHS_n,
\end{equation}
where
\begin{equation}\label{eq:theta}
	\Theta_n\coloneqq\mathbf{Mod}_0\cdot\textbf{v}_{n-1}
	+iH(b_n\mathsf V_n+\eta_n\mathsf U_n)
	-\bigg\{\calN_W\bigg(\sum_{j=1}^{n-1}P_j\bigg)\bigg\}_n.
\end{equation}
Since $\calN_W$ has no linear term and the $b_n,\eta_n$ terms cancel by \eqref{eq:g kernels}, $\Theta_n$ is determined by the previously constructed profiles and lower-index parameters. To cancel the weight-$n$ error, we set
\begin{equation}\label{eq:Sn equation}
	S_n\coloneqq H^{-1}(i\Theta_n),\qquad -iHS_n=\Theta_n.
\end{equation}
After iterating up to $n=L$, define $\Theta_{L+1}$ by \eqref{eq:theta} with $n=L+1$ and $b_{L+1}=\eta_{L+1}=0$,
\begin{equation}\label{eq:Theta L1 def}
	\Theta_{L+1}\coloneqq
	\{\Psi+\mathbf{Mod}\cdot\textbf{v}\}_{L+1}
	=\mathbf{Mod}_0\cdot\textbf{v}_{L}
	-\bigg\{\calN_W\bigg(\sum_{j=1}^{L}P_j\bigg)\bigg\}_{L+1}.
\end{equation}
Collecting the remaining weights gives
\begin{equation}\label{eq:residual expansion}
	\Psi=-\mathbf{Mod}\cdot\textbf{v}+\td\Psi,
\end{equation}
where $\Theta_n$ is the weight-$n$ error,
\begin{equation}\label{eq:Psi td def}
	\td\Psi\coloneqq\Theta_{L+1}+\sum_{n=L+2}^{5L}\{\Psi\}_n,\qquad
	\{\Psi\}_n=-\{\calN_W(P)\}_n\quad(L+2\le n\le5L).
\end{equation}

We now record the resulting properties of the modified profile and its error.

\begin{prop}[Modified profiles]\label{prop:finite profile}
	For $2\le n\le L$, the corrections $S_n$ defined by \eqref{eq:theta} and \eqref{eq:Sn equation} are polynomials of weight $n$, depending only on $b_j,\eta_j$ with $j<n$, and \eqref{eq:profile form} satisfies
	\begin{equation}\label{eq:profile cancellation}
		\{\Psi\}_n=-\{\mathbf{Mod}\cdot\textbf{v}\}_n\quad(0\le n\le L),\qquad
		\{\Psi\}_n=0\quad(n>5L).
	\end{equation}
	The radial profiles $f$ in $S_n$ and $\td\Psi$ are smooth and regular, satisfying, for every integer $m\ge0$,
	\begin{equation}\label{eq:correction pointwise}
		|f|_m\lesssim_m
		\begin{cases}
			\begin{aligned}
				r^2\langle r\rangle^{2n-4}
				&\quad \text{in}\quad S_n,
				&& 2\le n\le L,\\
				r^2\langle r\rangle^{2L-4}
				&\quad \text{in}\quad \Theta_{L+1},
				&\\
				r^4\langle r\rangle^{2n-9}
				&\quad \text{in}\quad \{\Psi\}_n,
				&& L+2\le n\le 5L.
			\end{aligned}
		\end{cases}
	\end{equation}
	Their adapted derivatives satisfy
	\begin{equation}\label{eq:lin profile decay}
		\begin{aligned}
			|\nabla^m((-iH)^nf)|&\lesssim_m\langle r\rangle^{-3-m}
			&&\text{in}\quad S_n, &&2\le n\le L,\\
			|\nabla^m((-iH)^Lf)|&\lesssim_m\langle r\rangle^{-3-m}
			&&\text{in} \quad \Theta_{L+1}.
		\end{aligned}
	\end{equation}
	Every profile $f$ in $\td\Psi$ also satisfies
	\begin{equation}\label{eq:weighted residual}
		\sum_{j=0}^m\|\langle x\rangle^{-10L+3+j}\nabla^jf\|_{L^2}
		+\|\langle x\rangle^{-10L+3}f\|_{H^m}\lesssim_m 1.
	\end{equation}
\end{prop}
\begin{proof}
	For convenience, we omit the coefficients involving $b_j,\eta_j$ in the spatial estimates. We first prove the first bounds in \eqref{eq:correction pointwise} and \eqref{eq:lin profile decay} for $S_n$, $2\le n\le L$, simultaneously by induction on $n$. Suppose that these bounds hold up to $S_n$ for some $1\le n<L$, with $S_1=0$.
	
	Substituting $P_n=b_n\mathsf V_n+\eta_n\mathsf U_n+S_n$ into \eqref{eq:theta} and using \eqref{eq:g kernels}, we obtain
	\begin{equation}
		\begin{aligned}
			\Theta_{n+1}=&b_1(\Lambda-2n+\tfrac12)(b_n\mathsf V_n+\eta_n\mathsf U_n)
			+\tfrac{\eta_1}2[b_n(i\mathsf V_n+\mathsf U_n)+\eta_n(i\mathsf U_n-\mathsf V_n)]\\
			&+(b_1\Lambda+\tfrac i2\eta_1)S_n
			-\sum_{j=1}^{n-1}\big[(\mathbf{Mod}_0)_{b_j}\partial_{b_j}S_n
			+(\mathbf{Mod}_0)_{\eta_j}\partial_{\eta_j}S_n\big]\\
			&-\bigg\{\calN_W\bigg(\sum_{j\le n}P_j\bigg)\bigg\}_{n+1}.
		\end{aligned}\label{eq:linear profile cancellation}
	\end{equation}
	By \eqref{eq:tail cancellations}, the first line has $|\cdot|_m$ size bounded by $r^{2n}\langle r\rangle^{-2}\le r^2\langle r\rangle^{2n-4}$. The second line involving $S_n$ satisfies the same bound by the first bound in \eqref{eq:correction pointwise} at order $n$. For the nonlinear term, \eqref{eq:kernel growth} and the induction hypothesis give
	\begin{equation}\label{eq:nonlinear global weight}
		\bigg|\bigg\{\calN_W\bigg(\sum_{j\le n}P_j\bigg)\bigg\}_{n+1}\bigg|_m
		\lesssim_m\sum_{d=2}^5r^{2d}\langle r\rangle^{2n-2d-3}
		\lesssim_m r^4\langle r\rangle^{2n-7}.
	\end{equation}
	Thus $|\Theta_{n+1}|_m\lesssim_m r^2\langle r\rangle^{2n-4}\lesssim_m\langle r\rangle^{2n-2}$. By \eqref{eq:Sn equation} and \eqref{eq:inverse growth}, we obtain
	\begin{equation*}
		|S_{n+1}|_m\lesssim_m r^2\langle r\rangle^{2n-2},
	\end{equation*}
	which proves the first bound in \eqref{eq:correction pointwise} at order $n+1$. Equation \eqref{eq:linear profile cancellation} also shows that $\Theta_{n+1}$, and hence $S_{n+1}$, has weight $n+1$ and depends only on parameters with indices at most $n$.
	
	We next prove the first bound in \eqref{eq:lin profile decay} at order $n+1$. Since $(-iH)^{n+1}S_{n+1}=(-iH)^n\Theta_{n+1}$, it suffices to estimate $(-iH)^n\Theta_{n+1}$. Applying $(-iH)^n$ to the first two lines of \eqref{eq:linear profile cancellation} and using \eqref{eq:g kernels}, we have
	\begin{equation*}
		\begin{aligned}
			(-iH)^n(\Lambda-2n+\tfrac12)\mathsf V_n
			&=(\Lambda+\tfrac12)\Lambda W
			+\big([(-iH)^n,\Lambda]-2n(-iH)^n\big)\mathsf V_n,\\
			(-iH)^n(\Lambda-2n+\tfrac12)\mathsf U_n
			&=\tfrac i2(\Lambda+\tfrac12)W
			+\big([(-iH)^n,\Lambda]-2n(-iH)^n\big)\mathsf U_n,\\
			(-iH)^n(i\mathsf V_n+\mathsf U_n)
			&=i(\Lambda+\tfrac12)W+[(-iH)^n,i]\mathsf V_n,\\
			(-iH)^n(i\mathsf U_n-\mathsf V_n)
			&=-(\Lambda+\tfrac12)W+[(-iH)^n,i]\mathsf U_n.
		\end{aligned}
	\end{equation*}
	We use the identities
	\begin{equation}\label{eq:ground state cancellations}
		(\Lambda+\tfrac12)W=W^3,\qquad
		(\Lambda+\tfrac12)\Lambda W=3W^5-\tfrac52 W^3.
	\end{equation}
	The first terms are $O(\langle r\rangle^{-3-m})$ after $m$ derivatives by \eqref{eq:ground state cancellations}, and the same bound holds for the commutator terms by \eqref{eq:lin commutator coefficients} and \eqref{eq:kernel growth}.
	
	For the terms involving $S_n$, the induction hypothesis applies to $(-iH)^n\partial_{b_j}S_n$ and $(-iH)^n\partial_{\eta_j}S_n$. Moreover,
	\begin{equation*}
		\begin{aligned}
			(-iH)^n\Lambda S_n
			&=(\Lambda+2n)(-iH)^nS_n
			+\big([(-iH)^n,\Lambda]-2n(-iH)^n\big)S_n,\\
			(-iH)^n(iS_n)
			&=i(-iH)^nS_n+[(-iH)^n,i]S_n.
		\end{aligned}
	\end{equation*}
	The first terms satisfy the required bound by induction. By \eqref{eq:lin commutator coefficients} and the first bound in \eqref{eq:correction pointwise}, the commutator terms are $O(\langle r\rangle^{-4-m})$ after $m$ derivatives. Finally, \eqref{eq:lin iterated products} and \eqref{eq:nonlinear global weight} give
	\begin{equation*}
		\bigg|\nabla^m(-iH)^n
		\bigg\{\calN_W\bigg(\sum_{j\le n}P_j\bigg)\bigg\}_{n+1}\bigg|
		\lesssim_m\sum_{k=0}^{2n+m}
		\frac{\langle r\rangle^{2n-3-k}}{\langle r\rangle^{2n+m-k}}
		\lesssim_m\langle r\rangle^{-3-m}.
	\end{equation*}
	This proves the first bound in \eqref{eq:lin profile decay} at order $n+1$ and completes the induction.
	
	The preceding estimates for $\Theta_{n+1}$ with $n=L$ prove the second bounds in \eqref{eq:correction pointwise} and \eqref{eq:lin profile decay}, including $L=1$ since $S_1=0$. For $L+2\le n\le5L$, \eqref{eq:residual expansion},
	\eqref{eq:kernel growth}, and the first bound in
	\eqref{eq:correction pointwise} give
	\begin{equation*}
		|\{\Psi\}_n|_m=|\{\calN_W(P)\}_n|_m
		\lesssim_m\sum_{d=2}^5r^{2d}\langle r\rangle^{2n-2d-5}
		\lesssim_m r^4\langle r\rangle^{2n-9},
	\end{equation*}
	which proves the last bound in \eqref{eq:correction pointwise}.	The profiles are smooth and regular by \eqref{eq:theta}, \eqref{eq:Sn equation}, and Lemma~\ref{lem:inverse}. Finally, \eqref{eq:residual expansion} and \eqref{eq:correction pointwise} give \eqref{eq:profile cancellation} and \eqref{eq:weighted residual}, respectively.
\end{proof}
As a consequence of this proposition, we have $P_j\in\dot{\calH}^{2L+1}$ for $j\le L$ by \eqref{eq:correction pointwise} and Lemma~\ref{lem:dot calH pointwise}.

To obtain the quantized blow-up rate $\lambda(t)\sim(T-t)^{2L}$, we first identify a special solution of the formal ODE system \eqref{eq:formal mod laws} that generates this rate. We will later construct solutions whose modulation parameters remain close to this special solution by a topological shooting argument. For this purpose, we linearize \eqref{eq:formal mod laws} around the special solution and separate the stable and unstable directions. Since \eqref{eq:formal mod laws} coincides with the modulation system in \cite{JeongKim2024arXiv}, we recall the corresponding special solution and its linearization. We observe that the following solution
\begin{equation}\label{eq:special}
	\begin{gathered}
		b_k^e(s)\coloneqq\frac{c_k}{s^k},\qquad \eta_k^e(s)\coloneqq0,
		\\
		c_1=\frac{2L}{4L-1},\qquad
		c_{k+1}=-\frac{L-k}{4L-1}c_k,
	\end{gathered}\qquad 1\le k\le L
\end{equation}
solves \eqref{eq:formal mod laws}. With $\lambda_s/\lambda=-b_1^e$ and $dt/ds=\lambda^2$, integration gives
\begin{equation*}
	\lambda(s)\sim s^{-\frac{2L}{4L-1}},\qquad
	T-t(s)\sim s^{-\frac{1}{4L-1}},\qquad
	\lambda(t)\sim(T-t)^{2L}.
\end{equation*}
We measure the deviation from \eqref{eq:special} by
\begin{equation}\label{eq:fluctuations}
	\mathfrak A_b\coloneqq \left(s^k(b_k-b_k^e)\right)_{k=1}^L,\qquad
	\mathfrak A_\eta\coloneqq \left(s^k\eta_k\right)_{k=1}^L.
\end{equation}
\begin{lem}[Linearized modulation system]
	There are fixed real invertible matrices $\mathsf S_b,\mathsf S_\eta$ such that, for a solution of \eqref{eq:formal mod laws}, the transformed fluctuations
	\begin{equation}\label{eq:eigen coordinates}
		\calA_b\coloneqq\mathsf S_b\mathfrak A_b,\qquad
		\calA_\eta\coloneqq\mathsf S_\eta\mathfrak A_\eta
	\end{equation}
	satisfy
	\begin{equation}\label{eq:linear maps}
		s\partial_s \calA_b
		=\mathsf D_b \calA_b+O(|\calA_b|^2+|\calA_\eta|^2),\quad
		s\partial_s \calA_\eta
		=\mathsf D_\eta \calA_\eta+O(|\calA_b|^2+|\calA_\eta|^2),
	\end{equation}
	where
	\begin{equation}
		\mathsf D_b=\operatorname{diag}\left(-1,\left(\frac{k}{4L-1}\right)_{k=2}^L\right), \qquad
		\mathsf D_\eta=\operatorname{diag}\left(\left(\frac{k}{4L-1}\right)_{k=1}^L\right).
		\label{eq:ODE diagonal}
	\end{equation}
\end{lem}
\begin{proof}
	Substituting \eqref{eq:fluctuations} into \eqref{eq:formal mod laws} gives
	\begin{equation*}
		\begin{aligned}
			s\partial_s(\mathfrak A_b)_k
			&=\frac{L-k}{4L-1}(\mathfrak A_b)_k+(\mathfrak A_b)_{k+1}
			-(2k-\tfrac12)c_k(\mathfrak A_b)_1+O(|\mathfrak A_b|^2+|\mathfrak A_\eta|^2),\\
			s\partial_s(\mathfrak A_\eta)_k
			&=\frac{L-k}{4L-1}(\mathfrak A_\eta)_k+(\mathfrak A_\eta)_{k+1}
			+\tfrac12c_k(\mathfrak A_\eta)_1+O(|\mathfrak A_b|^2+|\mathfrak A_\eta|^2),
		\end{aligned}
	\end{equation*}
	with $(\mathfrak A_b)_{L+1}=(\mathfrak A_\eta)_{L+1}=0$. These are the two linearizations in \cite{JeongKim2024arXiv}. Following the diagonalization in that work, there exist invertible matrices $\mathsf S_b,\mathsf S_\eta$ which diagonalize their linear parts with \eqref{eq:ODE diagonal}. Multiplying the two systems by these matrices gives \eqref{eq:linear maps}.
\end{proof}
Thus $(\calA_b)_1$ corresponds to the only stable direction. The coordinates $\left((\calA_b)_j\right)_{j=2}^L$ and $\calA_\eta$ give $2L-1$ unstable directions, as in \cite{JeongKim2024arXiv}. We use these diagonal coordinates in the bootstrap below; the unstable eigenmode $Y_+$ is handled separately.

\section{Trapped solutions}

In this section, we construct solutions realizing the formal blow-up dynamics described in the previous section. We first specify the initial data and fix the decomposition into the modified profile and a remainder by orthogonality conditions. We then introduce a trapped regime in which the modulation parameters remain close to the special solution \eqref{eq:special} and the remainder is sufficiently small to preserve the leading modulation laws. The bounds for the remainder and the stable modes are improved using the modulation estimates and the global Morawetz estimate of Proposition~\ref{prop:global spacetime}. The unstable modes are instead controlled by choosing their initial values through a topological argument based on Brouwer's fixed point theorem. This choice involves the $2L-1$ unstable directions in \eqref{eq:ODE diagonal} together with the unstable direction corresponding to $Y_+$.

\subsection{Decomposition}\label{sec:decomposition}

Fix $L\ge1$. We first decompose solutions near the soliton orbit. For $\delta_{\mathrm{dec}}>0$ sufficiently small, set
\begin{equation}\label{eq:decom tube}
	d_W(u)\coloneqq\inf_{\lambda>0, \gamma\in\mathbb R}\|u-W_{\lambda,\gamma}\|_{\dot H^1},\qquad
	\calO_{\mathrm{dec}}\coloneqq\{u\in\dot H^1_{\rad}:d_W(u)<\delta_{\mathrm{dec}}\}.
\end{equation}
We write the renormalized solution as
\begin{equation}\label{eq:w decom}
	w=W+\td\eps=W+P+\eps,
\end{equation}
where $P$ is the modified profile from \eqref{eq:profile form} and Proposition~\ref{prop:finite profile}. The profile $P$ grows at infinity, so $\eps=\td\eps-P$ need not belong to $\dot H^1$. We therefore keep $\td\eps$ in the energy space and first define $\eps$ in the weighted space $\langle r\rangle^{2L+1}L^2$. For $H^{2L+1}$ solutions, this decomposition yields $\eps\in\dot{\calH}^{2L+1}$, which is the space used in the higher-order analysis. 

We define $\calZ^\perp$ by orthogonality to $\calZ$, and $\calZ_{\mathrm h}^\perp$ by imposing two additional orthogonality conditions:
\begin{equation}
	\begin{gathered}
		\calZ\coloneqq\operatorname{span}_{\mathbb R}\{\phi_j,\varphi_j:0\le j\le L\},\\
		\calZ^\perp\coloneqq\{f\in L^2_{\mathrm{loc}}(\mathbb R^3):(\phi_j,f)=(\varphi_j,f)=0\quad(0\le j\le L)\},\\
		\calZ_{\mathrm h}^\perp\coloneqq\{f\in\langle r\rangle^{2L+1}L^2\cap\calZ^\perp:(-HY_-,f)=(-HY_+,f)=0\}.
	\end{gathered}
\end{equation}
The orthogonality conditions fix the modulation parameters and yield the following decomposition.

\begin{lem}[Decomposition]\label{lem:decomposition}
	There exist $0<\delta_{\mathrm{dec}}\ll\delta_p\ll1$ and functions $\psi_+,\psi_-\in C^\infty_{c,\rad}$ such that the following conditions hold. With $\psi_\pm$ fixed, the decomposition maps are unique.
	\begin{enumerate}
		\item[(1)] (Decomposition on $\dot H^1_{\rad}$) The map $(\lambda,\gamma,\td\eps):\calO_{\mathrm{dec}}\to\mathbb R_+\times\mathbb R/2\pi\mathbb Z\times\dot H^1_{\rad}$ satisfies
		\begin{equation}\label{eq:decomposition 1}
			u=[W+\td\eps]_{\lambda,\gamma},\qquad (\phi_0,\eps)=(\varphi_0,\eps)=0,\qquad \|\td\eps\|_{\dot H^1}<\delta_p.
		\end{equation}
		
		\item[(2)] (Decomposition on $\langle r\rangle^{2L+1}L^2_{\rad}$) With the same $(\lambda,\gamma,\td\eps)$ as in (1), the map $(b_1,\eta_1,\ldots,b_L,\eta_L,\eps):\calO_{\mathrm{dec}}\to\mathbb R^{2L}\times(\langle r\rangle^{2L+1}L^2_{\rad}\cap\calZ^\perp)$ satisfies
		\begin{equation}\label{eq:decomposition 2}
			\begin{gathered}
				u=[W+P+\eps]_{\lambda,\gamma},\qquad P+\eps=\td\eps,\qquad \sum_{j=1}^L(|b_j|+|\eta_j|)+\|\td\eps\|_{\dot H^1}<\delta_p,\\
				(\phi_j,\eps)=(\varphi_j,\eps)=0, \qquad 0\leq j\leq L.
			\end{gathered}
		\end{equation}
		
		\item[(3)] (Decomposition along $\psi_\pm$) The remainder $\eps$ from (2) satisfies
		\begin{equation}\label{eq:eigenmode decom}
			\eps-e_0^{-L}\alpha_+\psi_+-(-e_0)^{-L}\alpha_-\psi_-\in\calZ_{\mathrm h}^\perp,
		\end{equation}
		where
		\begin{equation}\label{eq:eigenmode mode def}
			\alpha_\pm\coloneqq(\pm e_0)^L(-HY_\mp,\eps).
		\end{equation}
		
		\item[(4)] ($C^1$-regularity) The parameter map $u\mapsto(\lambda,\gamma,b_1,\eta_1,\ldots,b_L,\eta_L,\alpha_+,\alpha_-)$ is $C^1$ on $\calO_{\mathrm{dec}}$.
	\end{enumerate}
\end{lem}
For $u\in H^{2L+1}_{\rad}\cap\calO_{\mathrm{dec}}$, we have $\td\eps\in\dot H^1_{\rad}\cap\dot H^{2L+1}_{\rad}$ and $\eps\in\dot{\calH}^{2L+1}_{\rad}$, since $W\in \dot H^1_{\rad}\cap\dot H^{2L+1}_{\rad}$ and $P\in \dot{\calH}^{2L+1}_{\rad}$. 
\begin{proof}[Proof of Lemma~\ref{lem:decomposition}]
	For $u\in\calO_{\mathrm{dec}}$, choose $(\mu,\theta)$ such that
	\begin{equation*}
		\|u_{\mu^{-1},-\theta}-W\|_{\dot H^1}
		=\|u-W_{\mu,\theta}\|_{\dot H^1}<\delta_{\mathrm{dec}}.
	\end{equation*}
	Set $\td\eps=(u_{\mu^{-1},-\theta})_{\lambda^{-1},-\gamma}-W$ and $\eps=\td\eps-P$, and define
	\begin{equation*}
		\calG(u_{\mu^{-1},-\theta},\lambda,\gamma,b_1,\eta_1,\ldots,b_L,\eta_L)
		\coloneqq\left((\phi_j,\eps),(\varphi_j,\eps)\right)_{j=0}^L.
	\end{equation*}
	Since $S_n$ has no linear terms, $\calG(W,1,0,\ldots,0)=0$ and \eqref{eq:duality} gives
	\begin{equation}
		\partial_{(\log\lambda,\gamma,b_1,\eta_1,\ldots,b_L,\eta_L)}
		\calG(W,1,0,\ldots,0)
		=\operatorname{diag}(1,-2,-1,\ldots,-1).
	\end{equation}
	By the implicit function theorem, there exist $0<\delta_p\ll\delta_1\ll1$ such that, for $\|u_{\mu^{-1},-\theta}-W\|_{\dot H^1}<\delta_p/100$, the equation $\calG(u_{\mu^{-1},-\theta},\cdot)=0$ has a unique solution in the range
	\begin{equation}\label{eq:parameter range}
		|\log\lambda|+|\gamma|+{\textstyle\sum_{j=1}^L}(|b_j|+|\eta_j|)<\delta_1.
	\end{equation}
	For fixed $(\mu,\theta)$, the parameters depend $C^1$ on $u$. From IFT and $\|\td\eps\|_{\dot H^1}=\|u_{\mu^{-1},-\theta}-W_{\lambda,\gamma}\|_{\dot H^1}$, we have
	\begin{equation*}
		|\log\lambda|+|\gamma|+{\textstyle\sum_{j=1}^L}(|b_j|+|\eta_j|)
		+\|\td\eps\|_{\dot H^1}
		\le C\|u_{\mu^{-1},-\theta}-W\|_{\dot H^1}<C\delta_{\mathrm{dec}}.
	\end{equation*}
	Choosing $\delta_{\mathrm{dec}}\ll\delta_p$, we obtain
	\begin{equation*}
		u=[W+\td\eps]_{\mu\lambda,\theta+\gamma}
		=[W+P+\eps]_{\mu\lambda,\theta+\gamma}
	\end{equation*}
	with the required smallness. Moreover, $\td\eps,P\in\langle r\rangle^{2L+1}L^2$ by Hardy's inequality, \eqref{eq:kernel growth}, and \eqref{eq:correction pointwise}, so $\eps\in\langle r\rangle^{2L+1}L^2$.
	
	Now, we show the uniqueness of the decomposition.
	Suppose that $u=[W+\td\eps']_{\lambda',\gamma'}=[W+P'+\eps']_{\lambda',\gamma'}$ is another decomposition satisfying \eqref{eq:decomposition 1} and \eqref{eq:decomposition 2}, with parameters $b_j',\eta_j'$. Here, we choose $\gamma'$ such that $\gamma'-\theta\in[-\pi,\pi]$. The explicit formula for $W$ gives
	\begin{equation*}
		\begin{aligned}
			\min\{1,|\log(\lambda'/\mu)|+|\gamma'-\theta|\}
			&\lesssim\|W_{\lambda'/\mu,\gamma'-\theta}-W\|_{\dot H^1}\\
			&\le\|\td\eps'\|_{\dot H^1}
			+\|u_{\mu^{-1},-\theta}-W\|_{\dot H^1}
			<\delta_p+\delta_{\mathrm{dec}}.
		\end{aligned}
	\end{equation*}
	Together with $\sum_{j=1}^L(|b_j'|+|\eta_j'|)<\delta_p$, this ensures that $(\lambda'/\mu,\gamma'-\theta,b_1',\eta_1',\ldots,b_L',\eta_L')$ satisfies \eqref{eq:parameter range}. These parameters also solve $\calG(u_{\mu^{-1},-\theta},\cdot)=0$, so uniqueness in IFT gives $\lambda'=\mu\lambda$, $\gamma'=\theta+\gamma$, $b_j'=b_j$, and $\eta_j'=\eta_j$. The decomposition identities then give $\td\eps'=\td\eps$ and $\eps'=\eps$.
	Choose and fix $\psi_+,\psi_-\in C^\infty_{c,\rad}$ such that
	\begin{equation}\label{eq:eigenmode duality}
		(\phi_j,\psi_{\iota})=(\varphi_j,\psi_{\iota})=0,\quad
		(-HY_{-\iota'},\psi_{\iota})=\delta_{\iota\iota'},\quad \iota,\iota'\in\{+,-\},\quad 0\le j\le L.
	\end{equation}
	By \eqref{eq:eigenmode duality} and \eqref{eq:eigenmode mode def}, we conclude (3), its uniqueness, and the $C^1$ dependence of $\alpha_\pm$.
\end{proof}

\subsection{Trapped solutions}
We prescribe initial data near \eqref{eq:special} and use a topological argument to select the parameters in the unstable directions so that the solution remains in the bootstrap regime for all $s\ge s_0$. Fix
\begin{equation}\label{eq:kappa exponent}
	\kappa\coloneqq\frac{1}{64(4L-1)}.
\end{equation}
Fix sufficiently small $\delta_{\mathrm c},\epsilon_{\mathrm h}>0$. We define $\calO_{\mathrm{init}}$ as the set of initial data of the following form for all sufficiently large $s_0$ and all $\gamma_0\in\mathbb R$:
\begin{equation}\label{eq:init form}
	u_0=\left[\chi_{\delta_{\mathrm c}\lambda_0^{-1}}(W+P)+a_+\psi_++a_-\psi_- +\eps_{\mathrm{in}}\right]_{\lambda_0,\gamma_0}
\end{equation}
Here $\psi_\pm$ are the fixed functions satisfying \eqref{eq:eigenmode duality}, and $a_+,a_-\in\mathbb R$ are defined by
\begin{equation*}
	\begin{aligned}
		a_+&\coloneqq e_0^{-L}\alpha_+(s_0)+\big(-HY_-,(1-\chi_{\delta_{\mathrm c}\lambda_0^{-1}})(W+P)\big),\\
		a_-&\coloneqq(-e_0)^{-L}\alpha_-(s_0)+\big(-HY_+,(1-\chi_{\delta_{\mathrm c}\lambda_0^{-1}})(W+P)\big).
	\end{aligned}
\end{equation*}
All profile parameters are evaluated at $s=s_0$, with $t(s_0)=0$. We impose the following conditions:
\begin{enumerate}
	\item Initial scale:
	\begin{equation}\label{eq:init scale}
		\lambda_0=s_0^{-c_1}=s_0^{-\frac{2L}{4L-1}}.
	\end{equation}
	\item Smallness of the stable modes:
	\begin{equation}\label{eq:init para stable}
		s_0^\kappa|(\calA_b)_1(s_0)|\le\frac14,\qquad \epsilon_{\mathrm h}s_0^{L+1}|\alpha_-(s_0)|\le\frac14.
	\end{equation}
	\item Smallness of the unstable modes:
	\begin{equation}\label{eq:init para unstable}
		\left(\big(s_0^\kappa(\calA_b)_j(s_0)\big)_{j=2}^L, s_0^\kappa \calA_\eta(s_0), \epsilon_{\mathrm h}s_0^{L+1}\alpha_+(s_0)\right)\in[-1,1]^{2L}.
	\end{equation}
	\item Smallness of the initial radiation:
	\begin{equation}\label{eq:init radiation}
		\eps_{\mathrm{in}}\in H^{2L+1}_{\rad}\cap\calZ_{\mathrm h}^\perp,\qquad
		\|\eps_{\mathrm{in}}\|_{H^{2L+1}}\le s_0^{-L-1}.
	\end{equation}
\end{enumerate}
Here, $\calA_b$ and $\calA_\eta$ denote the coordinates defined in \eqref{eq:eigen coordinates} with \eqref{eq:fluctuations}. These coordinates determine $b_j$ and $\eta_j$ from the corresponding components of the prescribed tuples in \eqref{eq:init para stable} and \eqref{eq:init para unstable}. Moreover,
\begin{equation}\label{eq:init para relation}
	|b_j(s_0)|\sim b_1(s_0)^j\sim s_0^{-j},\qquad
	|\eta_j(s_0)|\lesssim s_0^{-\kappa}b_1(s_0)^j \sim s_0^{-j-\kappa},
	\qquad 1\le j\le L.
\end{equation}
Recalling \eqref{eq:w decom}, we have
\begin{equation*}
	\td\eps(s_0)=[u_0]_{\lambda_0^{-1},-\gamma_0}-W,\qquad 
	\eps(s_0)=\td\eps(s_0)-P.
\end{equation*}
With the constants chosen as in Remark~\ref{rem:parameter dependence}, these remainders satisfy
\begin{equation}\label{eq:initial bounds}
	\begin{gathered}
		\|\td\eps(s_0)\|_{\dot H^1}\le C\delta_{\mathrm c}^{2L-\frac{1}{2}}\ll\frac{\delta_{\mathrm{dec}}}{4},\quad
		\|\eps(s_0)\|_{\dot{\calH}^{2L+1}}\le C\delta_{\mathrm c}^{-\frac{1}{2}}s_0^{-2Lc_1}
		\ll \frac K4s_0^{-2Lc_1}.
	\end{gathered}
\end{equation}
Indeed, Proposition~\ref{prop:finite profile} and \eqref{eq:init para relation} imply
\begin{equation}\label{eq:profile input}
	|P_j|_k\lesssim s_0^{-j}r^2\langle r\rangle^{2j-3},\qquad
	|\nabla^kP_j|\lesssim s_0^{-j}\langle r\rangle^{2j-1-k}.
\end{equation}
Thus, we have
\begin{equation}\label{eq:init a control}
	|a_\pm|\lesssim |\alpha_\pm(s_0)|+s_0^{-L-1}
	\lesssim_{\epsilon_{\mathrm h}}s_0^{-L-1}.
\end{equation}
Using (4), \eqref{eq:profile input}, and \eqref{eq:init a control}, we obtain
\begin{align*}
	\|\td\eps(s_0)\|_{\dot H^1}
	&\lesssim(\delta_{\mathrm c}\lambda_0^{-1})^{-\frac{1}{2}}+\sum_{j=1}^Ls_0^{-j}(\delta_{\mathrm c}\lambda_0^{-1})^{2j-\frac{1}{2}}+C_{\epsilon_{\mathrm h}}s_0^{-L-1}
	\lesssim\delta_{\mathrm c}^{2L-\frac{1}{2}},
	\\
	\|\eps(s_0)\|_{\dot{\calH}^{2L+1}}
	&\lesssim(\delta_{\mathrm c}\lambda_0^{-1})^{-2L-\frac{1}{2}}+\sum_{j=1}^Ls_0^{-j}(\delta_{\mathrm c}\lambda_0^{-1})^{2j-2L-\frac{1}{2}}+C_{\epsilon_{\mathrm h}}s_0^{-L-1}
	\lesssim\delta_{\mathrm c}^{-\frac{1}{2}}s_0^{-2Lc_1}.
\end{align*}
These estimates prove \eqref{eq:initial bounds}.

We also take $s_0$ large enough that $\chi_{\delta_{\mathrm c}\lambda_0^{-1}}=1$ on the supports of all the compact tests. Then $\eps(s_0)\in\calZ^\perp$, and \eqref{eq:initial bounds} gives $\calO_{\mathrm{init}}\subset H^{2L+1}_{\rad}\cap\calO_{\mathrm{dec}}$. By uniqueness in Lemma~\ref{lem:decomposition}, the decomposition of $u_0$ has exactly the prescribed parameters $\lambda_0,\gamma_0,b_j(s_0),\eta_j(s_0),\alpha_\pm(s_0)$ and remainder $\eps(s_0)$.

\begin{rem}[Parameter dependence]\label{rem:parameter dependence}
	We can track the order in which we choose the parameters. We will presume the size dependence
	\begin{equation}\label{eq:parameter dependence}
		0<s_0^{-1}\ll K^{-1}\ll\delta_{\mathrm c}^{\frac{1}{2}}\ll\delta_{\mathrm{dec}}^{\frac{1}{4L-1}}\ll\min\{1,\delta^*\}^{\frac{1}{4L-1}},\quad
		0<s_0^{-1}\ll\epsilon_{\mathrm h}\ll1.
	\end{equation}
	Fix $L$ and $\delta^*>0$ in Theorem~\ref{thm:main}. First choose the radius $\delta_{\mathrm{dec}}$ of $\calO_{\mathrm{dec}}$ in \eqref{eq:decom tube} as in Lemma~\ref{lem:decomposition}, with $\delta_{\mathrm{dec}}<\delta^*$. Next choose $\delta_{\mathrm c}$ for the profile cutoff in \eqref{eq:init form}, and then the constant $K$ in the higher-order energy bound \eqref{eq:main bootstrap assump}, satisfying the first inequality in \eqref{eq:parameter K s0 choice}. Choose $\epsilon_{\mathrm h}$, which normalizes $\alpha_\pm$ in \eqref{eq:init para stable} and \eqref{eq:init para unstable}, independently of $\delta_{\mathrm c}$ and $K$. Finally take $s_0$ sufficiently large for the second inequality in \eqref{eq:parameter K s0 choice} and the preceding initial estimates. Thus $\lambda_0=s_0^{-c_1}$ ranges over all sufficiently small positive values. These choices are independent of $\gamma_0$ and the length of the bootstrap interval.
\end{rem}

We construct trapped solutions by choosing the initial parameters in the unstable directions in \eqref{eq:init para unstable}. 
\begin{prop}[Trapped regime]\label{prop:trapped}
	Let $u_0(s_0)$ be of the form \eqref{eq:init form} satisfying \eqref{eq:init scale}, \eqref{eq:init para stable}, and \eqref{eq:init radiation}. Then there exist initial parameters in the unstable directions satisfying \eqref{eq:init para unstable} such that the corresponding solution $u$ to \eqref{eq:NLS} admits the decompositions \eqref{eq:decomposition 2} and \eqref{eq:eigenmode decom}, with the following bounds holding for all $s\ge s_0$:
	\begin{equation}\label{eq:main bootstrap assump}
		\begin{gathered}
			\left(s^\kappa\calA_b(s),s^\kappa\calA_\eta(s),\epsilon_{\mathrm h}s^{L+1}\alpha_+(s),\epsilon_{\mathrm h}s^{L+1}\alpha_-(s)\right)\in[-1,1]^{2L+2},\\
			\|\eps(s)\|_{\dot{\calH}^{2L+1}}\le Ks^{-2Lc_1},\qquad
			\|\td\eps(s)\|_{\dot H^1}\le\delta_{\mathrm{dec}}/2.
		\end{gathered}
	\end{equation}
\end{prop}
We call a solution satisfying \eqref{eq:main bootstrap assump} for all $s\ge s_0$ a trapped solution.

\begin{proof}[Proof of Proposition~\ref{prop:trapped} assuming Proposition~\ref{prop:main bootstrap} and Lemma~\ref{lem:exit map}]
	We fix $s_0$, $\gamma_0$, and $\eps_{\mathrm{in}}$, together with the stable parameters $(\calA_b)_1(s_0),\alpha_-(s_0)$, and vary the unstable initial parameters in \eqref{eq:init para unstable}. For each choice, we consider the corresponding maximal solution $u(s)=u(t(s))$.

	We use the refined parameters in \eqref{eq:refined mod para} and the coordinates $\td\calA_b,\td\calA_\eta$ in \eqref{eq:refined eigen coordinates}, defined later in Section~\ref{sec:refine mod}. They remain close to the original parameters and satisfy sharper modulation estimates, which we use to prove Lemma~\ref{lem:exit map}. We define
	\begin{equation}\label{eq:fast coordinates}
		\begin{aligned}
			\bfA(s)&\coloneqq\left(\big(s^\kappa(\calA_b)_j\big)_{j=2}^L,\ s^\kappa\calA_\eta,\ \epsilon_{\mathrm h}s^{L+1}\alpha_+\right),\\
			\td{\bfA}(s)&\coloneqq\left(\big(s^\kappa(\td\calA_b)_j\big)_{j=2}^L,\ s^\kappa\td\calA_\eta,\ \epsilon_{\mathrm h}s^{L+1}\alpha_+\right).
		\end{aligned}
	\end{equation}
	Thus the initial choices are parametrized by $\bfA(s_0)=\bfA_{\mathrm{init}}\in[-1,1]^{2L}$. Under \eqref{eq:main bootstrap assump}, we apply \eqref{eq:dot calH linear map} and \eqref{eq:W tails} to the definitions \eqref{eq:refined mod para} and \eqref{eq:refined eigen coordinates} to obtain
	\begin{equation}\label{eq:initial map error}
		|\td{\bfA}_j(s)-\bfA_j(s)|\lesssim_K s^{\frac{1}{2} \kappa-\frac{1}{4(4L-1)}}=o(1),\qquad 1\le j\le2L,
	\end{equation}
	uniformly on the interval as $s_0\to\infty$. The same estimate holds at $s=s_0$ for every initial choice $\bfA_{\mathrm{init}}$ by \eqref{eq:initial bounds}. When $\td{\bfA}(s_0;\bfA_{\mathrm{init}})\in(-1/2,1/2)^{2L}$, we define the exit time by
	\begin{equation}\label{eq:first exit time}
		s_{\mathrm{exit}}(\bfA_{\mathrm{init}})\coloneqq\sup\left\{S\ge s_0:\ 
		\begin{aligned}
			&\eqref{eq:decomposition 2},\ \eqref{eq:eigenmode decom},\ \eqref{eq:main bootstrap assump}\text{ hold on }[s_0,S],\\
			&\td{\bfA}(s)\in[-1/2,1/2]^{2L}\text{ for } s_0\le s\le S
		\end{aligned}
		\right\}.
	\end{equation}
	The bounds in \eqref{eq:first exit time} are initially strict by \eqref{eq:init para stable}, \eqref{eq:initial bounds}, and \eqref{eq:initial map error}, so $s_{\mathrm{exit}}>s_0$. It suffices to find a choice with $s_{\mathrm{exit}}=\infty$.

	We argue by contradiction and assume that every such exit time is finite. We first use the following proposition to improve the bounds for the remainder and the stable parameters.
	\begin{prop}[Main bootstrap]\label{prop:main bootstrap}
		Let $u(s)=u(t(s))$ be a solution to \eqref{eq:NLS}, with $u(s_0)\in\calO_{\mathrm{init}}$, that admits the decomposition \eqref{eq:decomposition 2} and \eqref{eq:eigenmode decom} on $[s_0,s_1]$ for some $s_1>s_0$. If \eqref{eq:main bootstrap assump} holds for all $s\in[s_0,s_1]$, then the following bounds hold for all $s\in[s_0,s_1]$:
		\begin{equation}\label{eq:main bootstrap goal}
			\begin{gathered}
				\|\td\eps(s)\|_{\dot H^1}\le\frac{\delta_{\mathrm{dec}}}{4},\qquad
				\|\eps(s)\|_{\dot{\calH}^{2L+1}}\le\frac K2s^{-2Lc_1},\\
				|s^\kappa(\calA_b)_1(s)|\le\frac12,\qquad
				\left|\epsilon_{\mathrm h}s^{L+1}\alpha_-(s)\right|\le\frac12.
			\end{gathered}
		\end{equation}
	\end{prop}
	By local well-posedness, \eqref{eq:decom tube}, and \eqref{eq:main bootstrap goal}, $d_W(u(s_{\mathrm{exit}}))\le\delta_{\mathrm{dec}}/4<\delta_{\mathrm{dec}}$, so $u(s_{\mathrm{exit}})\in\calO_{\mathrm{dec}}$. Lemma~\ref{lem:decomposition} therefore excludes failure of the decomposition at $s_{\mathrm{exit}}$.
	Moreover, \eqref{eq:main bootstrap goal} and \eqref{eq:initial map error} keep all bounds in \eqref{eq:main bootstrap assump} strict at $s_{\mathrm{exit}}$. Thus \eqref{eq:first exit time} forces an exit through an unstable direction:
	\begin{equation*}
		\td{\bfA}(s_{\mathrm{exit}};\bfA_{\mathrm{init}})\in\partial[-1/2,1/2]^{2L}.
	\end{equation*}
	To define the exit map on the entire initial cube, we normalize $\td{\bfA}(s_0;\bfA_{\mathrm{init}})$ when it lies outside $(-1/2,1/2)^{2L}$. We define $\Phi:[-1,1]^{2L}\to\partial[-1,1]^{2L}$ by
	\begin{equation}\label{eq:exit map}
		\Phi(\bfA_{\mathrm{init}})\coloneqq\begin{cases}
			2\td{\bfA}(s_{\mathrm{exit}}(\bfA_{\mathrm{init}});\bfA_{\mathrm{init}}),&\td{\bfA}(s_0;\bfA_{\mathrm{init}})\in(-1/2,1/2)^{2L},\\
			\frac{\td{\bfA}(s_0;\bfA_{\mathrm{init}})}{\max_{1\le j\le2L}|\td{\bfA}_j(s_0;\bfA_{\mathrm{init}})|},&\td{\bfA}(s_0;\bfA_{\mathrm{init}})\notin(-1/2,1/2)^{2L}.
		\end{cases}
	\end{equation}
	\begin{lem}[Exit map]\label{lem:exit map}
		The map $\Phi:[-1,1]^{2L}\to\partial[-1,1]^{2L}$ is continuous.
	\end{lem}

	We prove Proposition~\ref{prop:main bootstrap} and Lemma~\ref{lem:exit map} in Section~\ref{sec:main bootstrap}. By Lemma~\ref{lem:exit map}, $\bfA_{\mathrm{init}}\mapsto-\Phi(\bfA_{\mathrm{init}})$ is a continuous map from $[-1,1]^{2L}$ to itself. By Brouwer's fixed point theorem and Lemma~\ref{lem:exit map}, there exists $\bfA_{\mathrm{init}}\in \partial[-1,1]^{2L}$ such that $\bfA_{\mathrm{init}}=-\Phi(\bfA_{\mathrm{init}})$. Therefore, from \eqref{eq:initial map error}, we deduce
	\begin{equation*}
		\bfA_{\mathrm{init}}=-\Phi(\bfA_{\mathrm{init}})
		=-\bfA_{\mathrm{init}}+o(1),
	\end{equation*}
	which contradicts $\bfA_{\mathrm{init}}\in \partial[-1,1]^{2L}$.
	Hence, for some $\bfA_{\mathrm{init}}$, we have $s_{\mathrm{exit}}(\bfA_{\mathrm{init}})=\infty$. Therefore, the corresponding solution satisfies \eqref{eq:main bootstrap assump} for all $s\ge s_0$.
\end{proof}

\begin{rem}[Codimension $2L$ stability]\label{rem:codimension}
	For every admissible choice of the initial scale $\lambda_0$, phase $\gamma_0$, radiation $\eps_{\mathrm{in}}$, and stable parameters $(\calA_b)_1(s_0),\alpha_-(s_0)$, Proposition~\ref{prop:trapped} ensures the existence of at least one choice of the $2L$ unstable parameters in \eqref{eq:init para unstable} for which the corresponding solution remains trapped. Proposition~\ref{prop:main bootstrap} provides the improved remainder and stable bounds, while the shooting argument determines the unstable parameters. Thus, the remaining parameters are free, and the construction has formal codimension $2L$. We do not pursue the manifold structure here; see \cite{Collot2018MemAmer,KimKwon2020blowup} for related Lipschitz blow-up manifold constructions.
\end{rem}

\begin{quote}
	\emph{In the rest of the paper, we assume the bootstrap assumption \eqref{eq:main bootstrap assump}.}
\end{quote}

\subsection{Modulation estimates}\label{sec:refine mod}
In this subsection, we establish the modulation estimates that justify the formal modulation laws \eqref{eq:formal mod laws}. We derive the estimates for $\lambda,\gamma,b_j,\eta_j$ by differentiating the orthogonality conditions in \eqref{eq:decomposition 2}, and control the coefficients $\alpha_\pm$ by testing \eqref{eq:eps equ} against $-HY_\mp$. For the last parameters $b_L,\eta_L$, we use refined modulation parameters to obtain the additional smallness needed to close the bootstrap argument.

We first establish the following coercivity estimate. We denote
\begin{equation*}
	\eps_L\coloneqq(-iH)^L\eps.
\end{equation*}

\begin{lem}[Higher-order coercivity]\label{lem:energy coercivity}
	For $\eps\in\dot{\calH}^{2L+1}_{\rad}\cap\calZ^\perp$, we have
	\begin{equation}\label{eq:coercivity}
		\begin{gathered}
			\|\eps_L\|_{\dot H^1}\sim\|\eps\|_{\dot{\calH}^{2L+1}},\qquad |(H\eps_L,\eps_L)|\lesssim\|\eps\|_{\dot{\calH}^{2L+1}}^2,\\
			\|\eps\|_{\dot{\calH}^{2L+1}}^2\lesssim_L\tfrac12(H\eps_L,\eps_L)+|\alpha_+|^2+|\alpha_-|^2.
		\end{gathered}
	\end{equation}
\end{lem}
\begin{proof}
	The equivalence $\|\eps_L\|_{\dot H^1}\sim\|\eps\|_{\dot{\calH}^{2L+1}}$ follows from \eqref{eq:weighted reconstruction} and \eqref{eq:decomposition 2}. Boundedness of the quadratic form of $H$ on $\dot H^1$ then gives $|(H\eps_L,\eps_L)|\lesssim\|\eps\|_{\dot{\calH}^{2L+1}}^2$.

	For the last estimate, we use \eqref{eq:projection}, \eqref{eq:projection coercivity}, \eqref{eq:decomposition 2}, and \eqref{eq:eigenmode mode def}:
	\begin{align*}
		\|\eps\|_{\dot{\calH}^{2L+1}}^2\sim \|\eps_L\|_{\dot H^1}^2
		&\lesssim_L\|\Pi\eps_L\|_{\dot H^1}^2+|\alpha_+|^2+|\alpha_-|^2\\
		&\sim(H\Pi\eps_L,\Pi\eps_L)+|\alpha_+|^2+|\alpha_-|^2\\
		&=(H\eps_L,\eps_L)+2\alpha_+\alpha_-+|\alpha_+|^2+|\alpha_-|^2\\
		&\le2[\tfrac12(H\eps_L,\eps_L)+|\alpha_+|^2+|\alpha_-|^2]. \qedhere
	\end{align*}
\end{proof}

We now establish the modulation estimates. Note that, by \eqref{eq:fluctuations}, \eqref{eq:eigen coordinates}, and \eqref{eq:main bootstrap assump}, we have
\begin{equation}\label{eq:mod para s bound}
	b_1\sim s^{-1},\qquad |b_j|+|\eta_j|\lesssim s^{-j}\quad(1\le j\le L).
\end{equation}

\begin{lem}[Modulation estimates]
	We have
	\begin{equation}\label{eq:mod esti}
		|\mathbf{Mod}|\lesssim\|\eps\|_{\dot{\calH}^{2L+1}}+s^{-L-1}.
	\end{equation}
	The coefficients $\alpha_\pm$ satisfy
	\begin{equation}\label{eq:eigenmode esti}
		|\partial_s\alpha_\pm\mp e_0\alpha_\pm|\lesssim
		s^{-1}\|\eps\|_{\dot{\calH}^{2L+1}}
		+\|\eps\|_{\dot{\calH}^{2L+1}}^2
		+s^{-L-1}.
	\end{equation}
\end{lem}
\begin{proof}
	Subtracting \eqref{eq:profile error} from \eqref{eq:w equ} and using \eqref{eq:residual expansion}, we obtain
	\begin{equation}\label{eq:eps equ}
		\partial_s\eps=-iH\eps+\tfrac{\lambda_s}{\lambda}\Lambda\eps-i\gamma_s\eps
		+\calN_W(\td\eps)-\calN_W(P)
		+\mathbf{Mod}\cdot\textbf{v}-\td\Psi.
	\end{equation}
	Differentiating the orthogonality conditions in \eqref{eq:decomposition 2} and using \eqref{eq:eps equ}, we obtain, for $f=\phi_k,\varphi_k$, $0\le k\le L$,
	\begin{align*}
		(f,\mathbf{Mod}\cdot\textbf{v})
		=&(f,iH\eps+b_1\Lambda\eps+\tfrac i2\eta_1\eps)-\mathbf{Mod}_\lambda(f,\Lambda\eps)
		+\mathbf{Mod}_\gamma(f,i\eps)\\
		&-(f,\calN_W(\td\eps)-\calN_W(P))+(f,\td\Psi).
	\end{align*}
	For the profile error $\td\Psi$, Proposition~\ref{prop:finite profile} and \eqref{eq:mod para s bound} give
	\begin{equation}\label{eq:mod esti profile error}
		|(f,\td\Psi)|\lesssim\sum_{j=0}^k\|\langle x\rangle^{-10L+3+j}\nabla^j\td\Psi\|_{L^2}
		\lesssim_k s^{-L-1}.
	\end{equation}
	Next we consider the nonlinear terms. Since $\td\eps=P+\eps$, \eqref{eq:notation nonlinearity} gives
	\begin{equation*}
		|\calN_W(\td\eps)-\calN_W(P)|
		\lesssim |\eps|(|P|+|\eps|)(W+|P|+|\eps|)^3.
	\end{equation*}
	From \eqref{eq:dot cal pointwise}, we have
	\begin{equation}\label{eq:P and eps esti}
		|P(x)|\lesssim\sum_{j=1}^L(|b_j|+|\eta_j|)\langle x\rangle^{2L-1},\qquad
		|\eps(x)|\lesssim\|\eps\|_{\dot{\mathcal H}^{2L+1}}\langle x\rangle^{2L-\frac{1}{2}}.
	\end{equation}
	By \eqref{eq:decomposition 2} and Proposition~\ref{prop:finite profile}, we have
	\begin{equation}\label{eq:mod esti nonlinear part}
		|(f,\calN_W(\td\eps)-\calN_W(P))|
		\lesssim\bigg(\sum_{j=1}^L(|b_j|+|\eta_j|)+\|\eps\|_{\dot{\calH}^{2L+1}}\bigg)
		\|\eps\|_{\dot{\calH}^{2L+1}}.
	\end{equation}
	Thanks to \eqref{eq:decomposition 2} with $\eps=\td\eps-P$ and \eqref{eq:P and eps esti}, we obtain
	\begin{equation}\label{eq:mod esti scale phase}
		|\mathbf{Mod}_\lambda(f,\Lambda\eps)|
		+|\mathbf{Mod}_\gamma(f,i\eps)|
		\lesssim\min(\delta_p,\|\eps\|_{\dot{\mathcal H}^{2L+1}})|\mathbf{Mod}|.
	\end{equation}
	The remaining terms are bounded using \eqref{eq:P and eps esti} and \eqref{eq:mod para s bound}, yielding
	\begin{equation}\label{eq:mod esti linear eps}
		|(f,iH\eps)|\lesssim\|\eps\|_{\dot{\calH}^{2L+1}},\qquad
		|(f,b_1\Lambda\eps+\tfrac i2\eta_1\eps)|\lesssim s^{-1}\|\eps\|_{\dot{\calH}^{2L+1}}.
	\end{equation}
	Collecting \eqref{eq:mod esti profile error} and \eqref{eq:mod esti nonlinear part}--\eqref{eq:mod esti linear eps}, we derive
	\begin{equation*}
		|\mathbf{Mod}|\lesssim\|\eps\|_{\dot{\calH}^{2L+1}}
		+s^{-L-1}+\delta_p|\mathbf{Mod}|.
	\end{equation*}
	Absorbing the last term proves \eqref{eq:mod esti}.
	
	For $\alpha_\pm$, we take $f=-HY_\mp$. By \eqref{eq:eigenfunctions}, \eqref{eq:dual mode tests}, and \eqref{eq:g kernels},
	\begin{equation}\label{eq:mod esti orth}
		(Hi)f=\pm e_0f,\qquad
		(f,\mathsf V_j)=(f,\mathsf U_j)=0\quad(0\le j\le L).
	\end{equation}
	Using \eqref{eq:eigenmode mode def} and \eqref{eq:eps equ}, we obtain
	\begin{align*}
		(\pm e_0)^{-L}(\partial_s\alpha_\pm\mp e_0\alpha_\pm)
		=&(f,\mathbf{Mod}\cdot\textbf{v})
		+\mathbf{Mod}_\lambda(f,\Lambda\eps)
		-\mathbf{Mod}_\gamma(f,i\eps)\\
		&-(f,b_1\Lambda\eps+\tfrac i2\eta_1\eps)
		+(f,\calN_W(\td\eps)-\calN_W(P))-(f,\td\Psi).
	\end{align*}
	From the orthogonality \eqref{eq:mod esti orth}, Proposition~\ref{prop:finite profile}, and \eqref{eq:mod para s bound}, we obtain
	\begin{equation*}
		|(f,\mathbf{Mod}\cdot\textbf{v})|
		\lesssim s^{-1}|\mathbf{Mod}|.
	\end{equation*}
	Combining this bound with \eqref{eq:mod para s bound}, \eqref{eq:mod esti}, \eqref{eq:mod esti profile error},  \eqref{eq:mod esti nonlinear part}--\eqref{eq:mod esti linear eps}, we have \eqref{eq:eigenmode esti},
	\begin{equation*}
		|\partial_s\alpha_\pm\mp e_0\alpha_\pm|
		\lesssim
		s^{-1}\|\eps\|_{\dot{\calH}^{2L+1}}
		+\|\eps\|_{\dot{\calH}^{2L+1}}^2
		+s^{-L-1}. \qedhere
	\end{equation*}
\end{proof}

Following the refined modulation arguments in \cite{JendrejLawrieRodriguez2022ASENS,KimKwonOh2026CMP,JeongKim2024arXiv,JeongKimKimKwon2026arXiv}, we correct $b_L,\eta_L$ using moments of $\eps_L$ against truncated $W$ and $i\Lambda W$. Applying $(-iH)^L$ to \eqref{eq:eps equ}, we obtain
\begin{equation}\label{eq:epsL equation}
	\begin{aligned}
		\partial_s\eps_L=&-iH\eps_L+\frac{\lambda_s}{\lambda}(\Lambda+2L)\eps_L-i\gamma_s\eps_L\\
		&-\mathbf{Mod}_{b_L}\Lambda W-\tfrac 12\mathbf{Mod}_{\eta_L} iW
		-(-iH)^L\Theta_{L+1}+\td E_L.
	\end{aligned}
\end{equation}
The expression for $\td E_L$ is given in \eqref{eq:epsL error formula}. The following estimates will be used in the refined modulation estimates and the higher-order energy estimate.

\begin{lem}[Nonlinear estimates]\label{lem:nonlin esti}
	We have
	\begin{equation}\label{eq:nonlin esti}
		\|\td E_L\|_{\dot H^1}\lesssim s^{-1}(\|\eps\|_{\dot{\calH}^{2L+1}}+s^{-2Lc_1}),
	\end{equation}
	and, for $f\in\dot H^1$,
	\begin{equation}\label{eq:nonlin esti Mor}
		|(Hf,\td E_L)|
		\lesssim s^{-1}(\|\eps\|_{\dot{\calH}^{2L+1}}+s^{-2Lc_1})\|f\|_{\mathrm{Mor}}
		+C_Ks^{-2(L+1)c_1}\|f\|_{\dot H^1}.
	\end{equation}
	Here $C_K$ may depend on the fixed bootstrap constant $K$.
\end{lem}
We postpone the proof of Lemma~\ref{lem:nonlin esti} to the end of this subsection.

Recall the exponent $\kappa=\frac1{64(4L-1)}$ in \eqref{eq:kappa exponent}. We define
\begin{equation}\label{eq:reference scale}
	R=R(s)\coloneqq s^{\frac{1}{2}-\kappa},\qquad D_R\coloneqq(\chi_RW,\Lambda W).
\end{equation}
We define the refined parameters by
\begin{equation}\label{eq:refined mod para}
	\td b_L\coloneqq b_L+\frac{(\chi_RW,\eps_L)}{D_R},\qquad
	\td\eta_L\coloneqq\eta_L+\frac{2(i\chi_R\Lambda W,\eps_L)}{D_R}.
\end{equation}
For the following definition only, we set $\td b_j\coloneqq b_j$ and $\td\eta_j\coloneqq\eta_j$ for $j<L$. Using the matrices in \eqref{eq:eigen coordinates}, we define
\begin{equation}\label{eq:refined eigen coordinates}
	\td\calA_b\coloneqq\mathsf S_b\big(s^k(\td b_k-b_k^e)\big)_{k=1}^L,\qquad
	\td\calA_\eta\coloneqq\mathsf S_\eta\big(s^k\td\eta_k\big)_{k=1}^L.
\end{equation}
We set
\begin{equation*}
	\kappa'\coloneqq\frac1{4(4L-1)}-\frac32\kappa>0.
\end{equation*}
\begin{lem}[Refined modulation estimates]\label{lem:refined mod esti}
	The refined parameters satisfy
	\begin{equation}\label{eq:refined mod compare}
		|\td b_L-b_L|+|\td\eta_L-\eta_L|
		\lesssim_K s^{-L-\kappa'-2\kappa}
	\end{equation}
	and
	\begin{equation}\label{eq:refined mod esti}
		\begin{aligned}
			\big|\partial_s\td b_L+(2L-\tfrac12)b_1\td b_L+\tfrac12\eta_1\td\eta_L\big|&\lesssim_K s^{-L-1-\kappa'},\\
			\big|\partial_s\td\eta_L+(2L-\tfrac12)b_1\td\eta_L-\tfrac12\eta_1\td b_L\big|&\lesssim_K s^{-L-1-\kappa'}.
		\end{aligned}
	\end{equation}
	The refined coordinates satisfy
	\begin{equation}\label{eq:refined mod compare normalize}
		|\td\calA_b-\calA_b|+|\td\calA_\eta-\calA_\eta|\lesssim_K s^{-\kappa'-2\kappa}
	\end{equation}
	and
	\begin{equation}\label{eq:refined mod esti normalize}
		s\partial_s\begin{pmatrix}\td\calA_b\\\td\calA_\eta\end{pmatrix}
		=\begin{pmatrix}\mathsf D_b\td\calA_b\\\mathsf D_\eta\td\calA_\eta\end{pmatrix}
		+O_K\big(|\td\calA_b|^2+|\td\calA_\eta|^2+s^{-\kappa'}\big).
	\end{equation}
\end{lem}
\begin{proof}[Proof of Lemma~\ref{lem:refined mod esti} assuming Lemma~\ref{lem:nonlin esti}]
	By \eqref{eq:W tails},
	\begin{equation}\label{eq:moment denominator}
		D_R=-\frac32R\int_0^\infty\chi(r)dr+O(1),\qquad
		\|\chi_RW\|_{L^{6/5}}+\|\chi_R\Lambda W\|_{L^{6/5}}\lesssim R^{\frac{3}{2}}.
	\end{equation}
	Combining \eqref{eq:moment denominator}, Sobolev embedding, \eqref{eq:dot calH linear map}, and \eqref{eq:main bootstrap assump} with $2Lc_1=L+\frac14+\kappa'+\frac32\kappa$, we obtain \eqref{eq:refined mod compare}. By \eqref{eq:refined eigen coordinates}, multiplication by $s^L$ also gives \eqref{eq:refined mod compare normalize}.

	We now prove \eqref{eq:refined mod esti}. For $\calZ_R=\chi_RW$, $i\chi_R\Lambda W$, we test \eqref{eq:epsL equation} and use \eqref{eq:notation dilation adjoint} to obtain
	\begin{equation*}
		\begin{aligned}
			\partial_s&(\calZ_R,\eps_L)+\mathbf{Mod}_{b_L}(\calZ_R,\Lambda W)+\tfrac12\mathbf{Mod}_{\eta_L}(\calZ_R,iW)\\
			=&((-iH)^*\calZ_R,\eps_L)+(\calZ_R,\td E_L-(-iH)^L\Theta_{L+1})\\
			&+\tfrac{\lambda_s}{\lambda}((2L-2-\Lambda)\calZ_R,\eps_L)
			+\gamma_s(i\calZ_R,\eps_L)+R_s(\partial_R\calZ_R,\eps_L).
		\end{aligned}
	\end{equation*}
	On the other hand, from \eqref{eq:refined mod para}, the first line becomes
	\begin{equation*}
		\begin{aligned}
			\partial_s&(\calZ_R,\eps_L)+\mathbf{Mod}_{b_L}(\calZ_R,\Lambda W)
			+\tfrac12\mathbf{Mod}_{\eta_L}(\calZ_R,iW)\\
			=&(\calZ_R,\Lambda W)\big[\partial_s\td b_L+(2L-\tfrac12)b_1b_L+\tfrac12\eta_1\eta_L\big]\\
			&+\tfrac12(\calZ_R,iW)[\partial_s\td\eta_L+(2L-\tfrac12)b_1\eta_L-\tfrac12\eta_1b_L]
			+D_R^{-1}R_s\partial_RD_R(\calZ_R,\eps_L).
		\end{aligned}
	\end{equation*}
	Now, we estimate the remainder.
	Since $(-iH)^*W=(-iH)^*(i\Lambda W)=0$, we have
	\begin{equation*}
		|((-iH)^*\calZ_R,\eps_L)|\lesssim 
		 R^{-3}\|\chi_{2R}\|_{L^{6/5}}\|\eps_L\|_{L^6}
		\lesssim R^{-\frac{1}{2}}\|\eps\|_{\dot{\calH}^{2L+1}}.
	\end{equation*}
	By \eqref{eq:lin profile decay} and \eqref{eq:mod para s bound},
	\begin{equation*}
		\|\langle r\rangle(-iH)^L\Theta_{L+1}\|_{L^2}\lesssim s^{-L-1}.
	\end{equation*}
	Since $\|\langle r\rangle^{-1}\calZ_R\|_{L^2}\lesssim1$, \eqref{eq:nonlin esti} and \eqref{eq:moment denominator} yield
	\begin{equation*}
		|(\calZ_R,\td E_L-(-iH)^L\Theta_{L+1})|
		\lesssim s^{-1}R^{\frac{3}{2}}(\|\eps\|_{\dot{\calH}^{2L+1}}+s^{-2Lc_1})+s^{-L-1}.
	\end{equation*}
	By \eqref{eq:mod para s bound}, \eqref{eq:mod esti}, and \eqref{eq:reference scale},
	\begin{equation}\label{eq:scale phase bound}
		|\tfrac{\lambda_s}{\lambda}|+|\gamma_s|+|\tfrac{R_s}{R}|\lesssim s^{-1}.
	\end{equation}
	Hence,
	\begin{equation*}
		|\tfrac{\lambda_s}{\lambda}((2L-2-\Lambda)\calZ_R,\eps_L)+\gamma_s(i\calZ_R,\eps_L)|
		\lesssim s^{-1}R^{\frac{3}{2}}\|\eps\|_{\dot{\calH}^{2L+1}}.
	\end{equation*}
	Since $|\partial_RD_R|\lesssim1$ and $\|\partial_R\calZ_R\|_{L^{6/5}}\lesssim R^{\frac{1}{2}}$ with $|R_s|\lesssim s^{-1}R$, we get
	\begin{equation*}
		|R_s(\partial_R\calZ_R,\eps_L)|+|D_R^{-1}R_s\partial_RD_R(\calZ_R,\eps_L)|
		\lesssim s^{-1}R^{\frac{3}{2}}\|\eps\|_{\dot{\calH}^{2L+1}}.
	\end{equation*}
	Thus, dividing by $D_R$, with $|D_R|\sim R$ and using \eqref{eq:refined mod compare}, we derive \eqref{eq:refined mod esti},
	\begin{equation*}
		\begin{aligned}
			&|\partial_s\td b_L+(2L-\tfrac12)b_1\td b_L+\tfrac12\eta_1\td\eta_L|
			+|\partial_s\td\eta_L+(2L-\tfrac12)b_1\td\eta_L-\tfrac12\eta_1\td b_L|\\
			&\lesssim R^{-\frac{3}{2}}\|\eps\|_{\dot{\calH}^{2L+1}}
			+s^{-1} R^{\frac{1}{2}}(\|\eps\|_{\dot{\calH}^{2L+1}}+s^{-2Lc_1})+s^{-L-1}R^{-1}\lesssim_K s^{-L-1-\kappa'}.
		\end{aligned}
	\end{equation*}
	Using \eqref{eq:mod esti} for $j<L$ and \eqref{eq:refined mod esti} for $j=L$, we multiply the errors by $s^{j+1}$ and apply $\mathsf S_b,\mathsf S_\eta$. Then, \eqref{eq:linear maps}, \eqref{eq:main bootstrap assump}, \eqref{eq:refined mod compare}, and \eqref{eq:refined mod compare normalize} prove \eqref{eq:refined mod esti normalize}.
\end{proof}

Now, we finish the proof of Lemma~\ref{lem:nonlin esti}.

\begin{proof}[Proof of Lemma~\ref{lem:nonlin esti}]
	We set $R\coloneqq s^{c_1}$.
	Note that, by \eqref{eq:g kernels} and Proposition~\ref{prop:finite profile},
	\begin{equation*}
		\begin{gathered}
			(-iH)^L\partial_{b_L}P=\Lambda W,\quad
			(-iH)^L\partial_{\eta_L}P=\tfrac{i}{2}W,\quad
			(-iH)^L\Lambda W=(-iH)^L(iW)=0.
		\end{gathered}
	\end{equation*}
	Applying $(-iH)^L$ to \eqref{eq:eps equ}, we obtain \eqref{eq:epsL equation} with
	\begin{equation}\label{eq:epsL error formula}
		\td E_L=E_1+E_2+E_3,
	\end{equation}
	where
	\begin{align*}
		E_1\coloneqq&(-iH)^L(\calN_W(P+\eps)-\calN_W(P)-\td\Psi+\Theta_{L+1}),\\
		E_2\coloneqq&\frac{\lambda_s}{\lambda}([(-iH)^L,\Lambda]-2L(-iH)^L)\eps
		-\gamma_s[(-iH)^L,i]\eps,\\
		E_3\coloneqq&\mathbf{Mod}_\lambda(-iH)^L\Lambda P-\mathbf{Mod}_\gamma(-iH)^L(iP)\\
		&-\sum_{j=1}^{L-1}(\mathbf{Mod}_{b_j}(-iH)^L\partial_{b_j}P
		+\mathbf{Mod}_{\eta_j}(-iH)^L\partial_{\eta_j}P).
	\end{align*}
	We first estimate $E_1$.
	By \eqref{eq:Psi td def} and \eqref{eq:notation nonlinearity},
	\begin{equation}
		\begin{aligned}
			E_1=&(-iH)^L\bigg(\calN_W(P+\eps)-\sum_{n=2}^{L+1}\{\calN_W(P)\}_n\bigg)
			\\=&(-iH)^L(E_{1,1}+E_{1,2}+E_{1,3}+\calP_1),
		\end{aligned}
	\end{equation}
	where
	\begin{align*}
		E_{1,1}&\coloneqq\sum_{n=L+2}^{5L}\{\calN_W(P)\}_n,\qquad \calP_1\coloneqq4iW^3[3\Re(P_1)\eps+(P_1+\Re P_1)\ol\eps],\\
		E_{1,2}&\coloneqq3i(|W+P|^4-W^4)\eps+2i(|W+P|^2(W+P)^2-W^4)\ol\eps-\calP_1,\\
		E_{1,3}&\coloneqq\calN_{W+P}(\eps).
	\end{align*}
	Here, $\calP_1$ consists of the terms in $\calN_W(P+\eps)$ that are linear in $\eps$ and have parameter weight $1$.
	Since $|\nabla^m(W^3P_1)|\lesssim_m s^{-1}\langle r\rangle^{-2-m}$, \eqref{eq:lin iterated products} gives
	\begin{equation}\label{eq:localized source norm}
		\|\langle r\rangle(-iH)^L\calP_1\|_{L^2}
		+\|\langle r\rangle^{2}\nabla(-iH)^L\calP_1\|_{L^2}\lesssim s^{-1}\|\eps\|_{\dot{\calH}^{2L+1}}.
	\end{equation}
	
	Next, we estimate $E_{1,1}$.
	Each term in $\{\calN_W(P)\}_n$ is of the form $W^{5-d}P_{j_1}\cdots P_{j_d}$, up to complex conjugation, where $2\le d\le5$ and $j_1+\cdots+j_d=n$. Thus \eqref{eq:kernel growth}, \eqref{eq:correction pointwise}, and \eqref{eq:mod para s bound}, together with \eqref{eq:lin iterated products}, yield
	\begin{equation}\label{eq:lin source symbol}
		|\nabla^m(-iH)^L\{\calN_W(P)\}_n|\lesssim_m s^{-n}\langle r\rangle^{2n-2L-5-m}
		\quad(2\le n\le5L,\ m=0,1).
	\end{equation}
	Consequently, $E_{1,1}$ satisfies
	\begin{equation}
		\begin{aligned}
			\|\chi_R\nabla(-iH)^LE_{1,1}\|_{L^2}
			&\lesssim\sum_{n=L+2}^{5L}s^{-n}\big(1+R^{\max\{2n-2L-\frac92,0\}}\big)\\
			&\lesssim s^{-L-2}+s^{-5L}R^{8L-\frac92}\lesssim R^{-2L-2}.
		\end{aligned}\label{eq:E11 interior}
	\end{equation}
	All terms in $E_{1,2}$ have parameter weights $2\le n\le4L$, so
	\begin{equation}
		\begin{aligned}
			\|\chi_R\nabla(-iH)^LE_{1,2}\|_{L^2}
			&\lesssim\sum_{n=2}^{4L}s^{-n}\sum_{m=0}^{2L+1}\|\chi_R\langle r\rangle^{2n-2L-5+m}\nabla^m\eps\|_{L^2}\\
			&\lesssim R^{-2}\|\eps\|_{\dot{\calH}^{2L+1}}\lesssim_K R^{-2L-2}.
		\end{aligned}\label{eq:E12 interior}
	\end{equation}
	For $E_{1,3}$, which contains at least two factors of $\eps$, we use
	\begin{equation}\label{eq:nonlin pointwise derivatives}
		\begin{aligned}
			\chi_{R}|\nabla^m(W+P)|&\lesssim\langle r\rangle^{-\frac12-m}&& (0\le m\le2L+1),\\
			|\nabla^m\eps|&\lesssim \|\eps\|_{\dot{\calH}^{2L+1}}\langle r\rangle^{2L-\frac12-m}&& (0\le m\le2L-1).
		\end{aligned}
	\end{equation}
	By \eqref{eq:kernel growth}, \eqref{eq:correction pointwise}, and \eqref{eq:mod para s bound}, we have $|\nabla^m P_j|\lesssim_m s^{-j}\langle r\rangle^{2j-1-m}$. Since $r\lesssim R=s^{c_1}$, this gives the first bound. The second follows by applying \eqref{eq:dot cal pointwise} componentwise to $f=\nabla^m\eps$ with $k=2L+1-m$.
	
	Up to complex conjugation, each term in $\nabla^m E_{1,3}$ is of the form
	\begin{equation*}
		\prod_{\ell=1}^{5-q}\nabla^{m_\ell}(W+P)
		\prod_{\ell=1}^q\nabla^{j_\ell}\eps,
	\end{equation*}
	where $2\le q\le5$ and
	$\sum_\ell m_\ell+\sum_\ell j_\ell=m\le2L+1$.
	We may assume $j_1=\max_\ell j_\ell$, so $j_\ell\le L$ for $\ell\ge2$.
	By \eqref{eq:dot calH def} and \eqref{eq:nonlin pointwise derivatives},
	\begin{equation*}
		\bigg\|\chi_R\langle r\rangle^{m-2L-1}
		\prod_{\ell=1}^{5-q}\nabla^{m_\ell}(W+P)
		\prod_{\ell=1}^q\nabla^{j_\ell}\eps\bigg\|_{L^2}
		\lesssim R^{2L(q-1)-2}\|\eps\|_{\dot{\calH}^{2L+1}}^q.
	\end{equation*}
	Thus, \eqref{eq:lin iterated products} gives
	\begin{equation}\label{eq:E13 interior}
		\|\chi_R\nabla(-iH)^LE_{1,3}\|_{L^2}\lesssim
		\sum_{q=2}^5 R^{2L(q-1)-2}\|\eps\|_{\dot{\calH}^{2L+1}}^q
		\lesssim_K R^{-2L-2}.
	\end{equation}
	On the exterior region, we estimate $E_{1,1}+E_{1,2}+E_{1,3}+\calP_1$ together: the polynomial tails of $P$ cancel only in the sum. We retain $w=W+P+\eps$ and use
	\begin{equation}\label{eq:exterior identity}
		\calN_W(P+\eps)=i|w|^4w-3iW^4w-2iW^4\ol w+4iW^5.
	\end{equation}
	To control this, we first combine \eqref{eq:kernel growth}, \eqref{eq:correction pointwise}, \eqref{eq:mod para s bound}, and \eqref{eq:main bootstrap assump} to obtain
	\begin{equation}
		\begin{aligned}
			\|(1-\chi_{R/2})w\|_{\dot H^1}&\lesssim1,\\
			\|(1-\chi_{R/2})w\|_{\dot{\calH}^{2L+1}}
			&\lesssim \|\eps\|_{\dot{\calH}^{2L+1}}+R^{-2L-\frac12}+s^{-L}R^{-\frac12}\lesssim_K R^{-2L}.
		\end{aligned}\label{eq:exterior w}
	\end{equation}
	The same bounds hold with $w$ replaced by $W$. 
	By \eqref{eq:exterior identity}, each term in $\nabla^m\calN_W(P+\eps)$ is of the form $\prod_{j=1}^5\nabla^{m_j}\calW_j$, up to complex conjugation, where $\calW_j\in\{w,W\}$ and $\sum_{j=1}^5m_j=m$.
	By \eqref{eq:exterior w}, interpolation, Sobolev embedding, and Hardy's inequality, we deduce
	\begin{equation*}
		\bigg\|(1-\chi_R)\langle r\rangle^{m-2L-1}
		\prod_{j=1}^5\nabla^{m_j}\calW_j\bigg\|_{L^2}
		\lesssim_K R^{-2L-2}
		\qquad(0\le m\le2L+1).
	\end{equation*}
	Thus, \eqref{eq:lin iterated products} gives
	\begin{equation}\label{eq:exterior nonlinear}
		\|(1-\chi_R)\nabla(-iH)^L\calN_W(P+\eps)\|_{L^2}
		\lesssim_K R^{-2L-2}.
	\end{equation}
	For the subtracted profiles, \eqref{eq:lin source symbol} gives
	\begin{equation}\label{eq:exterior pure}
		\sum_{n=2}^{L+1}\|(1-\chi_R)\nabla(-iH)^L\{\calN_W(P)\}_n\|_{L^2}
		\lesssim s^{-L-1}R^{-\frac52}=s^{-1}R^{-2L-2}.
	\end{equation}
	Collecting \eqref{eq:localized source norm}, \eqref{eq:E11 interior}, \eqref{eq:E12 interior}, \eqref{eq:E13 interior}, \eqref{eq:exterior nonlinear}, and \eqref{eq:exterior pure}, we arrive at
	\begin{equation}\label{eq:E1 estimate}
		\|E_1-(-iH)^L\calP_1\|_{\dot H^1}
		\lesssim_K R^{-2L-2},\quad
		\|E_1\|_{\dot H^1}\lesssim s^{-1}\|\eps\|_{\dot{\calH}^{2L+1}}+C_KR^{-2L-2}.
	\end{equation}

	The estimate for $E_2$ follows directly from \eqref{eq:lin weighted commutators} and \eqref{eq:scale phase bound}:
	\begin{equation}\label{eq:E2 estimate}
		\|\nabla E_2\|_{L^2}\lesssim \|\langle r\rangle E_2\|_{L^2}+\|\langle r\rangle^2\nabla E_2\|_{L^2}\lesssim s^{-1}\|\eps\|_{\dot{\calH}^{2L+1}}.
	\end{equation}

	Finally, we estimate $E_3$. 
	First, by \eqref{eq:g kernels} and \eqref{eq:lin profile decay}, we have
	\begin{equation}\label{eq:iHP compute}
		(-iH)^LP=b_L\Lambda W+\tfrac{\eta_L}{2}iW
		+O(s^{-1}\langle r\rangle^{-3}).
	\end{equation}
	Moreover, combining \eqref{eq:lin commutator coefficients}, \eqref{eq:kernel growth}, \eqref{eq:correction pointwise}, and \eqref{eq:mod para s bound}, we get
	\begin{equation}\label{eq:iH Lambda P compute}
		\big([(-iH)^L,\Lambda]-2L(-iH)^L\big)P=O(s^{-1}\langle r\rangle^{-3}),\qquad
		[(-iH)^L,i]P=O(s^{-1}\langle r\rangle^{-3}).
	\end{equation}
	Since $(\Lambda+\tfrac12)\Lambda W$ and $(\Lambda+\tfrac12)iW$ decay like $\langle r\rangle^{-3}$, \eqref{eq:iHP compute} and \eqref{eq:iH Lambda P compute} imply
	\begin{equation*}
		\begin{aligned}
			(-iH)^L\Lambda P
			&=(2L-\tfrac12)(b_L\Lambda W+\tfrac{\eta_L}{2}iW)
			+O(s^{-1}\langle r\rangle^{-3}),\\
			(-iH)^L(iP)
			&=b_Li\Lambda W-\tfrac{\eta_L}{2}W
			+O(s^{-1}\langle r\rangle^{-3}).
		\end{aligned}
	\end{equation*}
	For $j<L$, the generalized-kernel parts of $\partial_{b_j}P$ and $\partial_{\eta_j}P$ vanish under $(-iH)^L$. The remaining terms come from derivatives of $S_n$ with parameter weight at least one. Hence
	\begin{equation*}
		(-iH)^L\partial_{b_j}P,\quad
		(-iH)^L\partial_{\eta_j}P
		=O(s^{-1}\langle r\rangle^{-3}).
	\end{equation*}
	The $O(s^{-1}\langle r\rangle^{-3})$ remainders in the preceding estimates also satisfy $\nabla^m O(s^{-1}\langle r\rangle^{-3})=O(s^{-1}\langle r\rangle^{-3-m})$. Hence, applying $H$ to these remainders yields $O(s^{-1}\langle r\rangle^{-5})$. Finally, we have
	\begin{equation*}
		H\Lambda W=H(iW)=0,\qquad
		H(i\Lambda W)=4iW^4\Lambda W,\qquad HW=-4W^5.
	\end{equation*}
	Thus, applying $H$ either eliminates the $r^{-1}$ tails or improves their decay to $O(\langle r\rangle^{-5})$. Together with the derivative bounds for the $O(s^{-1}\langle r\rangle^{-3})$ remainders and \eqref{eq:mod esti}, we obtain
	\begin{equation}\label{eq:E3 estimate}
		\|E_3\|_{\dot H^1}+\|\langle r\rangle^2HE_3\|_{L^2}
		\lesssim s^{-1}|\mathbf{Mod}|\lesssim s^{-1}(\|\eps\|_{\dot{\calH}^{2L+1}}+s^{-L-1}).
	\end{equation}
	Combining \eqref{eq:epsL error formula}, \eqref{eq:E1 estimate}, \eqref{eq:E2 estimate}, and \eqref{eq:E3 estimate}, we obtain
	\begin{equation*}
		\|\td E_L\|_{\dot H^1}\lesssim s^{-1}(\|\eps\|_{\dot{\calH}^{2L+1}}+s^{-L-1})+C_KR^{-2L-2}.
	\end{equation*}
	Since $L+1>2Lc_1$ and $R^{-2L-2}=s^{-1-\frac1{4L-1}}s^{-2Lc_1}$, taking $s_0$ sufficiently large depending on $K$ proves \eqref{eq:nonlin esti}.

	We next prove \eqref{eq:nonlin esti Mor}. We have
	\begin{equation}\label{eq:weighted duality}
		|(Hf,h)|\lesssim\|f\|_{\mathrm{Mor}}\min\big\{\|\langle r\rangle^2\nabla h\|_{L^2}+\|\langle r\rangle h\|_{L^2},\ \|\langle r\rangle^2Hh\|_{L^2}\big\}.
	\end{equation}
	By \eqref{eq:weighted duality}, \eqref{eq:localized source norm}, and \eqref{eq:E1 estimate}, we have
	\begin{equation*}
		\begin{aligned}
			|(Hf,E_1)|
			&\le |(Hf,(-iH)^L\calP_1)|+|(Hf,E_1-(-iH)^L\calP_1)|\\
			&\lesssim s^{-1}\|\eps\|_{\dot{\calH}^{2L+1}}\|f\|_{\mathrm{Mor}}+C_KR^{-2L-2}\|f\|_{\dot H^1}.
		\end{aligned}
	\end{equation*}
	Combining this with \eqref{eq:E2 estimate} and \eqref{eq:E3 estimate} using \eqref{eq:weighted duality}, we obtain
	\begin{equation*}
		|(Hf,\td E_L)|\lesssim s^{-1}(\|\eps\|_{\dot{\calH}^{2L+1}}+s^{-L-1})\|f\|_{\mathrm{Mor}}+C_KR^{-2L-2}\|f\|_{\dot H^1}.
	\end{equation*}
	Using $L+1>2Lc_1$ and $R^{-2L-2}=s^{-2(L+1)c_1}$ proves \eqref{eq:nonlin esti Mor}.
\end{proof}

\subsection{Higher-order energy estimate}
In this subsection, we close the bootstrap conclusion for the higher-order energy in \eqref{eq:main bootstrap assump}. By Lemma~\ref{lem:energy coercivity}, it suffices to estimate $\Pi\eps_L$ in $\dot H^1$, since the directions removed by $\Pi$ are controlled by the bootstrap bounds on $\alpha_\pm$. We estimate $\Pi\eps_L$ using the global spacetime estimates of Proposition~\ref{prop:global spacetime} together with the nonlinear estimate \eqref{eq:nonlin esti Mor}.

\begin{prop}[Higher-order energy estimate]\label{prop:higher order energy}
	For $K$ sufficiently large and then $s_0$ sufficiently large,
	\begin{equation}\label{eq:energy improvement}
		\|\eps(s)\|_{\dot{\calH}^{2L+1}}\le\tfrac 12 Ks^{-2Lc_1}\qquad(s_0\le s\le s_1).
	\end{equation}
\end{prop}

\begin{proof}
	We define the normalized higher-order variable
	\begin{equation}\label{eq:epsL projected}
		\mathfrak f\coloneqq s^{2Lc_1}\Pi\big[\eps_L-\chi(s^{-\frac{1}{4} c_1}r)(-iH)^{L-1}\Theta_{L+1}\big],\qquad
		\calF_L\coloneqq\tfrac12(H\mathfrak f,\mathfrak f).
	\end{equation}
	Here, $\Theta_{L+1}$ is the weight-$L+1$ profile error defined in \eqref{eq:Theta L1 def}.
	Since $\mathfrak f\in\Pi\dot H^1_{\rad}$, \eqref{eq:projection coercivity} gives $\calF_L\sim\|\mathfrak f\|_{\dot H^1}^2$.
	
	We first claim
	\begin{equation}\label{eq:combined coercivity}
		s^{4Lc_1}\|\eps\|_{\dot{\calH}^{2L+1}}^2\lesssim s^{4Lc_1}\|\eps_L\|_{\dot H^1}^2
		\lesssim\calF_L+s^{-2+\frac{5}{4} c_1}.
	\end{equation}
	Note that the first estimate comes from \eqref{eq:coercivity}.
	By \eqref{eq:adapted tests}, \eqref{eq:projection}, \eqref{eq:decomposition 2}, and \eqref{eq:eigenmode mode def},
	\begin{equation}\label{eq:continuous decomposition}
		\eps_L-s^{-2Lc_1}\mathfrak f=\Pi[\chi(s^{-\frac{1}{4} c_1}r)(-iH)^{L-1}\Theta_{L+1}]
		+\sum_{\iota=\pm}\alpha_\iota[Y_\iota-(\td\phi,Y_\iota)\Lambda W-\tfrac12(\td\varphi,Y_\iota)iW].
	\end{equation}
	For the first term in \eqref{eq:continuous decomposition}, we use \eqref{eq:lin iterated products}, \eqref{eq:correction pointwise}, \eqref{eq:mod para s bound}, and \eqref{eq:mod esti} to obtain
	\begin{equation}\label{eq:profile L plus 1 error}
		\begin{aligned}
			|\nabla^m(-iH)^{L-1}\Theta_{L+1}|&\lesssim_m s^{-L-1}\langle r\rangle^{-m},\\
			|\nabla^m(-iH)^{L-1}\partial_s\Theta_{L+1}|&\lesssim_m\big(s^{-L-2}+s^{-1}\|\eps\|_{\dot{\calH}^{2L+1}}\big)\langle r\rangle^{-m}.
		\end{aligned}
	\end{equation}
	Here, we used the fact that $\Theta_{L+1}$ has weight $L+1$ in the sense of \eqref{eq:weighted homogeneity}. Thus, we have
	\begin{equation}\label{eq:correction size}
		\|\Pi[\chi(s^{-\frac{1}{4} c_1}r)(-iH)^{L-1}\Theta_{L+1}]\|_{\dot H^1}
		\lesssim s^{-L-1+\frac{1}{8} c_1}.
	\end{equation}
	Using \eqref{eq:continuous decomposition}, \eqref{eq:correction size}, and $|\alpha_+|+|\alpha_-|\lesssim s^{-L-1}$ with $4Lc_1=2L+c_1$, we obtain
	\begin{align*}
		s^{4Lc_1}\|\eps_L\|_{\dot H^1}^2
		&\lesssim\|\mathfrak f\|_{\dot H^1}^2
		+s^{4Lc_1}\big\|\Pi[\chi(s^{-\frac{1}{4} c_1}r)(-iH)^{L-1}\Theta_{L+1}]\big\|_{\dot H^1}^2
		+s^{-2+c_1}\\
		&\lesssim\calF_L+s^{-2+\frac{5}{4} c_1},
	\end{align*}
	which proves \eqref{eq:combined coercivity}.

	To estimate $\calF_L$, we first derive the evolution of $\mathfrak f$. Differentiating \eqref{eq:epsL projected}, we get
	\begin{equation}\label{eq:epsL projected equ}
		\partial_s\mathfrak f=\Pi(-iH)\mathfrak f+\frac{\lambda_s}{\lambda}\Pi\Lambda\mathfrak f-\gamma_s\Pi(i\mathfrak f)
		+2L\bigg(\frac{\lambda_s}{\lambda}+\frac{c_1}{s}\bigg)\mathfrak f+F,
	\end{equation}
	where $F=F_1+F_2+F_3$ and
	\begin{align*}
		F_1&\coloneqq s^{2Lc_1}\Pi\big[-(1-\chi(s^{-\frac{1}{4} c_1}r))(-iH)^L\Theta_{L+1}
		+[-iH,\chi(s^{-\frac{1}{4} c_1}r)](-iH)^{L-1}\Theta_{L+1}\big],\\
		F_2&\coloneqq s^{2Lc_1}\tfrac{\lambda_s}{\lambda}\Pi(\Lambda+2L)(\eps_L-s^{-2Lc_1}\mathfrak f)
		-s^{2Lc_1}\gamma_s\Pi\big[i(\eps_L-s^{-2Lc_1}\mathfrak f)\big],\\
		F_3&\coloneqq s^{2Lc_1}\Pi\big[\td E_L-\partial_s\big(\chi(s^{-\frac{1}{4} c_1}r)(-iH)^{L-1}\Theta_{L+1}\big)\big].
	\end{align*}
	The projection removes the profiles $\Lambda W$ and $iW$ in \eqref{eq:epsL equation}. By \eqref{eq:projection bounds},
	\begin{equation}\label{eq:direct identity}
		\partial_s\calF_L=\frac{\lambda_s}{\lambda}(H\mathfrak f,\Lambda\mathfrak f)-\gamma_s(H\mathfrak f,i\mathfrak f)
		+4L\bigg(\frac{\lambda_s}{\lambda}+\frac{c_1}{s}\bigg)\calF_L+(H\mathfrak f,F).
	\end{equation}
	For the first two terms, integration by parts gives
	\begin{equation*}
		(H\mathfrak f,\Lambda\mathfrak f)=2(W^3\Lambda W(\mathfrak f+4\Re\mathfrak f),\mathfrak f), \qquad
		(H\mathfrak f,i\mathfrak f)=4(iW^4\Re\mathfrak f,\mathfrak f),
	\end{equation*}
	and hence
	\begin{equation}\label{eq:epsL projected Mor}
		|\tfrac{\lambda_s}{\lambda}(H\mathfrak f,\Lambda\mathfrak f)|+|\gamma_s(H\mathfrak f,i\mathfrak f)|\lesssim s^{-1}\|\mathfrak f\|_{\mathrm{Mor}}^2.
	\end{equation}

	For the third term in \eqref{eq:direct identity}, we use \eqref{eq:mod esti} and \eqref{eq:special}:
	\begin{equation}\label{eq:scale comparison}
		\frac{\lambda_s}{\lambda}+\frac{c_1}{s}
		=\bigg(\frac{\lambda_s}{\lambda}+b_1\bigg)-\bigg(b_1-\frac{c_1}{s}\bigg)
		=O_K(s^{-1-\kappa}).
	\end{equation}

	We estimate $|(H\mathfrak f,F)|$ and $\|F\|_{\dot H^1}$ term by term. For $F_1$, \eqref{eq:lin profile decay} and \eqref{eq:profile L plus 1 error}, with $2Lc_1=L+\frac{1}{2}c_1$, give
	\begin{equation*}
		\|F_1\|_{\dot H^1}\lesssim s^{-1+\frac{1}{8} c_1},\qquad
		|(H\mathfrak f,F_1)|\lesssim \|\langle r\rangle^2 HF_1\|_{L^2}\|\mathfrak f\|_{\mathrm{Mor}}\lesssim s^{-1+\frac{3}{8} c_1}\|\mathfrak f\|_{\mathrm{Mor}}.
	\end{equation*}
	For $F_2$, we use \eqref{eq:continuous decomposition}. For the first part of \eqref{eq:continuous decomposition}, we apply \eqref{eq:profile L plus 1 error}. For the profile parts, by \eqref{eq:lin commutator coefficients}, \eqref{eq:energy kernel}, and \eqref{eq:eigenfunctions}, we get
	\begin{equation}\label{eq:frak f nonlin profile}
		|H(\Lambda+2L)\psi|+|H(i\psi)|\lesssim\langle r\rangle^{-5},\qquad \psi=Y_+,Y_-,\Lambda W,iW.
	\end{equation}
	Using \eqref{eq:profile L plus 1 error}, \eqref{eq:frak f nonlin profile} with $|\alpha_+|+|\alpha_-|\lesssim s^{-L-1}$, and \eqref{eq:scale phase bound}, we obtain
	\begin{equation*}
		\|F_2\|_{\dot H^1}\lesssim s^{-2+\frac{5}{8} c_1},\qquad
		|(H\mathfrak f,F_2)|\lesssim s^{-2+\frac{7}{8} c_1}\|\mathfrak f\|_{\mathrm{Mor}}.
	\end{equation*}
	For $F_3$, combining \eqref{eq:nonlin esti}, \eqref{eq:nonlin esti Mor}, and \eqref{eq:profile L plus 1 error}, we derive
	\begin{align*}
		\|F_3\|_{\dot H^1}
		&\lesssim s^{2Lc_1}\|\td E_L\|_{\dot H^1}
		+s^{2Lc_1+\frac{1}{8} c_1}(s^{-L-2}+s^{-1}\|\eps\|_{\dot{\calH}^{2L+1}})
		\lesssim_K s^{-1+\frac{1}{8} c_1},
		\\
		|(H\mathfrak f,F_3)|
		&\lesssim s^{2Lc_1}|(H\mathfrak f,\td E_L)|+s^{2Lc_1+\frac{3}{8} c_1}(s^{-L-2}+s^{-1}\|\eps\|_{\dot{\calH}^{2L+1}})\|\mathfrak f\|_{\mathrm{Mor}}
		\\
		&\lesssim_K s^{-1+\frac{3}{8} c_1}\|\mathfrak f\|_{\mathrm{Mor}}
		+s^{-1-\frac{1}{4L-1}}\|\mathfrak f\|_{\dot H^1}.
	\end{align*}
	Thus, we have
	\begin{equation}\label{eq:energy boots nonlinear esti goal}
		\|F\|_{\dot H^1}\lesssim_K s^{-1+\frac{1}{8} c_1},
		\qquad
		|(H\mathfrak f,F)|\lesssim_K s^{-1+\frac{3}{8} c_1}\|\mathfrak f\|_{\mathrm{Mor}}+s^{-1-\frac{1}{2L}c_1}.
	\end{equation}
	Here, we used \eqref{eq:dot calH linear map} and \eqref{eq:correction size} to obtain $\|\mathfrak f\|_{\dot H^1}\lesssim_K1$. 
	Integrating \eqref{eq:direct identity}, we use \eqref{eq:epsL projected Mor}, \eqref{eq:scale comparison}, and \eqref{eq:energy boots nonlinear esti goal} with $\calF_L\lesssim_K1$ and $\kappa<\frac{c_1}{2L}$ to deduce
	\begin{equation}\label{eq:higher energy calF esti}
		|\calF_L(s_1)-\calF_L(s_0)|
		\lesssim\int_{s_0}^{s_1}s^{-1}\|\mathfrak f\|_{\mathrm{Mor}}^2 ds
		+C_K\int_{s_0}^{s_1}s^{-1+\frac{3}{8} c_1}\|\mathfrak f\|_{\mathrm{Mor}}ds+C_Ks_0^{-\kappa}.
	\end{equation}

	We estimate the two integrals in \eqref{eq:higher energy calF esti} using Proposition~\ref{prop:global spacetime}. We first assume $s_1\le(1+\delta)s_0$, where $\delta>0$ is sufficiently small. We undo the scaling in \eqref{eq:epsL projected equ} using the time $t$ in \eqref{eq:renormal}. We define $\mu$ and use \eqref{eq:scale phase bound} to obtain
	\begin{equation}\label{eq:energy boots mu}
		\mu\coloneqq\frac{\lambda}{\lambda_0},\qquad \mu(s_0)=1, \qquad
		|\mu-1|\lesssim\delta,\qquad \frac{\mu_s}{\mu}=\frac{\lambda_s}{\lambda} \qquad\frac{dt}{\lambda_0^2}=\mu^2 ds\sim ds.
	\end{equation}
	Using $\Pi(-iH)(1-\Pi)=0$, we have
	\begin{equation*}
		\mu^{2}[\lambda_0^2\partial_t-\Pi(-iH)]\Pi(\mathfrak f_\mu)
		=\lambda^2\partial_t\Pi(\mathfrak f_\mu)-\mu^2\Pi(-iH)(\mathfrak f_\mu).
	\end{equation*}
	Applying the scaling to \eqref{eq:epsL projected equ}, we obtain
	\begin{equation}\label{eq:rescaled equation}
		\begin{aligned}
			\mu^{2}&[\lambda_0^2\partial_t-\Pi(-iH)]\Pi(\mathfrak f_\mu)=\lambda^2\partial_t\Pi(\mathfrak f_\mu)-\mu^2\Pi(-iH)(\mathfrak f_\mu)
			\\
			=&\Pi\big[(-iH\mathfrak f)_\mu-\mu^2(-iH)(\mathfrak f_\mu)\big]
			-\Pi\big[\{(1-\Pi)(-iH\mathfrak f+\tfrac{\lambda_s}{\lambda}\Lambda\mathfrak f)\}_\mu\big]\\
			&-\gamma_s\Pi\big[(\Pi(i\mathfrak f))_\mu\big]
			+2L(\tfrac{\lambda_s}{\lambda}+\tfrac{c_1}{s})\Pi(\mathfrak f_\mu)+\Pi(F_\mu).
		\end{aligned}
	\end{equation}
	We estimate the first two terms on the right in the form required by \eqref{eq:retarded source}. For the first term,
	\begin{gather*}
		(-iH\mathfrak f)_\mu-\mu^2(-iH)(\mathfrak f_\mu)
		=i\big[W^4(x/\mu)-\mu^2W^4(x)\big](\mathfrak f_\mu+4\Re\mathfrak f_\mu),\\
		|\nabla^j[W^4(x/\mu)-\mu^2W^4(x)]|
		\lesssim\delta\langle r\rangle^{-4-j}\qquad(j=0,1).
	\end{gather*}
	For the second term, writing $(1-\Pi)f=\sum_{\psi} c_{\psi,f}\psi$ and using $\Pi\psi=0$ for $\psi=Y_+,Y_-,\Lambda W,iW$, we obtain
	\begin{align*}
		\Pi\big[\big((1-\Pi)f\big)_\mu\big]
		=\sum_{\psi} c_{\psi,f}\Pi(\psi_\mu)=\sum_{\psi} c_{\psi,f}\Pi(\psi_\mu-\psi).
	\end{align*}
	The decay of these profiles and the identities $\Delta(iW)=-iW^5, \Delta(\Lambda W)=-5W^4\Lambda W$ yield
	\begin{align*}
		|\nabla^j(\psi_\mu-\psi)|\lesssim\delta\langle r\rangle^{-1-j},\quad
		|\Delta(\psi_\mu-\psi)|\lesssim\delta\langle r\rangle^{-5},\quad
		\|\langle r\rangle^2H(\psi_\mu-\psi)\|_{L^2}\lesssim\delta.
	\end{align*}
	Moreover, by the definition of $\Pi$ \eqref{eq:projection} with $f=-iH\mathfrak f+\tfrac{\lambda_s}{\lambda}\Lambda\mathfrak f$, we get $|c_{\psi,f}|\lesssim \|\mathfrak f\|_{\mathrm{Mor}}$.
	Thus, for $g\in\Pi\dot H^1_{\rad}$, we estimate the first two terms on the right side of \eqref{eq:rescaled equation}:
	\begin{equation}\label{eq:rescaled equation error esti 1}
		\begin{aligned}
			&\big|(Hg,\Pi[(-iH\mathfrak f)_\mu-\mu^2(-iH)(\mathfrak f_\mu)])\big|
			+|(Hg,\Pi\big[((1-\Pi)(-iH\mathfrak f+\tfrac{\lambda_s}{\lambda}\Lambda\mathfrak f))_\mu\big])|
			\\
			&\lesssim\delta\|g\|_{\mathrm{Mor}}\|\mathfrak f\|_{\mathrm{Mor}}.
		\end{aligned}
	\end{equation}
	For the remaining three terms, we have
	\begin{equation}\label{eq:rescaled equation error esti 2}
		\|-\gamma_s\Pi\big[(\Pi(i\mathfrak f))_\mu\big]
		+2L(\tfrac{\lambda_s}{\lambda}+\tfrac{c_1}{s})\Pi(\mathfrak f_\mu)+\Pi(F_\mu)\|_{\dot H^1}\lesssim_K s^{-1+\frac{1}{8}c_1}.
	\end{equation}
	Here we used \eqref{eq:scale phase bound}, \eqref{eq:energy boots nonlinear esti goal}, and $\|\mathfrak f\|_{\dot H^1}\lesssim_K1$.
	
	Moving the scaling to the tests in \eqref{eq:projection} and using $\Pi\mathfrak f=\mathfrak f$ with $|\mu-1|\lesssim \delta$, we have $\|(1-\Pi)\mathfrak f_\mu\|_{\mathrm{Mor}}\lesssim\delta\|\mathfrak f\|_{\mathrm{Mor}}$. By this, we obtain
	\begin{equation*}
		\|\mathfrak f\|_{\mathrm{Mor}}\sim \|\mathfrak f_\mu\|_{\mathrm{Mor}}
		\leq \delta\|\mathfrak f\|_{\mathrm{Mor}}+\|\Pi(\mathfrak f_\mu)\|_{\mathrm{Mor}}
		\lesssim\|\Pi(\mathfrak f_\mu)\|_{\mathrm{Mor}}.
	\end{equation*}
	We apply Duhamel's formula to \eqref{eq:rescaled equation} with time variable $t/\lambda_0^2$. Since $dt/\lambda_0^2\sim ds$ by \eqref{eq:energy boots mu}, estimates \eqref{eq:rescaled equation error esti 1} and \eqref{eq:rescaled equation error esti 2}, together with Proposition~\ref{prop:global spacetime}, yield
	\begin{align*}
		\int_{s_0}^{s_1}\|\mathfrak f\|_{\mathrm{Mor}}^2 ds
		\lesssim\|\mathfrak f(s_0)\|_{\dot H^1}^2 +C_K\bigg(\int_{s_0}^{s_1}s^{-1+\frac{1}{8}c_1}ds\bigg)^2
		+\delta\int_{s_0}^{s_1}\|\mathfrak f\|_{\mathrm{Mor}}^2 ds.
	\end{align*}
	Using $\|\mathfrak f(s_0)\|_{\dot H^1}\lesssim_K1$ and $\int_{s_0}^{s_1}s^{-1+\frac{1}{8}c_1}ds\lesssim\delta s_0^{\frac{1}{8}c_1}$, we absorb the last term for $\delta$ sufficiently small and obtain
	\begin{equation}\label{eq:local spacetime}
		\int_{s_0}^{s_1}\|\mathfrak f\|_{\mathrm{Mor}}^2 ds\lesssim_K s_0^{\frac{1}{4} c_1}
		\qquad\big(s_1\le(1+\delta)s_0\big).
	\end{equation}
	For general $s_1$, we apply \eqref{eq:local spacetime} on $[(1+\delta)^js_0,\min\{(1+\delta)^{j+1}s_0,s_1\}]$ whenever $(1+\delta)^js_0<s_1$. The implicit constant is independent of $j$, and hence
	\begin{equation}\label{eq:higher energy Mor esti}
		\int_{s_0}^{s_1}s^{-\frac{1}{2}+\frac{1}{4} c_1}\|\mathfrak f\|_{\mathrm{Mor}}^2 ds
		\lesssim_K\sum_{j\ge0}\big((1+\delta)^js_0\big)^{-\frac{1}{2}+\frac{1}{2} c_1}
		\lesssim_K s_0^{-\frac{1-c_1}{2}}.
	\end{equation}
	Using $s\ge s_0$ and Cauchy--Schwarz with $c_1\le\frac{2}{3}$ and $\kappa<\frac{1}{6}$, \eqref{eq:higher energy calF esti} and \eqref{eq:higher energy Mor esti} give
	\begin{equation}\label{eq:higher summed energy}
		\calF_L(s)\le\calF_L(s_0)+C_Ks_0^{-\kappa}\qquad(s_0\le s\le s_1).
	\end{equation}
	Thanks to the initial estimate \eqref{eq:initial bounds} with \eqref{eq:correction size}, we arrive at
	\begin{equation}\label{eq:higher energy init control}
		\calF_L(s_0)\lesssim s_0^{4Lc_1}\|\eps(s_0)\|_{\dot{\calH}^{2L+1}}^2+s_0^{-2+\frac{5}{4} c_1}
		\lesssim_L\delta_{\mathrm c}^{-1}+s_0^{-2+\frac{5}{4} c_1}.
	\end{equation}
	Using \eqref{eq:combined coercivity}, \eqref{eq:higher summed energy}, and \eqref{eq:higher energy init control}, we conclude
	\begin{equation}
		s^{4Lc_1}\|\eps(s)\|_{\dot{\calH}^{2L+1}}^2
		\le C_L\delta_{\mathrm c}^{-1}+C_Ks_0^{-\kappa}\qquad(s_0\le s\le s_1).
	\end{equation}
	We choose $K$ and then $s_0$ sufficiently large that
	\begin{equation}\label{eq:parameter K s0 choice}
		C_L(\delta_{\mathrm c}^{-1}+1)\le\tfrac{1}{16}K^2,\qquad C_Ks_0^{-\kappa}\le \tfrac{1}{16}K^2.
	\end{equation}
	These choices are independent of $s_1$ and prove \eqref{eq:energy improvement}.
\end{proof}

\subsection{Closing the bootstrap}\label{sec:main bootstrap}

We first improve the remainder and stable bounds in \eqref{eq:main bootstrap goal}, then prove the continuity of $\Phi$ needed in the proof of Proposition~\ref{prop:trapped}. We finally deduce Theorem~\ref{thm:main}.

\begin{proof}[Proof of Proposition~\ref{prop:main bootstrap}]
	First, the estimate for $\|\eps(s)\|_{\dot{\calH}^{2L+1}}$ in \eqref{eq:main bootstrap goal} follows from Proposition~\ref{prop:higher order energy}.
	We next improve the estimate on $\td\eps$. We have
	\begin{equation*}
		|E(W+\td\eps)-E(W)-2\pi(H\td\eps,\td\eps)|
		\lesssim\|W\|_{L^6}^{3}\|\td\eps\|_{L^6}^3+\|\td\eps\|_{L^6}^6
		\lesssim\delta_{\mathrm{dec}}\|\td\eps\|_{\dot H^1}^2.
	\end{equation*}
	From this estimate with energy conservation, the scaling, and \eqref{eq:definition H},
	we absorb the last term to obtain
	\begin{equation*}
		\|\td\eps\|_{\dot H^1}^2
		\lesssim|E(u_0)-E(W)|+\int_0^\infty W^4|\td\eps|^2r^2dr
		\lesssim \|\td\eps(s_0)\|_{\dot H^1}^2+\|\chi_R\td\eps\|_{L^2}^2+R^{-2}\|\td\eps\|_{\dot H^1}^2.
	\end{equation*}
	With $R=100\delta_{\mathrm{dec}}^{-1}$, combining \eqref{eq:initial bounds}, the decomposition \eqref{eq:w decom}, \eqref{eq:kernel growth}, \eqref{eq:correction pointwise}, \eqref{eq:mod para s bound}, and the parameter dependence \eqref{eq:parameter dependence}, we derive
	\begin{equation*}
		\|\td\eps(s)\|_{\dot H^1}^2\le C\delta_{\mathrm c}^{4L-1}+C\delta_{\mathrm{dec}}^{-4L-2}(s^{-2}+K^2s^{-4Lc_1})
		+\tfrac{1}{10^2}\delta_{\mathrm{dec}}^2
		\le\tfrac{1}{16}\delta_{\mathrm{dec}}^2.
	\end{equation*}

	For $(\calA_b)_1$ in \eqref{eq:main bootstrap goal}, using \eqref{eq:refined mod compare normalize}, \eqref{eq:refined mod esti normalize}, and \eqref{eq:ODE diagonal}, we get
	\begin{equation}\label{eq:b1 closing}
		s\partial_s(\td \calA_b)_1+(\td \calA_b)_1=O_K(s^{-2\kappa}+s^{-\kappa'}).
	\end{equation}
	By \eqref{eq:init para stable} and \eqref{eq:refined mod compare normalize}, $s_0^\kappa|(\td \calA_b)_1(s_0)|\le\frac14+C_Ks_0^{-\kappa'-\kappa}$. Integrating \eqref{eq:b1 closing} and using $2\kappa<\kappa'<1$, we obtain
	\begin{align*}
		s^\kappa|(\calA_b)_1(s)|\le& s^\kappa|(\td \calA_b)_1(s)|+C_Ks_0^{-\kappa}\\
		\le& (\tfrac14+C_Ks_0^{-\kappa'-\kappa})(\tfrac{s_0}s)^{1-\kappa}
		+C_K(s^{-\kappa}+s^{-\kappa'+\kappa})+C_Ks_0^{-\kappa}\leq \tfrac12.
	\end{align*}

	Finally, for $\alpha_-$, we integrate \eqref{eq:eigenmode esti} using \eqref{eq:main bootstrap assump} to obtain
	\begin{equation*}
		|\alpha_-|\le e^{-e_0(s-s_0)}|\alpha_-(s_0)|+Cs^{-L-1}.
	\end{equation*}
	From $\epsilon_{\mathrm h}\ll1$ and \eqref{eq:init para stable}, we conclude
	\begin{equation*}
		|\epsilon_{\mathrm h}s^{L+1}\alpha_-(s)|\leq \tfrac12. \qedhere
	\end{equation*}
\end{proof}

\begin{proof}[Proof of Lemma~\ref{lem:exit map}]
	Recall the definition of $\bfA$ and $\td{\bfA}$ in \eqref{eq:fast coordinates}.
	We first consider $\bfA_{\mathrm{init}}$ such that $\td{\bfA}(s_0;\bfA_{\mathrm{init}})\in(-1/2,1/2)^{2L}$. By continuous dependence and Lemma~\ref{lem:decomposition}, it suffices to show that $\bfA_{\mathrm{init}}\mapsto s_{\mathrm{exit}}(\bfA_{\mathrm{init}})$ is continuous.

	By \eqref{eq:ODE diagonal}, \eqref{eq:main bootstrap assump}, \eqref{eq:refined mod compare normalize}, and \eqref{eq:refined mod esti normalize}, we obtain the following estimates for $\td\calA_b$ with $2\le k\le L$ and $\td\calA_\eta$ with $1\le k\le L$:
	\begin{align*}
		\frac{d}{ds}|s^\kappa(\td\calA_b)_k|^2
		&=\frac2s\left(\kappa+\frac{k}{4L-1}\right)|s^\kappa(\td\calA_b)_k|^2
		+O_K(s^{-1-\kappa}+s^{-1-\kappa'+\kappa}),\\
		\frac{d}{ds}|s^\kappa(\td\calA_\eta)_k|^2
		&=\frac2s\left(\kappa+\frac{k}{4L-1}\right)|s^\kappa(\td\calA_\eta)_k|^2
		+O_K(s^{-1-\kappa}+s^{-1-\kappa'+\kappa}).
	\end{align*}
	By \eqref{eq:main bootstrap assump} and \eqref{eq:eigenmode esti}, we also have
	\begin{equation*}
		\frac{d}{ds}|\epsilon_{\mathrm h}s^{L+1}\alpha_+|^2
		=2\left(e_0+\frac{L+1}{s}\right)|\epsilon_{\mathrm h}s^{L+1}\alpha_+|^2+O(\epsilon_{\mathrm h}).
	\end{equation*}
	From \eqref{eq:parameter dependence} and $\kappa'>\kappa$, we deduce, for $1\le j\le2L$,
	\begin{equation}\label{eq:boundary transversality}
		\frac{d}{ds}|\td{\bfA}_j(s)|^2\gtrsim s^{-1}>0
		\quad\text{whenever}\quad \frac14\le|\td{\bfA}_j(s)|\le\frac34.
	\end{equation}

	By \eqref{eq:initial map error} and \eqref{eq:main bootstrap goal}, the bounds in \eqref{eq:main bootstrap assump} are strict at $s_{\mathrm{exit}}(\bfA_{\mathrm{init}})$.  Thus, continuous dependence and Lemma~\ref{lem:decomposition} extend the nearby solutions and decompositions to a common interval beyond this exit time, with \eqref{eq:main bootstrap assump}. By \eqref{eq:boundary transversality}, $\td{\bfA}(s)\in(-1/2,1/2)^{2L}$ for $s_0\le s<s_{\mathrm{exit}}$. If $|\td{\bfA}_j(s_{\mathrm{exit}})|=1/2$, then $|\td{\bfA}_j(s)|>1/2$ for $s>s_{\mathrm{exit}}$ sufficiently close to $s_{\mathrm{exit}}$. By continuous dependence, $s_{\mathrm{exit}}(\bfA_{\mathrm{init}}')\to s_{\mathrm{exit}}(\bfA_{\mathrm{init}})$, and hence $\Phi(\bfA_{\mathrm{init}}')\to\Phi(\bfA_{\mathrm{init}})$ as $\bfA_{\mathrm{init}}'\to\bfA_{\mathrm{init}}$.

	If $\td{\bfA}(s_0;\bfA_{\mathrm{init}})\in\partial[-1/2,1/2]^{2L}$, then \eqref{eq:init para stable}, \eqref{eq:initial bounds}, and \eqref{eq:initial map error} make the conditions in \eqref{eq:main bootstrap assump} strict at $s_0$. By continuous dependence and Lemma~\ref{lem:decomposition}, \eqref{eq:main bootstrap assump} holds on a common interval after $s_0$ for nearby initial choices. Choose $1\le j\le2L$ such that $|\td{\bfA}_j(s_0;\bfA_{\mathrm{init}})|=1/2$. By \eqref{eq:boundary transversality}, for any fixed $s_1>s_0$ sufficiently close to $s_0$, we have
	\begin{equation*}
		|\td{\bfA}_j(s_1;\bfA_{\mathrm{init}})|>1/2.
	\end{equation*}
	By continuous dependence, for $\bfA_{\mathrm{init}}'$ sufficiently close to $\bfA_{\mathrm{init}}$ with $\td{\bfA}(s_0;\bfA_{\mathrm{init}}')\in(-1/2,1/2)^{2L}$, we have
	\begin{equation*}
		|\td{\bfA}_j(s_1;\bfA_{\mathrm{init}}')|>1/2,\qquad
		s_0<s_{\mathrm{exit}}(\bfA_{\mathrm{init}}')<s_1.
	\end{equation*}
	Letting $s_1\to s_0$, we obtain
	\begin{equation*}
		s_{\mathrm{exit}}(\bfA_{\mathrm{init}}')\to s_0,\qquad
		\Phi(\bfA_{\mathrm{init}}')\to2\td{\bfA}(s_0;\bfA_{\mathrm{init}})
	\end{equation*}
	as $\bfA_{\mathrm{init}}'\to\bfA_{\mathrm{init}}$ with $\td{\bfA}(s_0;\bfA_{\mathrm{init}}')\in(-1/2,1/2)^{2L}$.
	
	When $\max_{1\le j\le2L}|\td{\bfA}_j(s_0;\bfA_{\mathrm{init}})|\ge1/2$, the definition of $\Phi$ \eqref{eq:exit map} directly implies the continuity.
	Hence, $\Phi$ is continuous.
\end{proof}

We establish an estimate for $F(w)$ needed in the proof of Theorem~\ref{thm:main}.
\begin{lem}
	We have
	\begin{equation}\label{eq:energy error}
		\left\|F(w)-b_1\Lambda W-\tfrac{i}{2}\eta_1W\right\|_{\dot H^1}
		\lesssim s^{-2c_1}.
	\end{equation}
\end{lem}
\begin{proof}
	Set $R\coloneqq s^{c_1}$. From \eqref{eq:mod para s bound}, \eqref{eq:nonlin pointwise derivatives}, and \eqref{eq:main bootstrap assump}, we deduce
	\begin{equation}\label{eq:enegy error 1}
		\chi_R(|\eps|+|W+P|+\langle r\rangle|\nabla(W+P)|)\lesssim\langle r\rangle^{-\frac12}.
	\end{equation}
	Combining $\td\eps=P+\eps$, \eqref{eq:notation nonlinearity}, and \eqref{eq:enegy error 1}, we get
	\begin{equation}\label{eq:enegy error 2}
		\chi_R|\nabla(\calN_W(\td\eps)-\calN_W(P))|
		\lesssim\langle r\rangle^{-2}|\nabla\eps|+\langle r\rangle^{-3}|\eps|.
	\end{equation}
	Since $F(w)-F(W+P)=-iH\eps+\calN_W(\td\eps)-\calN_W(P)$, \eqref{eq:dot calH def}, \eqref{eq:main bootstrap assump}, and \eqref{eq:enegy error 2} give
	\begin{equation}\label{eq:enegy error goal2}
		\|\chi_R\nabla(F(w)-F(W+P))\|_{L^2}
		\lesssim R^{2L-2}\|\eps\|_{\dot{\calH}^{2L+1}}\lesssim s^{-2c_1}.
	\end{equation}
	By \eqref{eq:nonlinearity} and \eqref{eq:Psi td def}, we have
	\begin{equation*}
		F(W+P)=\mathbf{Mod}_0\cdot\textbf{v}-\td\Psi.	
	\end{equation*}
	Using \eqref{eq:tail cancellations}, \eqref{eq:correction pointwise}, and \eqref{eq:mod para s bound}, we obtain
	\begin{equation}\label{eq:enegy error goal1}
		\begin{aligned}
			&\|\chi_R\nabla(F(W+P)-b_1\Lambda W-\tfrac{i}{2}\eta_1W)\|_{L^2}\\
			&\lesssim s^{-L}R^{2L-\frac52}+s^{-L-1}R^{2L-\frac32}+s^{-5L}R^{10L-\frac92}
			\lesssim s^{-2c_1}.
		\end{aligned}
	\end{equation}
	
	Now, we estimate the exterior part.
	By \eqref{eq:main bootstrap assump} and \eqref{eq:exterior w}, interpolation and Sobolev embedding yield
	\begin{equation*}
		\|(1-\chi_{\frac{R}2})w\|_{\dot H^3}\lesssim s^{-2c_1},\qquad \|(1-\chi_{\frac{R}2})w\|_{L^\infty}\lesssim s^{-c_1/2}.
	\end{equation*}
	Since $(1-\chi_{\frac{R}2})w=w$ on the support of $1-\chi_R$, we obtain
	\begin{equation}\label{eq:enegy error goal3}
		\|(1-\chi_R)\nabla F(w)\|_{L^2}\lesssim s^{-2c_1}.
	\end{equation}
	In addition, we have
	\begin{equation}\label{eq:enegy error goal4}
		\left\|(1-\chi_R)\nabla\left(b_1\Lambda W+\tfrac{i}{2}\eta_1W\right)\right\|_{L^2}
		\lesssim s^{-1}R^{-\frac12}\le s^{-2c_1}.
	\end{equation}
	Collecting \eqref{eq:enegy error goal1}, \eqref{eq:enegy error goal2}, \eqref{eq:enegy error goal3}, and \eqref{eq:enegy error goal4}, we conclude the proof.
\end{proof}

Finally, we finish the proof of the main theorem.

\begin{proof}[Proof of Theorem~\ref{thm:main}]
	We apply Proposition~\ref{prop:trapped} with $\delta_{\mathrm{dec}}<\delta^*$ and $\eps_{\mathrm{in}}\in C^\infty_{c,\rad}\cap\calZ_{\mathrm h}^\perp$ satisfying \eqref{eq:init radiation}. Then, the corresponding initial data is smooth. From \eqref{eq:renormal}, we have
	\begin{equation}\label{eq:t variable}
		u(t)=[w(s)]_{\lambda(s),\gamma(s)}|_{s=s(t)},\qquad \frac{dt}{ds}=\lambda^2(s)
	\end{equation}
	By Lemma~\ref{lem:decomposition}, \eqref{eq:fluctuations}, \eqref{eq:eigen coordinates}, and \eqref{eq:main bootstrap assump}, the trapped solution satisfies
	\begin{equation}\label{eq:thm proof input}
		\|w-W\|_{\dot H^1}\le\delta_{\mathrm{dec}}/2,\quad \|\eps\|_{\dot{\calH}^{2L+1}}\lesssim s^{-2Lc_1},\quad
		|s^{j}b_j-c_j|+|s^{j}\eta_j|\lesssim s^{-\kappa}.
	\end{equation}
	Thus, \eqref{eq:t variable} implies \eqref{eq:thm tube stability}.

	We next determine the blow-up rate. Integrating \eqref{eq:scale comparison}, we obtain $\ell_0>0$ such that
	\begin{equation}\label{eq:lambda s asymptotic}
		|\log(s^{c_1}\lambda(s))-\log\ell_0|
		\lesssim\int_s^\infty\sigma^{-1-\kappa}\,d\sigma\lesssim s^{-\kappa},\quad 
		\lambda(s)=\ell_0s^{-c_1}(1+O(s^{-\kappa}))
	\end{equation}
	Since $2c_1-1=1/(4L-1)>0$, from \eqref{eq:t variable}, we have
	\begin{equation}
		T=\int_{s_0}^\infty\lambda(s)^2\,ds<\infty,\qquad
		T-t(s)=(4L-1)\ell_0^2s^{1-2c_1}(1+O(s^{-\kappa})).
	\end{equation}
	Eliminating $s$, we obtain
	\begin{equation}
		\lambda(t)=\ell(T-t)^{2L}(1+o_{t\to T}(1)),\qquad
		\ell\coloneqq \big[(4L-1)^{2L}\ell_0^{4L-1}\big]^{-1}.
	\end{equation}
	For the phase $\gamma$, \eqref{eq:mod esti} and \eqref{eq:thm proof input} give
	\begin{equation}
		|\gamma_s|\le\left|\gamma_s-\tfrac{1}{2}\eta_1\right|+\tfrac{1}{2}|\eta_1|
		\lesssim s^{-2Lc_1}+s^{-1-\kappa}\lesssim s^{-1-\kappa},
	\end{equation}
	which implies $\gamma=\gamma_*+o_{t\to T}(1)$ for some $\gamma_*\in\mathbb R$. Hence we conclude \eqref{eq:thm main law}.

	We next prove \eqref{eq:thm main remainder}. By \eqref{eq:energy error}, \eqref{eq:mod esti}, \eqref{eq:thm proof input}, and \eqref{eq:lambda s asymptotic},
	\begin{equation}\label{eq:remainder derivative}
		\|F(w)+\tfrac{\lambda_s}{\lambda}\Lambda W-i\gamma_sW\|_{\dot H^1}
		\lesssim s^{-2c_1}\sim\lambda^2.
	\end{equation}
	Returning to the original variables with $s=s(t)$, we use \eqref{eq:remainder derivative} to obtain
	\begin{equation}\label{eq:remainder with scale}
		\partial_t\td\eps_{\lambda,\gamma}
		=\lambda^{-2}\left[F(w)+\tfrac{\lambda_s}{\lambda}\Lambda W-i\gamma_sW\right]_{\lambda,\gamma},\qquad
		\|\partial_t\td\eps_{\lambda,\gamma}\|_{\dot H^1}\lesssim1.
	\end{equation}
	By \eqref{eq:remainder with scale} and \eqref{eq:thm proof input}, $\td\eps_{\lambda,\gamma}(t)$ converges to some $u^*\in\dot H^1_{\rad}$ with
	\begin{equation}\label{eq:moving limit}
		\|\td\eps_{\lambda,\gamma}(t)-u^*\|_{\dot H^1}\lesssim T-t,\qquad
		\|u^*\|_{\dot H^1}\le\delta_{\mathrm{dec}},
	\end{equation}
	which proves $\dot H^1$ convergence of \eqref{eq:thm main remainder}.

	We next prove the $L^2$ convergence. By \eqref{eq:thm proof input}, $u$ is uniformly bounded in $\dot H^1$. Thus, Sobolev embedding implies
	\begin{equation*}
		\|\partial_tu\|_{H^{-1}}
		=\|\Delta u+|u|^4u\|_{H^{-1}}
		\lesssim\|u\|_{\dot H^1}+\|u\|_{L^6}^5\lesssim1.
	\end{equation*}
	By \eqref{eq:conservation laws} and the uniform boundedness of $u$ in $\dot H^1$, we arrive at
	\begin{equation}
		\|u(t)-u(t')\|_{L^2}^2
		\le\|u(t)-u(t')\|_{H^{-1}}\|u(t)-u(t')\|_{H^1}\lesssim|t-t'|.
	\end{equation}
	Hence $u(t)$ converges in $L^2$. Moreover, for any $R>0$, we have $\|\chi_RW_\lambda\|_{L^2}^2\lesssim \lambda R \to 0$. This implies $u\to u^*$ and $u^*\in H^1_{\rad}$. Therefore, we conclude \eqref{eq:thm main remainder}.

	Finally, we prove the regularity of $u^*$ and \eqref{eq:thm origin asymptotics}. For each fixed $A>0$, collecting \eqref{eq:kernel growth}, \eqref{eq:profile form}, \eqref{eq:correction pointwise}, \eqref{eq:mod para s bound}, and \eqref{eq:lambda s asymptotic}, we derive
	\begin{equation}\label{eq:profile rescaling bounds}
		\left\|\chi_A(P-b_L\mathsf V_L-\eta_L\mathsf U_L)_{\lambda,\gamma}\right\|_{L^\infty}
		\lesssim_A s^{-(2c_1-1)}+\lambda\to0.
	\end{equation}
	By \eqref{eq:kernel tails}, \eqref{eq:thm proof input}, and \eqref{eq:lambda s asymptotic}, $(b_L\mathsf V_L+\eta_L\mathsf U_L)_{\lambda,\gamma}\to a|x|^{2L-1}$ in $L^2_{\mathrm{loc}}$, where $a\coloneqq e^{i\gamma_*}C_Lc_L\ell_0^{\frac12-2L}\ne0$. Combining this convergence with \eqref{eq:moving limit} and \eqref{eq:profile rescaling bounds}, we obtain
	\begin{equation}\label{eq:physical profile limit}
		P_{\lambda,\gamma}(t)\to a|x|^{2L-1},\qquad
		\eps_{\lambda,\gamma}(t)\to u^*-a|x|^{2L-1}
		\quad\text{in }L^2_{\mathrm{loc}}.
	\end{equation}
	By \eqref{eq:thm proof input} with \eqref{eq:lambda s asymptotic} and \eqref{eq:dot cal pointwise}, we obtain
	\begin{equation}\label{eq:critical norm}
		\begin{gathered}
			\|\eps_{\lambda,\gamma}(t)\|_{\dot H^{2L+1}}
			=\lambda(t)^{-2L}\|\eps(s(t))\|_{\dot H^{2L+1}}\lesssim1,\\
			|\eps_{\lambda,\gamma}(t,x)|\lesssim\lambda^{2L-\frac12}\langle x/\lambda\rangle^{2L-\frac12}
			=(\lambda^2+|x|^2)^{L-\frac14}.
		\end{gathered}
	\end{equation}
	By \eqref{eq:critical norm}, $\nabla^{2L+1}\eps_{\lambda,\gamma}(t)$ is uniformly bounded in $L^2$ and hence converges weakly in $L^2$, up to a subsequence. By \eqref{eq:physical profile limit}, the weak limit is $\nabla^{2L+1}(u^*-a|x|^{2L-1})$. Therefore, weak lower semicontinuity yields $\|\nabla^{2L+1}(u^*-a|x|^{2L-1})\|_{L^2}\lesssim1$. Since $\nabla^{2L+1}((1-\chi)|x|^{2L-1})\in L^2$ and $u^*\in L^2$, we conclude
	\begin{equation}\label{eq:asymptotic profile expansion}
		u^*-a\chi|x|^{2L-1}\in H^{2L+1}.
	\end{equation}
	Passing to the limit along a subsequence in the pointwise bound of \eqref{eq:critical norm}, we conclude \eqref{eq:thm origin asymptotics}.
	Finally, since $\chi|x|^{2L-1}\in H^{2L+\frac12-}$ by Sobolev embedding, \eqref{eq:asymptotic profile expansion} yields $u^*\in H^{2L+\frac12-}_{\rad}$.
\end{proof}

\bibliographystyle{abbrv}
\bibliography{reference}

@article {Dodson2019ASENS,
	AUTHOR = {Dodson, Benjamin},
	TITLE = {Global well-posedness and scattering for the focusing, cubic {S}chr\"odinger equation in dimension {$d=4$}},
	JOURNAL = {Ann. Sci. \'Ec. Norm. Sup\'er. (4)},
	FJOURNAL = {Annales Scientifiques de l'\'Ecole Normale Sup\'erieure. Quatri\`eme S\'erie},
	VOLUME = {52},
	YEAR = {2019},
	NUMBER = {1},
	PAGES = {139--180},
	ISSN = {0012-9593,1873-2151},
	DOI = {10.24033/asens.2385},
	URL = {https://doi.org/10.24033/asens.2385},
}

@article{Martel2022PMP,
	AUTHOR = {Martel, Yvan},
	TITLE = {Asymptotic stability of solitary waves for the {1D} cubic-quintic {S}chr\"odinger equation with no internal mode},
	JOURNAL = {Probab. Math. Phys.},
	FJOURNAL = {Probability and Mathematical Physics},
	VOLUME = {3},
	YEAR = {2022},
	NUMBER = {4},
	PAGES = {839--867},
	DOI = {10.2140/pmp.2022.3.839},
	URL = {https://doi.org/10.2140/pmp.2022.3.839},
}

@article{Kim2026arXivNoEmbeeded,
	title={Absence of non-threshold embedded eigenvalues for nonlinear {S}chr\"{o}dinger equations}, 
	author={Kim, Taegyu},
	year={2026},
	journal = {preprint}
}

@article{MMMZ2025arXiv,
	title={Dynamics of subcritical threshold solutions for the {$4d$} energy-critical {NLS}},
	journal={preprint arXiv:2508.02608, to appear in Anal. PDE},
	author={Ma, Zuyu and Miao, Changxing and Murphy, Jason and Zheng, Jiqiang},
	year={2025}
}

@article{JXYZ2026arXiv,
	author = {Jendrej, Jacek and Xu, Jitao and Yuan, Xu and Zhao, Lifeng},
	journal = {preprint arXiv:2610.20313},
	title = {Quantized slow blowup dynamics for the energy-critical {S}chr\"{o}dinger map flow},
	year = {2026}
}

@article{SuZhao2026arXivScattering,
	title={Scattering below the ground state for the three-dimensional focusing energy-critical {NLS}},
	author={Su, Qingtang and Zhao, Zehua},
	journal = {preprint arXiv:2609.30228},
	year={2026}
}

@article{ChungHan2026arXivScattering,
	title={Threshold scattering for the three-dimensional focusing energy-critical nonlinear {S}chr{\"o}dinger equation},
	author={Chung, Pang-hung and Han, Dan},
	journal = {preprint arXiv:2610.10658},
	year={2026}
}

@article{JendrejLawrieRodriguez2022ASENS,
	author = {Jendrej, Jacek and Lawrie, Andrew and Rodriguez, Casey},
	doi = {10.24033/asens.2514},
	fjournal = {Annales Scientifiques de l'\'{E}cole Normale Sup\'{e}rieure. Quatri\`eme S\'{e}rie},
	issn = {0012-9593},
	journal = {Ann. Sci. \'{E}c. Norm. Sup\'{e}r. (4)},
	mrclass = {58E20 (35A21 37J51 58E05)},
	mrnumber = {4468859},
	number = {4},
	pages = {1135--1198},
	title = {Dynamics of bubbling wave maps with prescribed radiation},
	volume = {55},
	year = {2022}}

@article{OrtolevaPerelman2013,
	author = {Ortoleva, C. and Perelman, G.},
	doi = {10.1090/S1061-0022-2014-01290-3},
	fjournal = {Rossi\u{\i}skaya Akademiya Nauk. Algebra i Analiz},
	issn = {0234-0852},
	journal = {Algebra i Analiz},
	mrclass = {35Q55 (35B44)},
	mrnumber = {3114854},
	mrreviewer = {Amin\ Esfahani},
	number = {2},
	pages = {162--192},
	title = {Nondispersive vanishing and blow up at infinity for the energy critical nonlinear {S}chr\"{o}dinger equation in {$\mathbb{R}^3$}},
	url = {https://doi.org/10.1090/S1061-0022-2014-01290-3},
	volume = {25},
	year = {2013}}

@article {Schmid2026CVPDE,
    AUTHOR = {Schmid, Tobias},
     TITLE = {Blow up dynamics for the 3{D} energy-critical nonlinear
              {S}chr\"odinger equation},
   JOURNAL = {Calc. Var. Partial Differential Equations},
  FJOURNAL = {Calculus of Variations and Partial Differential Equations},
    VOLUME = {65},
      YEAR = {2026},
    NUMBER = {3},
     PAGES = {Paper No. 69, 114},
      ISSN = {0944-2669,1432-0835},
   MRCLASS = {35B44 (35Q55)},
  MRNUMBER = {5032421},
       DOI = {10.1007/s00526-025-03240-x},
       URL = {https://doi.org/10.1007/s00526-025-03240-x},
}

@article{JeongKim2024arXiv,
      title={Quantized blow-up dynamics for {C}alogero--{M}oser derivative nonlinear {S}chr\"odinger equation}, 
      author={Jeong, Uihyeon and Kim, Taegyu},
    journal = {preprint arXiv:2412.12518},
      year={2024}
}

@article {KimKwonOh2026CMP,
	AUTHOR = {Kim, Kihyun and Kwon, Soonsik and Oh, Sung-Jin},
	TITLE = {Blow-up dynamics for radial self-dual {C}hern--{S}imons--{S}chr{\"o}dinger equation with prescribed asymptotic profile},
	JOURNAL = {Comm. Math. Phys.},
	FJOURNAL = {Communications in Mathematical Physics},
	VOLUME = {407},
	YEAR = {2026},
	NUMBER = {4},
	PAGES = {Paper No. 70},
	ISSN = {0010-3616,1432-0916},
	DOI = {10.1007/s00220-026-05577-6},
	URL = {https://doi.org/10.1007/s00220-026-05577-6},
}

@article{HerreroVelazquez1992,
  title={A blow up result for semilinear heat equations in the supercritical case},
  author={Herrero, Miguel A. and Vel\'{a}zquez, Juan J. L.},
  journal={preprint},
  year={1992}
}

@article {HerreroVelazquez1994CRASPSI,
    AUTHOR = {Herrero, Miguel A. and Vel\'{a}zquez, Juan J. L.},
     TITLE = {Explosion de solutions d'\'{e}quations paraboliques
              semilin\'{e}aires supercritiques},
   JOURNAL = {C. R. Acad. Sci. Paris S\'{e}r. I Math.},
  FJOURNAL = {Comptes Rendus de l'Acad\'{e}mie des Sciences. S\'{e}rie I.
              Math\'{e}matique},
    VOLUME = {319},
      YEAR = {1994},
    NUMBER = {2},
     PAGES = {141--145},
      ISSN = {0764-4442},
   MRCLASS = {35B40 (35K57)},
  MRNUMBER = {1288393},
MRREVIEWER = {Sergey\ A.\ Vakulenko},
}

@article{FilippasHerreroVelazquez2000,
	author = {Filippas, Stathis and Herrero, Miguel A. and Vel\'{a}zquez, Juan J. L.},
	doi = {10.1098/rspa.2000.0648},
	fjournal = {The Royal Society of London. Proceedings. Series A. Mathematical, Physical and Engineering Sciences},
	issn = {1364-5021},
	journal = {R. Soc. Lond. Proc. Ser. A Math. Phys. Eng. Sci.},
	mrclass = {35K55 (35B33 35B40 35C20)},
	mrnumber = {1843848},
	mrreviewer = {Juli\'{a}n Aguirre},
	number = {2004},
	pages = {2957--2982},
	title = {Fast blow-up mechanisms for sign-changing solutions of a semilinear parabolic equation with critical nonlinearity},
	url = {https://doi.org/10.1098/rspa.2000.0648},
	volume = {456},
	year = {2000}}

@article {Mizoguchi2004AdvDE,
    AUTHOR = {Mizoguchi, Noriko},
     TITLE = {Type-{II} blowup for a semilinear heat equation},
   JOURNAL = {Adv. Differential Equations},
  FJOURNAL = {Advances in Differential Equations},
    VOLUME = {9},
      YEAR = {2004},
    NUMBER = {11-12},
     PAGES = {1279--1316},
      ISSN = {1079-9389},
   MRCLASS = {35K55 (35B40 35K15 35K57)},
  MRNUMBER = {2099557},
MRREVIEWER = {Varga\ Kalantarov},
}

@article {Mizoguchi2007Math.Ann.,
    AUTHOR = {Mizoguchi, Noriko},
     TITLE = {Rate of type {II} blowup for a semilinear heat equation},
   JOURNAL = {Math. Ann.},
  FJOURNAL = {Mathematische Annalen},
    VOLUME = {339},
      YEAR = {2007},
    NUMBER = {4},
     PAGES = {839--877},
      ISSN = {0025-5831,1432-1807},
   MRCLASS = {35K55 (35B40 35K15)},
  MRNUMBER = {2341904},
MRREVIEWER = {Juli\'{a}n\ Aguirre},
       DOI = {10.1007/s00208-007-0133-z},
       URL = {https://doi.org/10.1007/s00208-007-0133-z},
}

@article {Mizoguchi2011TranAMS,
    AUTHOR = {Mizoguchi, Noriko},
     TITLE = {Blow-up rate of type {II} and the braid group theory},
   JOURNAL = {Trans. Amer. Math. Soc.},
  FJOURNAL = {Transactions of the American Mathematical Society},
    VOLUME = {363},
      YEAR = {2011},
    NUMBER = {3},
     PAGES = {1419--1443},
      ISSN = {0002-9947,1088-6850},
   MRCLASS = {35K91 (35B07 35B44 35K20 57M07)},
  MRNUMBER = {2737271},
MRREVIEWER = {Christian\ Stinner},
       DOI = {10.1090/S0002-9947-2010-04784-1},
       URL = {https://doi.org/10.1090/S0002-9947-2010-04784-1},
}

@article{HadzicRaphael2019JEMS,
	author = {Had\v{z}i\'{c}, Mahir and Rapha\"{e}l, Pierre},
	doi = {10.4171/JEMS/904},
	fjournal = {Journal of the European Mathematical Society (JEMS)},
	issn = {1435-9855},
	journal = {J. Eur. Math. Soc. (JEMS)},
	mrclass = {35R35 (35A20 35K05)},
	mrnumber = {4012340},
	mrreviewer = {Rodica Luca},
	number = {11},
	pages = {3259--3341},
	title = {On melting and freezing for the 2{D} radial {S}tefan problem},
	url = {https://doi.org/10.4171/JEMS/904},
	volume = {21},
	year = {2019}}

@article {CollotGhoulMasmoudiNguyen2022CPAM,
    AUTHOR = {Collot, Charles and Ghoul, Tej-Eddine and Masmoudi, Nader and
              Nguyen, Van Tien},
     TITLE = {Refined description and stability for singular solutions of
              the 2{D} {K}eller-{S}egel system},
   JOURNAL = {Comm. Pure Appl. Math.},
  FJOURNAL = {Communications on Pure and Applied Mathematics},
    VOLUME = {75},
      YEAR = {2022},
    NUMBER = {7},
     PAGES = {1419--1516},
      ISSN = {0010-3640,1097-0312},
   MRCLASS = {35Q92 (92C17)},
  MRNUMBER = {4438587},
}

@article {GhoulIbrahimNguyen2018JDE,
    AUTHOR = {Ghoul, T. and Ibrahim, S. and Nguyen, V. T.},
     TITLE = {Construction of type {II} blowup solutions for the
              1-corotational energy supercritical wave maps},
   JOURNAL = {J. Differential Equations},
  FJOURNAL = {Journal of Differential Equations},
    VOLUME = {265},
      YEAR = {2018},
    NUMBER = {7},
     PAGES = {2968--3047},
      ISSN = {0022-0396,1090-2732},
   MRCLASS = {35L71 (35B40 35B44 35L51 35R01)},
  MRNUMBER = {3812220},
       DOI = {10.1016/j.jde.2018.04.058},
       URL = {https://doi.org/10.1016/j.jde.2018.04.058},
}

@article {MerleRaphaelRodnianskiSzeftel2022Invention,
    AUTHOR = {Merle, Frank and Rapha\"el, Pierre and Rodnianski, Igor and
              Szeftel, Jeremie},
     TITLE = {On blow up for the energy super critical defocusing nonlinear
              {S}chr\"odinger equations},
   JOURNAL = {Invent. Math.},
  FJOURNAL = {Inventiones Mathematicae},
    VOLUME = {227},
      YEAR = {2022},
    NUMBER = {1},
     PAGES = {247--413},
      ISSN = {0020-9910,1432-1297},
   MRCLASS = {35Q55},
  MRNUMBER = {4359478},
       DOI = {10.1007/s00222-021-01067-9},
       URL = {https://doi.org/10.1007/s00222-021-01067-9},
}

@article {RS2011JAMSinhomo,
	AUTHOR = {Rapha\"el, Pierre and Szeftel, Jeremie},
	TITLE = {Existence and uniqueness of minimal blow-up solutions to an
	inhomogeneous mass critical {NLS}},
	JOURNAL = {J. Amer. Math. Soc.},
	FJOURNAL = {Journal of the American Mathematical Society},
	VOLUME = {24},
	YEAR = {2011},
	NUMBER = {2},
	PAGES = {471--546},
	ISSN = {0894-0347,1088-6834},
	MRCLASS = {35Q55 (35B35 35B44)},
	MRNUMBER = {2748399},
	MRREVIEWER = {Masahito\ Ohta},
	DOI = {10.1090/S0894-0347-2010-00688-1},
	URL = {https://doi.org/10.1090/S0894-0347-2010-00688-1},

}

@article{Jeong2025AnalPDE,
     TITLE = {Quantized slow blow-up dynamics for the energy-critical corotational wave map problem}, 
    AUTHOR = {Uihyeon Jeong},   
   JOURNAL = {Anal. PDE},
  FJOURNAL = {Analysis \& PDE},
    VOLUME = {18},
      YEAR = {2025},
    NUMBER = {10},
     PAGES = {2415--2480},
       DOI = {10.2140/apde.2025.18.2415},
       URL = {https://doi.org/10.2140/apde.2025.18.2415},
}

@article{KimKwon2020blowup,
	author = {Kim, Kihyun and Kwon, Soonsik},
	doi = {10.1007/s40818-023-00147-8},
	fjournal = {Annals of PDE. Journal Dedicated to the Analysis of Problems from Physical Sciences},
	issn = {2524-5317,2199-2576},
	journal = {Ann. PDE},
	mrclass = {35B44 (35Q55)},
	mrnumber = {4564552},
	mrreviewer = {Van\ Duong\ Dinh},
	number = {1},
	pages = {Paper No. 6, 129},
	title = {Construction of blow-up manifolds to the equivariant self-dual {C}hern-{S}imons-{S}chr\"{o}dinger equation},
	url = {https://doi.org/10.1007/s40818-023-00147-8},
	volume = {9},
	year = {2023}}

@article{RodnianskiSterbenz2010,
	author = {Rodnianski, Igor and Sterbenz, Jacob},
	doi = {10.4007/annals.2010.172.187},
	fjournal = {Annals of Mathematics. Second Series},
	issn = {0003-486X,1939-8980},
	journal = {Ann. of Math. (2)},
	mrclass = {58E20 (35B44 35L70 35Q40 81T10)},
	mrnumber = {2680419},
	mrreviewer = {Andreas\ Gastel},
	number = {1},
	pages = {187--242},
	title = {On the formation of singularities in the critical {${\rm O}(3)$} {$\sigma$}-model},
	url = {https://doi.org/10.4007/annals.2010.172.187},
	volume = {172},
	year = {2010}}

@article{RaphaelRodnianski2012,
	author = {Rapha\"{e}l, Pierre and Rodnianski, Igor},
	doi = {10.1007/s10240-011-0037-z},
	fjournal = {Publications Math\'{e}matiques. Institut de Hautes \'{E}tudes Scientifiques},
	issn = {0073-8301,1618-1913},
	journal = {Publ. Math. Inst. Hautes \'{E}tudes Sci.},
	mrclass = {58E20 (35A20 35B44 35L70 58E15 81T13)},
	mrnumber = {2929728},
	mrreviewer = {Andreas\ Gastel},
	pages = {1--122},
	title = {Stable blow up dynamics for the critical co-rotational wave maps and equivariant {Y}ang-{M}ills problems},
	url = {https://doi.org/10.1007/s10240-011-0037-z},
	volume = {115},
	year = {2012}}

@article{Merle1993Duke,
	author = {Merle, F.},
	doi = {10.1215/S0012-7094-93-06919-0},
	fjournal = {Duke Mathematical Journal},
	issn = {0012-7094,1547-7398},
	journal = {Duke Math. J.},
	mrclass = {35Q55 (34G20 35B35)},
	mrnumber = {1203233},
	mrreviewer = {Song\ Mu\ Zheng},
	number = {2},
	pages = {427--454},
	title = {Determination of blow-up solutions with minimal mass for nonlinear {S}chr\"{o}dinger equations with critical power},
	url = {https://doi.org/10.1215/S0012-7094-93-06919-0},
	volume = {69},
	year = {1993}}

@article{MerleRaphaelRodnianski2013Invention,
	author = {Merle, Frank and Rapha\"{e}l, Pierre and Rodnianski, Igor},
	doi = {10.1007/s00222-012-0427-y},
	fjournal = {Inventiones Mathematicae},
	issn = {0020-9910,1432-1297},
	journal = {Invent. Math.},
	mrclass = {35K45 (35B44 35Q55 35Q60 58J99)},
	mrnumber = {3090180},
	mrreviewer = {Dimitrios\ E.\ Tzanetis},
	number = {2},
	pages = {249--365},
	title = {Blowup dynamics for smooth data equivariant solutions to the critical {S}chr\"{o}dinger map problem},
	url = {https://doi.org/10.1007/s00222-012-0427-y},
	volume = {193},
	year = {2013}}

@article{MerleRaphaelRodnianski2015CambJMath,
	author = {Merle, Frank and Rapha\"{e}l, Pierre and Rodnianski, Igor},
	doi = {10.4310/CJM.2015.v3.n4.a1},
	fjournal = {Cambridge Journal of Mathematics},
	issn = {2168-0930,2168-0949},
	journal = {Camb. J. Math.},
	mrclass = {35Q55 (35B44 35C06)},
	mrnumber = {3435273},
	mrreviewer = {Thomas\ Duyckaerts},
	number = {4},
	pages = {439--617},
	title = {Type {II} blow up for the energy supercritical {NLS}},
	url = {https://doi.org/10.4310/CJM.2015.v3.n4.a1},
	volume = {3},
	year = {2015}}

@article{RaphaelSchweyer2013CPAMHeat,
	author = {Rapha\"{e}l, Pierre and Schweyer, Remi},
	doi = {10.1002/cpa.21435},
	fjournal = {Communications on Pure and Applied Mathematics},
	issn = {0010-3640,1097-0312},
	journal = {Comm. Pure Appl. Math.},
	mrclass = {35R01 (35B44 35K91 53Cxx 58J35)},
	mrnumber = {3008229},
	mrreviewer = {Shu-Yu\ Hsu},
	number = {3},
	pages = {414--480},
	title = {Stable blowup dynamics for the 1-corotational energy critical harmonic heat flow},
	url = {https://doi.org/10.1002/cpa.21435},
	volume = {66},
	year = {2013}}

@article{RaphaelSchweyer2014APDE,
	author = {Rapha\"{e}l, Pierre and Schweyer, Remi},
	doi = {10.2140/apde.2014.7.1713},
	fjournal = {Analysis \& PDE},
	issn = {2157-5045,1948-206X},
	journal = {Anal. PDE},
	mrclass = {35K91 (35B44)},
	mrnumber = {3318739},
	mrreviewer = {Christopher\ P.\ Grant},
	number = {8},
	pages = {1713--1805},
	title = {Quantized slow blow-up dynamics for the corotational energy-critical harmonic heat flow},
	url = {https://doi.org/10.2140/apde.2014.7.1713},
	volume = {7},
	year = {2014}}

@article{Collot2017AnalPDE,
	author = {Collot, Charles},
	doi = {10.2140/apde.2017.10.127},
	fjournal = {Analysis \& PDE},
	issn = {2157-5045,1948-206X},
	journal = {Anal. PDE},
	mrclass = {35K91 (35B20 35B44 35C06 35K58)},
	mrnumber = {3611015},
	mrreviewer = {Hongwei\ Chen},
	number = {1},
	pages = {127--252},
	title = {Nonradial type {II} blow up for the energy-supercritical semilinear heat equation},
	url = {https://doi.org/10.2140/apde.2017.10.127},
	volume = {10},
	year = {2017}}

@article{Collot2018MemAmer,
	author = {Collot, Charles},
	doi = {10.1090/memo/1205},
	fjournal = {Memoirs of the American Mathematical Society},
	isbn = {978-1-4704-2813-6; 978-1-4704-4379-5},
	issn = {0065-9266,1947-6221},
	journal = {Mem. Amer. Math. Soc.},
	mrclass = {35L71 (35B44 58B99)},
	mrnumber = {3778126},
	mrreviewer = {Herbert\ Koch},
	number = {1205},
	pages = {v+163},
	title = {Type {II} blow up manifolds for the energy supercritical semilinear wave equation},
	url = {https://doi.org/10.1090/memo/1205},
	volume = {252},
	year = {2018}}

@article{delPinoMussoWeiZhang2020HeatQuantized,
	author = {del Pino, Manuel and Musso, Monica and Wei, Juncheng and Zhang, Qidi and Zhang, Yifu},
	journal = {preprint arXiv:2002.05765},
	title = {Type {II} Finite time blow-up for the three dimensional energy critical heat equation},
	year = {2020}}

@article{RaphaelSchweyer2014MathAnn,
	author = {Rapha\"{e}l, Pierre and Schweyer, R\'{e}mi},
	doi = {10.1007/s00208-013-1002-6},
	fjournal = {Mathematische Annalen},
	issn = {0025-5831,1432-1807},
	journal = {Math. Ann.},
	mrclass = {35M31 (35B20 35B35 35B44 92C17)},
	mrnumber = {3201901},
	mrreviewer = {Adrien\ Blanchet},
	number = {1-2},
	pages = {267--377},
	title = {On the stability of critical chemotactic aggregation},
	url = {https://doi.org/10.1007/s00208-013-1002-6},
	volume = {359},
	year = {2014}}

@article{MerleRaphaelRodnianskiSzeftel2022-1Ann.Math.,
	author = {Merle, Frank and Rapha\"{e}l, Pierre and Rodnianski, Igor and Szeftel, Jeremie},
	doi = {10.4007/annals.2022.196.2.3},
	fjournal = {Annals of Mathematics. Second Series},
	issn = {0003-486X},
	journal = {Ann. of Math. (2)},
	mrclass = {35Q35 (34C37)},
	mrnumber = {4445442},
	number = {2},
	pages = {567--778},
	title = {On the implosion of a compressible fluid {I}: smooth self-similar inviscid profiles},
	url = {https://doi.org/10.4007/annals.2022.196.2.3},
	volume = {196},
	year = {2022}}

@article{MerleRaphaelRodnianskiSzeftel2022-2Ann.Math.,
	author = {Merle, Frank and Rapha\"{e}l, Pierre and Rodnianski, Igor and Szeftel, Jeremie},
	doi = {10.4007/annals.2022.196.2.4},
	fjournal = {Annals of Mathematics. Second Series},
	issn = {0003-486X},
	journal = {Ann. of Math. (2)},
	mrclass = {35Q35 (35B44)},
	mrnumber = {4445443},
	number = {2},
	pages = {779--889},
	title = {On the implosion of a compressible fluid {II}: singularity formation},
	url = {https://doi.org/10.4007/annals.2022.196.2.4},
	volume = {196},
	year = {2022}}

@article {CollotMerleRaphael2020JAMS,
    AUTHOR = {Collot, Charles and Merle, Frank and Rapha\"el, Pierre},
     TITLE = {Strongly anisotropic type {II} blow up at an isolated point},
   JOURNAL = {J. Amer. Math. Soc.},
  FJOURNAL = {Journal of the American Mathematical Society},
    VOLUME = {33},
      YEAR = {2020},
    NUMBER = {2},
     PAGES = {527--607},
      ISSN = {0894-0347,1088-6834},
   MRCLASS = {35K91 (35B44 35K15 35K58)},
  MRNUMBER = {4073868},
MRREVIEWER = {Alain\ Brillard},
       DOI = {10.1090/jams/941},
       URL = {https://doi.org/10.1090/jams/941},
}

@article {Harada2020AIHPC,
    AUTHOR = {Harada, Junichi},
     TITLE = {A higher speed type {II} blowup for the five dimensional
              energy critical heat equation},
   JOURNAL = {Ann. Inst. H. Poincar\'e{} C Anal. Non Lin\'eaire},
  FJOURNAL = {Annales de l'Institut Henri Poincar\'e{} C. Analyse Non
              Lin\'eaire},
    VOLUME = {37},
      YEAR = {2020},
    NUMBER = {2},
     PAGES = {309--341},
      ISSN = {0294-1449,1873-1430},
   MRCLASS = {35K91 (35B44)},
  MRNUMBER = {4072807},
MRREVIEWER = {Daniele\ Andreucci},
       DOI = {10.1016/j.anihpc.2019.09.006},
       URL = {https://doi.org/10.1016/j.anihpc.2019.09.006},
}

@article{Kim2025JEMS,
	author = {Kim, Kihyun},
	journal = {J. Eur. Math. Soc. (JEMS), published online first},
	title = {Rigidity of smooth finite-time blow-up for equivariant self-dual {C}hern-{S}imons-{S}chr\"{o}dinger equation},
  url = {https://doi.org/10.4171/jems/1569},
	year = {2025}}

@article {KimKimKwon2024arxiv,
	author={Kim, Kihyun and Kim, Taegyu and Kwon, Soonsik},
	journal = {preprint arXiv:2404.09603, to appear in Mem. Amer. Math. Soc.},
	title={Construction of smooth chiral finite-time blow-up solutions to {C}alogero--{M}oser derivative nonlinear {S}chr\"{o}dinger equation}, 
	year={2024}
}

@article {Raphael2005MathAnnalen,
	AUTHOR = {Raphael, Pierre},
	TITLE = {Stability of the log-log bound for blow up solutions to the
	critical non linear {S}chr\"odinger equation},
	JOURNAL = {Math. Ann.},
	FJOURNAL = {Mathematische Annalen},
	VOLUME = {331},
	YEAR = {2005},
	NUMBER = {3},
	PAGES = {577--609},
	ISSN = {0025-5831,1432-1807},
	MRCLASS = {35Q55 (35B35 35B40)},
	MRNUMBER = {2122541},
	MRREVIEWER = {Masahito\ Ohta},
	DOI = {10.1007/s00208-004-0596-0},
	URL = {https://doi.org/10.1007/s00208-004-0596-0},
}

@article {MerleRaphael2005CMP,
	AUTHOR = {Merle, Frank and Raphael, Pierre},
	TITLE = {Profiles and quantization of the blow up mass for critical
	nonlinear {S}chr\"odinger equation},
	JOURNAL = {Comm. Math. Phys.},
	FJOURNAL = {Communications in Mathematical Physics},
	VOLUME = {253},
	YEAR = {2005},
	NUMBER = {3},
	PAGES = {675--704},
	ISSN = {0010-3616,1432-0916},
	MRCLASS = {35Q55 (35B40)},
	MRNUMBER = {2116733},
	MRREVIEWER = {Justin\ A.\ Holmer},
	DOI = {10.1007/s00220-004-1198-0},
	URL = {https://doi.org/10.1007/s00220-004-1198-0},
}

@article {KenigMerle2006Invent,
	AUTHOR = {Kenig, Carlos E. and Merle, Frank},
	TITLE = {Global well-posedness, scattering and blow-up for the
	energy-critical, focusing, non-linear {S}chr\"odinger equation
	in the radial case},
	JOURNAL = {Invent. Math.},
	FJOURNAL = {Inventiones Mathematicae},
	VOLUME = {166},
	YEAR = {2006},
	NUMBER = {3},
	PAGES = {645--675},
	ISSN = {0020-9910,1432-1297},
	MRCLASS = {35Q55 (35B30 35B40 35P25)},
	MRNUMBER = {2257393},
	MRREVIEWER = {Justin\ A.\ Holmer},
	DOI = {10.1007/s00222-006-0011-4},
	URL = {https://doi.org/10.1007/s00222-006-0011-4},
}

@article {Jendrej2017AnalPDE,
	AUTHOR = {Jendrej, Jacek},
	TITLE = {Construction of two-bubble solutions for the energy-critical
	{NLS}},
	JOURNAL = {Anal. PDE},
	FJOURNAL = {Analysis \& PDE},
	VOLUME = {10},
	YEAR = {2017},
	NUMBER = {8},
	PAGES = {1923--1959},
	ISSN = {2157-5045,1948-206X},
	MRCLASS = {35Q55 (35B40)},
	MRNUMBER = {3694010},
	MRREVIEWER = {Yuichiro\ Kawahara},
	DOI = {10.2140/apde.2017.10.1923},
	URL = {https://doi.org/10.2140/apde.2017.10.1923},
}

@article {MerleRaphael2005AnnMath,
	AUTHOR = {Merle, Frank and Raphael, Pierre},
	TITLE = {The blow-up dynamic and upper bound on the blow-up rate for
	critical nonlinear {S}chr\"odinger equation},
	JOURNAL = {Ann. of Math. (2)},
	FJOURNAL = {Annals of Mathematics. Second Series},
	VOLUME = {161},
	YEAR = {2005},
	NUMBER = {1},
	PAGES = {157--222},
	ISSN = {0003-486X,1939-8980},
	MRCLASS = {35Q55 (35B40)},
	MRNUMBER = {2150386},
	MRREVIEWER = {John\ Albert},
	DOI = {10.4007/annals.2005.161.157},
	URL = {https://doi.org/10.4007/annals.2005.161.157},
}

@article {MerleRaphael2003GAFA,
	AUTHOR = {Merle, F. and Raphael, P.},
	TITLE = {Sharp upper bound on the blow-up rate for the critical
	nonlinear {S}chr\"odinger equation},
	JOURNAL = {Geom. Funct. Anal.},
	FJOURNAL = {Geometric and Functional Analysis},
	VOLUME = {13},
	YEAR = {2003},
	NUMBER = {3},
	PAGES = {591--642},
	ISSN = {1016-443X,1420-8970},
	MRCLASS = {35Q55 (35B40)},
	MRNUMBER = {1995801},
	DOI = {10.1007/s00039-003-0424-9},
	URL = {https://doi.org/10.1007/s00039-003-0424-9},
}

@article {MerleRaphael2004Invent,
	AUTHOR = {Merle, Frank and Raphael, Pierre},
	TITLE = {On universality of blow-up profile for {$L^2$} critical
	nonlinear {S}chr\"odinger equation},
	JOURNAL = {Invent. Math.},
	FJOURNAL = {Inventiones Mathematicae},
	VOLUME = {156},
	YEAR = {2004},
	NUMBER = {3},
	PAGES = {565--672},
	ISSN = {0020-9910,1432-1297},
	MRCLASS = {35Q55 (35B33 35B40 35B65 35Q51)},
	MRNUMBER = {2061329},
	MRREVIEWER = {Masahito\ Ohta},
	DOI = {10.1007/s00222-003-0346-z},
	URL = {https://doi.org/10.1007/s00222-003-0346-z},
}

@article {MerleRaphael2006JAMS,
	AUTHOR = {Merle, Frank and Raphael, Pierre},
	TITLE = {On a sharp lower bound on the blow-up rate for the {$L^2$}
	critical nonlinear {S}chr\"odinger equation},
	JOURNAL = {J. Amer. Math. Soc.},
	FJOURNAL = {Journal of the American Mathematical Society},
	VOLUME = {19},
	YEAR = {2006},
	NUMBER = {1},
	PAGES = {37--90},
	ISSN = {0894-0347,1088-6834},
	MRCLASS = {35Q55 (35B40)},
	MRNUMBER = {2169042},
	MRREVIEWER = {Masahito\ Ohta},
	DOI = {10.1090/S0894-0347-05-00499-6},
	URL = {https://doi.org/10.1090/S0894-0347-05-00499-6},
}

@article{JendrejKrieger2025arXiv,
	author = {Jendrej, Jacek and Krieger, Joachim},
	journal = {preprint arXiv:2501.08396},
	title = {Concentric bubbles concentrating in finite time for the energy critical wave maps equation},
	year = {2025}
}

@article{JeongKimKimKwon2026arXiv,
	author = {Jeong, Uihyeon and Kim, Kihyun and Kim, Taegyu and Kwon, Soonsik},
	journal = {preprint arXiv:2601.07410},
	title = {Classification of single-bubble blow-up solutions for {C}alogero--{M}oser derivative nonlinear {S}chr\"odinger equation},
	year = {2026}
}

@article{JendrejLiXu2026arXivNLS,
	author = {Jendrej, Jacek and Li, Xuemei and Xu, Guixiang},
	journal = {preprint arXiv:2608.16186},
	title = {Construction of two-bubble solutions for the energy-critical {NLS} in dimension 6},
	year = {2026}
}

@article {Schweyer2012JFA,
    AUTHOR = {Schweyer, R\'emi},
     TITLE = {Type {II} blow-up for the four dimensional energy critical
              semi linear heat equation},
   JOURNAL = {J. Funct. Anal.},
  FJOURNAL = {Journal of Functional Analysis},
    VOLUME = {263},
      YEAR = {2012},
    NUMBER = {12},
     PAGES = {3922--3983},
      ISSN = {0022-1236,1096-0783},
   MRCLASS = {35K91 (35A20 35B44 35B45 35K15)},
  MRNUMBER = {2990063},
MRREVIEWER = {Juli\'an\ Aguirre},
       DOI = {10.1016/j.jfa.2012.09.015},
       URL = {https://doi.org/10.1016/j.jfa.2012.09.015},
}

@article {PMW2019,
    AUTHOR = {del Pino, Manuel and Musso, Monica and Wei, Jun Cheng},
     TITLE = {Type {II} blow-up in the 5-dimensional energy critical heat
              equation},
   JOURNAL = {Acta Math. Sin. (Engl. Ser.)},
  FJOURNAL = {Acta Mathematica Sinica (English Series)},
    VOLUME = {35},
      YEAR = {2019},
    NUMBER = {6},
     PAGES = {1027--1042},
      ISSN = {1439-8516,1439-7617},
   MRCLASS = {35K91 (35B40 35B44 35K15)},
  MRNUMBER = {3952701},
MRREVIEWER = {Hongwei\ Chen},
       DOI = {10.1007/s10114-019-8341-5},
       URL = {https://doi.org/10.1007/s10114-019-8341-5},
}

@article{DuyckaertsMerle2009Threshold,
	author = {Duyckaerts, Thomas and Merle, Frank},
	title = {Dynamic of threshold solutions for energy-critical {NLS}},
	journal = {Geom. Funct. Anal.},
	volume = {18},
	number = {6},
	pages = {1787--1840},
	year = {2009},
	doi = {10.1007/s00039-009-0707-x},
	url = {https://doi.org/10.1007/s00039-009-0707-x}
}

@book{Teschl2014Quantum,
	author = {Teschl, Gerald},
	title = {Mathematical Methods in Quantum Mechanics: With Applications to {Schr{\"o}dinger} Operators},
	edition = {Second},
	series = {Graduate Studies in Mathematics},
	volume = {157},
	publisher = {American Mathematical Society},
	address = {Providence, RI},
	year = {2014},
	url = {https://www.mat.univie.ac.at/~gerald/ftp/book-schroe/schroe2.pdf}
}

@article {GermainPusateri2022Pi,
    AUTHOR = {Germain, Pierre and Pusateri, Fabio},
     TITLE = {Quadratic {K}lein-{G}ordon equations with a potential in one
              dimension},
   JOURNAL = {Forum Math. Pi},
  FJOURNAL = {Forum of Mathematics. Pi},
    VOLUME = {10},
      YEAR = {2022},
     PAGES = {Paper No. e17, 172},
      ISSN = {2050-5086},
   MRCLASS = {43A32 (35P25 42B37)},
  MRNUMBER = {4451296},
MRREVIEWER = {Valeri\ S.\ Serov},
       DOI = {10.1017/fmp.2022.9},
       URL = {https://doi.org/10.1017/fmp.2022.9},
}

@article {KairzhanPusateri2023PAA,
    AUTHOR = {Kairzhan, Adilbek and Pusateri, Fabio},
     TITLE = {Asymptotic stability near the soliton for quartic
              {K}lein-{G}ordon equation in 1{D}},
   JOURNAL = {Pure Appl. Anal.},
  FJOURNAL = {Pure and Applied Analysis},
    VOLUME = {5},
      YEAR = {2023},
    NUMBER = {4},
     PAGES = {795--832},
      ISSN = {2578-5893,2578-5885},
   MRCLASS = {35P25 (35C08 35Q55 42B37 43A32)},
  MRNUMBER = {4680525},
MRREVIEWER = {Luis\ Vazquez},
       DOI = {10.2140/paa.2023.5.795},
       URL = {https://doi.org/10.2140/paa.2023.5.795},
}

@article {CazenaveWeissler1990,
    AUTHOR = {Cazenave, Thierry and Weissler, Fred B.},
     TITLE = {The {C}auchy problem for the critical nonlinear
              {S}chr\"odinger equation in {$H^s$}},
   JOURNAL = {Nonlinear Anal.},
  FJOURNAL = {Nonlinear Analysis. Theory, Methods \& Applications. An
              International Multidisciplinary Journal},
    VOLUME = {14},
      YEAR = {1990},
    NUMBER = {10},
     PAGES = {807--836},
      ISSN = {0362-546X,1873-5215},
   MRCLASS = {35Q55 (45K05)},
  MRNUMBER = {1055532},
MRREVIEWER = {Victor\ V.\ Zharinov},
       DOI = {10.1016/0362-546X(90)90023-A},
       URL = {https://doi.org/10.1016/0362-546X(90)90023-A},
}

@article {BuslaevPerelman1992,
    AUTHOR = {Buslaev, V. S. and Perelman, G. S.},
     TITLE = {Scattering for the nonlinear {S}chr\"odinger equation: states
              that are close to a soliton},
   JOURNAL = {Algebra i Analiz},
  FJOURNAL = {Rossi\u iskaya Akademiya Nauk. Algebra i Analiz},
    VOLUME = {4},
      YEAR = {1992},
    NUMBER = {6},
     PAGES = {63--102},
      ISSN = {0234-0852},
   MRCLASS = {35Q55 (35Q51 47N20)},
  MRNUMBER = {1199635},
MRREVIEWER = {Meng\ Ru\ Li},
}

@incollection {BuslaevPerelman1995,
    AUTHOR = {Buslaev, V. S. and Perelman, G. S.},
     TITLE = {On the stability of solitary waves for nonlinear
              {S}chr\"odinger equations},
 BOOKTITLE = {Nonlinear evolution equations},
    SERIES = {Amer. Math. Soc. Transl. Ser. 2},
    VOLUME = {164},
     PAGES = {75--98},
 PUBLISHER = {Amer. Math. Soc., Providence, RI},
      YEAR = {1995},
      ISBN = {0-8218-4123-8},
   MRCLASS = {35Q55 (35B40)},
  MRNUMBER = {1334139},
MRREVIEWER = {Wen-Xiu\ Ma},
       DOI = {10.1090/trans2/164/04},
       URL = {https://doi.org/10.1090/trans2/164/04},
}

@article {Perelman2001,
    AUTHOR = {Perelman, Galina},
     TITLE = {On the formation of singularities in solutions of the critical nonlinear {S}chr\"odinger equation},
   JOURNAL = {Ann. Henri Poincar\'e},
  FJOURNAL = {Annales Henri Poincar\'e. A Journal of Theoretical and
              Mathematical Physics},
    VOLUME = {2},
      YEAR = {2001},
    NUMBER = {4},
     PAGES = {605--673},
      ISSN = {1424-0637,1424-0661},
   MRCLASS = {35Q55 (35A20)},
  MRNUMBER = {1852922},
MRREVIEWER = {Markus\ Kunze},
       DOI = {10.1007/PL00001048},
       URL = {https://doi.org/10.1007/PL00001048},
}

@article {Schlag2009annals,
    AUTHOR = {Schlag, Wilhelm},
     TITLE = {Stable manifolds for an orbitally unstable nonlinear
              {S}chr\"odinger equation},
   JOURNAL = {Ann. of Math. (2)},
  FJOURNAL = {Annals of Mathematics. Second Series},
    VOLUME = {169},
      YEAR = {2009},
    NUMBER = {1},
     PAGES = {139--227},
      ISSN = {0003-486X,1939-8980},
   MRCLASS = {35Q55},
  MRNUMBER = {2480603},
MRREVIEWER = {Olivier\ J.\ Goubet},
       DOI = {10.4007/annals.2009.169.139},
       URL = {https://doi.org/10.4007/annals.2009.169.139},
}

@article {KriegerSchlag2006JAMS,
    AUTHOR = {Krieger, J. and Schlag, W.},
     TITLE = {Stable manifolds for all monic supercritical focusing
              nonlinear {S}chr\"odinger equations in one dimension},
   JOURNAL = {J. Amer. Math. Soc.},
  FJOURNAL = {Journal of the American Mathematical Society},
    VOLUME = {19},
      YEAR = {2006},
    NUMBER = {4},
     PAGES = {815--920},
      ISSN = {0894-0347,1088-6834},
   MRCLASS = {35Q55 (35Q51 37K10 37K40 37K45)},
  MRNUMBER = {2219305},
MRREVIEWER = {Arthur\ H.\ Vartanian},
       DOI = {10.1090/S0894-0347-06-00524-8},
       URL = {https://doi.org/10.1090/S0894-0347-06-00524-8},
}

@article {ErdoganSchlag2,
    AUTHOR = {Erdo{\u{g}}an, M. Burak and Schlag, Wilhelm},
     TITLE = {Dispersive estimates for {S}chr\"odinger operators in the
              presence of a resonance and/or an eigenvalue at zero energy in
              dimension three. {II}},
   JOURNAL = {J. Anal. Math.},
  FJOURNAL = {Journal d'Analyse Math\'ematique},
    VOLUME = {99},
      YEAR = {2006},
     PAGES = {199--248},
      ISSN = {0021-7670,1565-8538},
   MRCLASS = {35J10 (35Q40 47F05 47N50)},
  MRNUMBER = {2279551},
MRREVIEWER = {Gabriele\ Grillo},
       DOI = {10.1007/BF02789446},
       URL = {https://doi.org/10.1007/BF02789446},
}

@article {KillipVisan2010AJM,
    AUTHOR = {Killip, Rowan and Visan, Monica},
     TITLE = {The focusing energy-critical nonlinear {S}chr\"odinger
              equation in dimensions five and higher},
   JOURNAL = {Amer. J. Math.},
  FJOURNAL = {American Journal of Mathematics},
    VOLUME = {132},
      YEAR = {2010},
    NUMBER = {2},
     PAGES = {361--424},
      ISSN = {0002-9327,1080-6377},
   MRCLASS = {35Q55 (35P25)},
  MRNUMBER = {2654778},
MRREVIEWER = {Benedetta\ Pellacci},
       DOI = {10.1353/ajm.0.0107},
       URL = {https://doi.org/10.1353/ajm.0.0107},
}

@article {LiZhang2009JFA,
    AUTHOR = {Li, Dong and Zhang, Xiaoyi},
     TITLE = {Dynamics for the energy critical nonlinear {S}chr\"odinger
              equation in high dimensions},
   JOURNAL = {J. Funct. Anal.},
  FJOURNAL = {Journal of Functional Analysis},
    VOLUME = {256},
      YEAR = {2009},
    NUMBER = {6},
     PAGES = {1928--1961},
      ISSN = {0022-1236,1096-0783},
   MRCLASS = {35Q55 (35B44)},
  MRNUMBER = {2498565},
MRREVIEWER = {Arthur\ H.\ Vartanian},
       DOI = {10.1016/j.jfa.2008.12.007},
       URL = {https://doi.org/10.1016/j.jfa.2008.12.007},
}

@article {RaphaelSzeftel2009CMP,
    AUTHOR = {Rapha\"el, Pierre and Szeftel, J\'er\'emie},
     TITLE = {Standing ring blow up solutions to the {$N$}-dimensional
              quintic nonlinear {S}chr\"odinger equation},
   JOURNAL = {Comm. Math. Phys.},
  FJOURNAL = {Communications in Mathematical Physics},
    VOLUME = {290},
      YEAR = {2009},
    NUMBER = {3},
     PAGES = {973--996},
      ISSN = {0010-3616,1432-0916},
   MRCLASS = {35Q55 (35B44)},
  MRNUMBER = {2525647},
       DOI = {10.1007/s00220-009-0796-2},
       URL = {https://doi.org/10.1007/s00220-009-0796-2},
}

@article {HolmerRoudenko2012APDE,
    AUTHOR = {Holmer, Justin and Roudenko, Svetlana},
     TITLE = {Blow-up solutions on a sphere for the 3{D} quintic {NLS} in
              the energy space},
   JOURNAL = {Anal. PDE},
  FJOURNAL = {Analysis \& PDE},
    VOLUME = {5},
      YEAR = {2012},
    NUMBER = {3},
     PAGES = {475--512},
      ISSN = {2157-5045,1948-206X},
   MRCLASS = {35Q55 (35B44)},
  MRNUMBER = {2994505},
MRREVIEWER = {Nai-Heng\ Chang},
       DOI = {10.2140/apde.2012.5.475},
       URL = {https://doi.org/10.2140/apde.2012.5.475},
}

@article {Zwiers2011APDE,
	AUTHOR = {Zwiers, Ian},
	TITLE = {Standing ring blowup solutions for cubic nonlinear {S}chr\"odinger equations},
	JOURNAL = {Anal. PDE},
	FJOURNAL = {Analysis \& PDE},
	VOLUME = {4},
	YEAR = {2011},
	NUMBER = {5},
	PAGES = {677--727},
	ISSN = {2157-5045,1948-206X},
	MRNUMBER = {2901563},
	DOI = {10.2140/apde.2011.4.677},
	URL = {https://doi.org/10.2140/apde.2011.4.677},
}

@article {DuyckaertsRoudenko2015CMP,
    AUTHOR = {Duyckaerts, Thomas and Roudenko, Svetlana},
     TITLE = {Going beyond the threshold: scattering and blow-up in the
              focusing {NLS} equation},
   JOURNAL = {Comm. Math. Phys.},
  FJOURNAL = {Communications in Mathematical Physics},
    VOLUME = {334},
      YEAR = {2015},
    NUMBER = {3},
     PAGES = {1573--1615},
      ISSN = {0010-3616,1432-0916},
   MRCLASS = {35Q55 (35B44 35P25)},
  MRNUMBER = {3312444},
MRREVIEWER = {Tohru\ Ozawa},
       DOI = {10.1007/s00220-014-2202-y},
       URL = {https://doi.org/10.1007/s00220-014-2202-y},
}

@article {NakanishiRoy2016CPAA,
    AUTHOR = {Nakanishi, Kenji and Roy, Tristan},
     TITLE = {Global dynamics above the ground state for the energy-critical
              {S}chr\"odinger equation with radial data},
   JOURNAL = {Commun. Pure Appl. Anal.},
  FJOURNAL = {Communications on Pure and Applied Analysis},
    VOLUME = {15},
      YEAR = {2016},
    NUMBER = {6},
     PAGES = {2023--2058},
      ISSN = {1534-0392,1553-5258},
   MRCLASS = {35Q55 (35B07 35B33)},
  MRNUMBER = {3565931},
       DOI = {10.3934/cpaa.2016026},
       URL = {https://doi.org/10.3934/cpaa.2016026},
}

@article {SuZhao2023,
    AUTHOR = {Su, Qingtang and Zhao, Zehua},
     TITLE = {Dynamics of subcritical threshold solutions for
              energy-critical {NLS}},
   JOURNAL = {Dyn. Partial Differ. Equ.},
  FJOURNAL = {Dynamics of Partial Differential Equations},
    VOLUME = {20},
      YEAR = {2023},
    NUMBER = {1},
     PAGES = {37--72},
      ISSN = {1548-159X,2163-7873},
   MRCLASS = {35Q55 (35R01 37K06 37L50)},
  MRNUMBER = {4530060},
       DOI = {10.4310/dpde.2023.v20.n1.a3},
       URL = {https://doi.org/10.4310/dpde.2023.v20.n1.a3},
}

@article {Cuccagna2001CPAM,
    AUTHOR = {Cuccagna, Scipio},
     TITLE = {Stabilization of solutions to nonlinear {S}chr\"odinger
              equations},
   JOURNAL = {Comm. Pure Appl. Math.},
  FJOURNAL = {Communications on Pure and Applied Mathematics},
    VOLUME = {54},
      YEAR = {2001},
    NUMBER = {9},
     PAGES = {1110--1145},
      ISSN = {0010-3640,1097-0312},
   MRCLASS = {35Q55 (35B35 93D15)},
  MRNUMBER = {1835384},
MRREVIEWER = {Markus\ Kunze},
       DOI = {10.1002/cpa.1018},
       URL = {https://doi.org/10.1002/cpa.1018},
}

@article {MartelNaumkin2023Tunis,
    AUTHOR = {Martel, Yvan and Naumkin, Ivan},
     TITLE = {Nonflat conformal blow-up profiles for the 1-dimensional critical nonlinear {S}chr\"odinger equation},
   JOURNAL = {Tunis. J. Math.},
  FJOURNAL = {Tunisian Journal of Mathematics},
    VOLUME = {5},
      YEAR = {2023},
    NUMBER = {3},
     PAGES = {505--572},
      ISSN = {2576-7658,2576-7666},
   MRCLASS = {35B44 (35Q41 37K40)},
  MRNUMBER = {4662321},
       DOI = {10.2140/tunis.2023.5.505},
       URL = {https://doi.org/10.2140/tunis.2023.5.505},
}

@article {Kato1965,
    AUTHOR = {Kato, Tosio},
     TITLE = {Wave operators and similarity for some non-selfadjoint operators},
   JOURNAL = {Math. Ann.},
  FJOURNAL = {Mathematische Annalen},
    VOLUME = {162},
      YEAR = {1965/66},
     PAGES = {258--279},
      ISSN = {0025-5831,1432-1807},
   MRCLASS = {47.48},
  MRNUMBER = {190801},
MRREVIEWER = {J.\ T.\ Schwartz},
       DOI = {10.1007/BF01360915},
       URL = {https://doi.org/10.1007/BF01360915},
}

@article {Martel2024invent,
    AUTHOR = {Martel, Yvan},
     TITLE = {Asymptotic stability of small standing solitary waves of the one-dimensional cubic-quintic {S}chr\"odinger equation},
   JOURNAL = {Invent. Math.},
  FJOURNAL = {Inventiones Mathematicae},
    VOLUME = {237},
      YEAR = {2024},
    NUMBER = {3},
     PAGES = {1253--1328},
      ISSN = {0020-9910,1432-1297},
   MRCLASS = {35C08 (35B35 35Q55)},
  MRNUMBER = {4777089},
       DOI = {10.1007/s00222-024-01270-4},
       URL = {https://doi.org/10.1007/s00222-024-01270-4},
}

@book {LinaresPonce,
    AUTHOR = {Linares, Felipe and Ponce, Gustavo},
     TITLE = {Introduction to nonlinear dispersive equations},
    SERIES = {Universitext},
   EDITION = {Second},
 PUBLISHER = {Springer, New York},
      YEAR = {2015},
     PAGES = {xiv+301},
      ISBN = {978-1-4939-2180-5; 978-1-4939-2181-2},
   MRCLASS = {35-01 (35B30 35Q53 35Q55 47F05 47G30)},
  MRNUMBER = {3308874},
       DOI = {10.1007/978-1-4939-2181-2},
       URL = {https://doi.org/10.1007/978-1-4939-2181-2},
}

@article {BourgainWang1997,
	AUTHOR = {Bourgain, Jean and Wang, W.},
	TITLE = {Construction of blowup solutions for the nonlinear
	{S}chr\"odinger equation with critical nonlinearity},
	NOTE = {Dedicated to Ennio De Giorgi},
	JOURNAL = {Ann. Scuola Norm. Sup. Pisa Cl. Sci. (4)},
	FJOURNAL = {Annali della Scuola Normale Superiore di Pisa. Classe di
	Scienze. Serie IV},
	VOLUME = {25},
	YEAR = {1997},
	NUMBER = {1-2},
	PAGES = {197--215},
	ISSN = {0391-173X,2036-2145},
	MRCLASS = {35Q55 (35B99 35C99)},
	MRNUMBER = {1655515},
	MRREVIEWER = {Marcel\ Oliver},
	URL = {http://www.numdam.org/item?id=ASNSP_1997_4_25_1-2_197_0},
}

@article {MRS2013AJM,
	AUTHOR = {Merle, Frank and Rapha\"el, Pierre and Szeftel, Jeremie},
	TITLE = {The instability of {B}ourgain-{W}ang solutions for the {$L^2$}
	critical {NLS}},
	JOURNAL = {Amer. J. Math.},
	FJOURNAL = {American Journal of Mathematics},
	VOLUME = {135},
	YEAR = {2013},
	NUMBER = {4},
	PAGES = {967--1017},
	ISSN = {0002-9327,1080-6377},
	MRCLASS = {35Q55 (35B44)},
	MRNUMBER = {3086066},
	MRREVIEWER = {Corentin\ Audiard},
	DOI = {10.1353/ajm.2013.0033},
	URL = {https://doi.org/10.1353/ajm.2013.0033},
}

@misc{KwakKwon2026arXiv,
      title={A universal bound on the blow-up rate for the focusing mass-critical nonlinear {S}chr\"odinger equation}, 
      author={Kwak, Beomjong and Kwon, Soonsik},
      howpublished = {arXiv:2602.21968},
      year={2026},
}

@article{Kim2026arXivNobubbletree,
    title={No bubble trees for the $1$-equivariant harmonic map heat flow and the radial energy-critical nonlinear heat equation in low dimensions}, 
    author={Kim, Taegyu},
    journal = {preprint arXiv:2608.22447},
    year={2026}
}

\end{document}